\documentclass[11pt]{amsart}
\usepackage[colorlinks = true,
linkcolor = red,
urlcolor  = blue,
citecolor = blue,
anchorcolor = blue]{hyperref}

\makeatletter
\def\@tocline#1#2#3#4#5#6#7{\relax
	\ifnum #1>\c@tocdepth 
	\else
	\par \addpenalty\@secpenalty\addvspace{#2}%
	\begingroup \hyphenpenalty\@M
	\@ifempty{#4}{%
		\@tempdima\csname r@tocindent\number#1\endcsname\relax
	}{%
		\@tempdima#4\relax
	}%
	\parindent\z@ \leftskip#3\relax \advance\leftskip\@tempdima\relax
	\rightskip\@pnumwidth plus4em \parfillskip-\@pnumwidth
	#5\leavevmode\hskip-\@tempdima
	\ifcase #1
	\or\or \hskip 1em \or \hskip 2em \else \hskip 3em \fi%
	#6\nobreak\relax
	\dotfill\hbox to\@pnumwidth{\@tocpagenum{#7}}\par
	\nobreak
	\endgroup
	\fi}
\makeatother

\usepackage{tikz,amsthm,amsmath,amstext,amssymb,amscd,epsfig,euscript, mathrsfs, dsfont,pspicture,multicol, mathtools,graphpap,graphics,graphicx,times,enumerate,subfig,sidecap,wrapfig,color}

\usepackage{ hyperref}
\usepackage{enumerate}
\usepackage{esint}
\usepackage{verbatim} 

\numberwithin{equation}{section}

\newcommand{\R}{{\mathbb R}}       

\newcommand{\Z}{{\mathbb Z}}       
\newcommand{\DD}{{\mathcal D}}

\newcommand{\HH}{{\mathcal H}}
\newcommand{\LL}{{\mathcal L}}

\newcommand{\WW}{{\mathcal W}}

\newcommand{\cN}{{\mathcal N}}
\newcommand{\cM}{{\mathcal M}}

\newcommand{\EE}{{\mathcal E}}

\newcommand{\CC}{{\mathcal C}}

\newcommand{\diam}{{\rm diam}}
\newcommand{\dist}{{\rm dist}}

\newcommand{\rf}[1]{{(\ref{#1})}}

\newcommand{\supp}{\operatorname{supp}}

\newcommand{\vphi}{{\varphi}}
\newcommand{\ve}{{\varepsilon}}
\newcommand{\vv}{{\vspace{2mm}}}
\newcommand{\vvv}{{\vspace{3mm}}}
\newcommand{\wt}[1]{{\widetilde{#1}}}
\newcommand{\wh}[1]{{\widehat{#1}}}

\newcommand{\noi}{\noindent}

\newcommand{\pom}{{\partial\Omega}}
\newcommand{\wpom}{{\partial\wt\Omega}}

\def\Xint#1{\mathchoice
	{\XXint\displaystyle\textstyle{#1}}%
	{\XXint\textstyle\scriptstyle{#1}}%
	{\XXint\scriptstyle\scriptscriptstyle{#1}}%
	{\XXint\scriptscriptstyle\scriptscriptstyle{#1}}%
	\!\int}
\def\XXint#1#2#3{{\setbox0=\hbox{$#1{#2#3}{\int}$ }
		\vcenter{\hbox{$#2#3$ }}\kern-.58\wd0}}

\def\avint{\;\Xint-}

\def\Lip{\mathop\mathrm{Lip}} 	                       
\def\H11{\textup{H}^1_1} 					

\newtheorem*{prob}{\bf Problem}
\newcommand{\dv}{\mathop{\rm div}}
\def\om{\Omega}

\usepackage{pgf,tikz} 

\definecolor{ffffff}{rgb}{1.0,1.0,1.0}
\definecolor{qqqqff}{rgb}{0.0,0.0,1.0}
\definecolor{ffqqqq}{rgb}{1.0,0.0,0.0}
\definecolor{zzzzqq}{rgb}{0.6,0.6,0.0}
\definecolor{marronet}{rgb}{0.6,0.2,0}
\definecolor{negre}{rgb}{0,0,0}
\definecolor{vermell}{rgb}{0.8,0.05,0.05}
\definecolor{blau}{rgb}{0.3,0.2,1.}
\definecolor{blauclar}{rgb}{0.,0.,1.}
\definecolor{grisfosc}{rgb}{0.25098039215686274,0.25098039215686274,0.25098039215686274}
\definecolor{verd}{rgb}{0.1,0.6,0.1}
\definecolor{taronja}{rgb}{0.9,0.6,0.05}
\definecolor{vermellclar}{rgb}{1.,0.,0.}
\definecolor{verdet}{rgb}{0,0.8,0.1}
\definecolor{blauverd}{rgb}{0,0.4,0.2}
\definecolor{grisclar}{rgb}{0.6274509803921569,0.6274509803921569,0.6274509803921569}

\newtheorem{theorem}{Theorem}[section]
\newtheorem{lemma}[theorem]{Lemma}

\newtheorem{coro}[theorem]{Corollary}
\newtheorem{propo}[theorem]{Proposition}

\newtheorem{claim}[theorem]{Claim}
\newtheorem*{claim*}{Claim}

\newtheorem*{theorem*}{Theorem}

\theoremstyle{definition}

\theoremstyle{remark}
\newtheorem{rem}[theorem]{\bf Remark}

\numberwithin{equation}{section}

\newcommand{\brem}{\begin{rem}}
	\newcommand{\erem}{\end{rem}}

\begin{document}

	\title[The Neumann problem in chord-arc domains]{Solvability of the Neumann problem for elliptic equations in chord-arc domains with very big pieces of good superdomains}

	\author[Mihalis Mourgoglou]{Mihalis Mourgoglou}
	\address{Departamento de Matem\'aticas, Universidad del Pa\' is Vasco, UPV/EHU, Barrio Sarriena s/n 48940 Leioa, Spain and\\
		IKERBASQUE, Basque Foundation for Science, Bilbao, Spain.}
	\email{michail.mourgoglou@ehu.eus}
	
	\author[Xavier Tolsa ]{Xavier Tolsa}
	\address{ICREA, Barcelona\\
		Dept.  de Matemàtiques, Universitat Autònoma de Barcelona \\
		and Centre de Recerca Matemàtica, Barcelona, Catalonia.}
	\email{xavier.tolsa@uab.cat}
	\thanks{M.M. was supported  by IKERBASQUE and partially supported by the grants PID2020-118986GB-I00 and PID2024-157724NB-I00 of the Ministerio de Econom\'ia y Competitividad (Spain), and by IT-1615-22 (Basque Government). X.T. was supported by the European Research Council (ERC) under the European Union's Horizon 2020 research and innovation programme (grant agreement 101018680). Also partially supported by MICINN (Spain) under the grant PID2020-114167GB-I00, the María de Maeztu Program for units of excellence (CEX2020-001084-M), and 2021-SGR-00071 (Catalonia).
	}
	
	\begin{abstract}
		Let $\Omega \subset \mathbb{R}^{n+1}$ be a bounded chord-arc domain,  let $\mathcal L=-{\rm div} A\nabla$ be an elliptic operator in $\Omega$
		associated with a matrix $A$ having Dini mean oscillation coefficients, and let $1<p\leq 2$. 
		In this paper we show that if the regularity problem for $\mathcal L$  is solvable in $L^q$ for some $q>p$ in $\Omega$,   $\pom$ supports a weak $p$-Poincar\'e inequality,  and $\Omega$ has very big pieces of superdomains for which the Neumann problem for $\mathcal L$ is solvable uniformly in $L^q$, then the Neumann problem for $\mathcal L$ is solvable in $L^p$ in $\Omega$. 
	\end{abstract}

	\maketitle
	
	
	\section{Introduction}

	Let  $A=(a_{ij})_{1 \leq i,j \leq n+1}$ be a matrix with real measurable coefficients in $\R^{n+1}$.  We say that $A$ is {\it uniformly elliptic} in $\R^{n+1}$ with constant $\Lambda \geq 1$ if it satisfies the following conditions:
	\begin{align}\label{eqelliptic1}
		&\Lambda^{-1}|\xi|^2\leq \langle A(x) \xi,\xi\rangle,\quad \mbox{ for all $\xi \in\R^{n+1}$ and a.e. $x\in\R^{n+1}$.}\\
		& \| a_{ij} \|_{L^\infty(\R^{n+1})}  \leq\Lambda, \quad \mbox{ for all $i,j \in \{1,  2, \dots, n+1\}$.} \label{eqelliptic2}
	\end{align}
	Notice that the matrix $A$ is {\it possibly non-symmetric}. 
	
	For a ball $B\subset\R^{n+1}$, we denote $m_B(A) = \avint_B A(y)\,dy$ and we consider the mean oscillation function $\omega_A:(0,\infty)\to
	(0,\infty)$ defined by
	$$\omega_A(r) = \sup_{x\in\R^{n+1}} \avint_{B(x,r)} |A(y) - m_{B(x,r)}(A)|\,dy.$$
	We say that $A$ has Dini mean oscillation if 
	$$\int_0^1 \omega_A(r)\,\frac{dr}r<\infty.$$
	If 
	$$\mathcal{L}= -{\rm div} (A(\cdot)\nabla)$$
	is an elliptic operator of divergence form associated with a uniformly elliptic matrix $A$,
	{we write $\mathcal{L} \in \mathcal{E}( \R^{n+1})$. If, moreover $A$ has
		Dini mean oscillation in $\R^{n+1}$, we write $\mathcal{L} \in \mathcal{E}_{\rm DMO}( \R^{n+1})$.  We denote  its formal adjoint operator by $\LL^*=-{\rm div} (A^T(\cdot)\nabla)$, where $A^T$ is the transpose matrix of $A$. If the operator $\LL$ is only assumed to be defined in an open set $\Omega\subset\R^{n+1}$, we write either $\mathcal{L} \in \mathcal{E}( \Omega)$ or $\mathcal{L} \in \mathcal{E}_{\rm DMO}( \Omega)$.}
	
	Let $\Omega \subset \R^{n+1}$ be a bounded chord-arc domain.  For $\xi\in\pom$, let $\nu_A(\xi) = A^T(\xi)\,\nu(\xi)$, where $\nu(\xi)$ is the outer unit normal to $\pom$.
	Given $g\in L^{\frac{2n}{n+1}}(\pom)$ with $\int_\pom g\,d\sigma=0$, we consider the Neumann problem
	\begin{equation}\label{eqneumann1}
		\left\{\begin{array}{ll}
			\LL u=0& \quad\mbox{in $\Omega$,}\\
			\partial_{\nu_A} u = g& \quad\mbox{in $\pom$},\\
			u \in W^{1,2}(\Omega).
		\end{array}
		\right.
	\end{equation}
	The identity $\partial_{\nu_A} u = g$ in $\pom$ should be understood in the following  weak sense:
	\begin{equation}\label{eqweak93}
		\int_\Omega A\nabla u\,\nabla \vphi\,dx = \int_\pom g\,\vphi\,d\sigma \quad \mbox{ for all $\vphi\in C_c^\infty(\R^{n+1})$,}
	\end{equation}
	where $\sigma:=\HH^n|_\pom$ is the ``surface measure'' on $\pom$.

	For a function $v:\Omega\to\R$, we define the {\it non-tangential maximal function} of $v$ by 
	\begin{equation}\label{eqweak93Nom}
		\cN_\Omega v(\xi) = \sup_{x\in\gamma_\Omega(\xi)} |v(x)|,\quad \mbox{ for $\xi\in\pom$,}
	\end{equation}
	where, for $\xi\in\pom$ and a fixed 
	$\alpha>0$, $\gamma_\Omega(\xi)\equiv \gamma_{\Omega,\alpha}(\xi)$ is the non-tangential ``cone'' in $\Omega$ with vertex in $\xi$
	defined by
	\begin{equation}\label{eqconealpha}
		\gamma_\alpha(\xi)=\{ x \in \Omega: |\xi-x|<(1+\alpha)\dist(x, \pom)\}.
	\end{equation}
	We also define the {\it modified non-tangential maximal function} of $v$ by
	\begin{equation}\label{eqweak94Nom}
		\wt \cN_\Omega v(\xi) = \sup_{x\in\gamma_\Omega(\xi)} \bigg(\avint_{B(x,\delta_\Omega(x)/2)}|v|^2\,dm\bigg)^{1/2},\quad \mbox{ for $\xi\in\pom$,}
	\end{equation}
	where $\delta_\Omega(x)=\dist(x,\pom)$.
	\vv
	
	For $1<p<\infty$, we say that the {\it Neumann problem} (for $\LL$) is solvable in $L^p$ if there exists. some contant $C_\LL(N_p)>0$ such that the variational solution $u:\Omega\to\R$ of the
	Neumann problem \rf{eqneumann1} satisfies
	$$\|\wt\cN_\Omega(\nabla u)\|_{L^p(\pom)} \leq C_\LL(N_p) \|g\|_{L^{p}(\pom)}.$$
	For the sake of brevity,  we will  write  that   $(N_p)_{\LL}$ (or $(N_{L^p})_{\LL}$) is solvable in $\Omega$.
	
	We say that {\it the regularity problem is solvable in $L^p$}   (write $(R_{p})_\LL$ is solvable) if  there exists some constant $C_\LL(R_p)>0$  such that, for any    Lipschitz function $f:\pom\to\R$,  the solution $u:\Omega\to\R$ of the continuous
	Dirichlet problem for $\LL$ in $\Omega$ with boundary data $f$ satisfies
	\begin{equation}\label{eq:main-est-reg}
		\|\wt \cN_\Omega(\nabla u)\|_{L^p(\sigma)} \leq C_\LL(R_p)\|\nabla_t f\|_{L^p(\sigma)}.
	\end{equation}
	As in \cite{MT}, one may also define the regularity problem in terms of the Haj\l asz-Sobolev space. However, for the purposes of the present paper we prefer the above definition. See Section \ref{sec24.} for more details about tangential derivatives and Sobolev spaces on $\pom$.

	\vv
	
	In this paper we will prove the following result.
	
	\begin{theorem}\label{teobigpi}
		Let $\Omega\subset\R^{n+1}$ be a bounded $C_1$-chord-arc domain and let $\mathcal{L} \in\mathcal{E}_{\rm DMO}(\R^{n+1})$. Let $p\in (1,2)$, suppose that $(R_q)_\LL$ is solvable in $\Omega$ for some $q>p$, and that $\partial\Omega$ supports a weak $p$-Poincar\'e inequality. 
		Suppose that for every $\xi\in\pom$ and $0<r\leq\diam(\pom)$
		there exists a $C_2$-chord-arc domain $U_{\xi,r}$ such that $ B(\xi,r) \cap \om \subset U_{\xi,r}$, $(N_q)_\LL$ is solvable in $U_{\xi,r}$ uniformly on $\xi$ and $r$, and
		\begin{equation}\label{eqbigpiece}
			\HH^n(B(\xi,r)\cap \pom\setminus \partial U_{\xi,r})\leq \ve\,r^n,
		\end{equation}
		for some $\ve>0$.
		If  $\ve>0$ is small enough (depending only on $n$, $C_1$, $C_2$, the solvability of 
		$(R_q)_\LL$ in $\Omega$, and the uniform solvability of $(N_q)_\LL$  in the domains $U_{\xi,r}$), then $(N_p)_\LL$ is solvable in $\Omega$.
	\end{theorem}
	\vv
	
	Under the assumptions above, we call $U_{\xi,r}$ a superdomain for $\Omega$ relative to the ball $B(\xi,r)$.
	For the precise definition of $C$-chord-arc domain and the notion of weak $p$-Poincar\'e inequality, see Sections \ref{subsec-chordarc}
	and \ref{sec24.}, respectively. 
	We remark that Theorem \ref{teobigpi}
	is new even for the Laplace operator. In fact, since $(N_2)$ is solvable for the Laplacian in Lipschitz domains (by \cite{JeK1b}),
	we get the following.
	
	\begin{coro}\label{corobigpi}
		Let $\Omega\subset\R^{n+1}$ be a bounded $C_1$-chord-arc domain. Let $p\in (1,2)$, suppose that $(R_q)_\Delta$ is solvable in $\Omega$ for some $q>p$, and that $\partial\Omega$ supports a weak $p$-Poincar\'e inequality. 
		Suppose that for every $\xi\in\pom$ and $0<r\leq\diam(\pom)$
		there exists a Lipschitz domain $U_{\xi,r}$ (with Lipschitz character uniform on $\xi$ and $r$) such that $ B(\xi,r) \cap \om \subset U_{\xi,r}$ and
		\begin{equation}\label{eqbigpiece''}
			\HH^n(B(\xi,r)\cap \pom\setminus \partial U_{\xi,r})\leq \ve\,r^n,
		\end{equation}
		for some $\ve>0$.
		If  $\ve>0$ is small enough (depending only on $n$, the solvability of 
		$(R_q)_\Delta$ in $\Omega$, and the Lipschitz character of $U_{\xi,r}$), then $(N_p)_\Delta$ is solvable in $\Omega$.
	\end{coro}

	\vv
	In the last decade, there has been significant activity in the {area of boundary value problems for elliptic PDE's  in rough domains and related free boundary problems.} The primary goal of this research program was to find necessary and sufficient geometric conditions for a domain \(\Omega\) {with $n$-Ahlfors regular boundary} that guarantee the solvability of the Dirichlet problem with \(L^p\) boundary data for the Laplace operator or more general operators $\LL$ in the domain, denoted as \((D_p)_{\mathcal{L}}\) being solvable in \(\Omega\). This question was settled for the Laplace operator  by Azzam, Hofmann, Martell, and the authors of the present manuscript in \cite{AHMMT}.
	
	The method to achieve solvability of \((D_p)_{\mathcal{L}}\) in domains more general than Lipschitz is via the so-called Big Pieces functor. Specifically, chord-arc domains have Big Pieces of (interior or exterior) starlike Lipschitz subdomains (see \cite{DJ} and \cite{Semmes}). Using either the work of Dahlberg \cite{Dah} for the Laplace operator or the work of Kenig and Pipher \cite{KP01} for the so-called Dahlberg-Kenig-Pipher {or DKP} operators  (i.e., operators where \(\nabla A\) satisfies certain \(L^2\) type Carleson measure conditions), combined with the maximum principle, it can be shown that harmonic/elliptic measures belong to the \(A_\infty\) class of Muckenhoupt weights. This condition implies that there exists \(p > 1\) such that \((D_p)_{\mathcal{L}}\) is solvable in the chord-arc domain. Similarly, one can use the aforementioned result to obtain \((D_p)_{\mathcal{L}}\) solvability in domains that have Big Pieces of interior chord-arc domains. This is shown in \cite{AHMMT} to be the optimal class of corkscrew domains with Ahlfors regular boundaries in which \((D_p)_{\Delta}\) is solvable. 
	
	Despite the fact that solvability of \((D_p)_{\Delta}\) in chord-arc domains was settled as early as 1990 by David and Jerison \cite{DJ}, and independently by Semmes \cite{Semmes},  solvability of the regularity problem in \(L^q\) for \(\Delta\) had only been proved in Lipschitz domains up to recently (see \cite{JeK1} for \(q=2\), \cite{V} for \(1 < q \leq 2\), and \cite{DaKe} for the optimal range of exponents).  The existence of \(q\) such that \((R_q)_{\Delta}\) is solvable in chord-arc domains,  a question posed by Kenig in 1991 \cite[Problem 3.2.2]{Ke}, was finally proved in 2021 by the authors (see \cite{MT}). In fact, in this work a more general result was shown, namely, that in corkscrew domains with Ahlfors regular boundaries, \((D_p)_{\Delta} \Rightarrow (R_{p'})_{\Delta}\), where \(p'\) is the Hölder conjugate of \(p\). Moreover, the same paper demonstrated that \((R_{p'})_{\mathcal{L}} \Rightarrow (D_p)_{\mathcal{L}}\) for any \(\mathcal{L} \in \mathcal{E}(\Omega)\).

	For elliptic operators in divergence form with non-constant coefficients less is known.
	For example,
	for the so-called DKP operators, solvability of \((R_q)_{\mathcal{L}}\) for some $q>1$ was not known even in the ball until quite recently. The only known result was by Dindoš, Pipher, and Rule \cite{DPR}, where the authors showed in 2017 that \((R_q)_{\mathcal{L}}\) is solvable in Lipschitz domains with sufficiently small Lipschitz constants and for DKP operators whose coefficients have sufficiently small oscillations. In \cite{MPT}, inspired by ideas in \cite{MT}, the authors, in collaboration with Poggi, proved that if a corkscrew domain has a uniformly \(n\)-rectifiable boundary, then \((D_p)_{\mathcal{L}} \Rightarrow (R_{p'})_{\mathcal{L}^*}\) for any operator \(\mathcal{L}\) satisfying the DKP condition (without any smallness assumption on the oscilation of the coefficients). To do so, they introduced two new Poisson problems with data in certain scale-invariant \(L^p\)-Carleson-type spaces. They called them Poisson Dirichlet and Poisson regularity problems, denoted by \((PD_{p})_{\mathcal{L}}\) and \((PR_{p'})_{\mathcal{L}^*}\), and proved that in corkscrew domains with Ahlfors regular boundaries, for any \(\mathcal{L} \in \mathcal{L}(\Omega)\), it holds that
	$$
	(D_p)_{\mathcal{L}} \Longleftrightarrow (PD_{p})_{\mathcal{L}} \Longleftrightarrow (PR_{p'})_{\mathcal{L}^*}.
	$$
	Simultaneously and independently, Dindo\v s, Hofmann, and Pipher \cite{DHP} (see also \cite{Fe}) showed that \((R_q)_{\mathcal{L}}\) is solvable for DKP operators in Lipschitz graph domains. Their proof is significantly shorter although it uses the existence of a preferred direction and thus it can't be generalized to rougher domains. We would also like to mention the work by Gallegos and us \cite{GMT}, where extrapolation of solvability of the regularity and the Poisson regularity problems is obtained in corkscrew domains with Ahlfors regular boundaries for any \(\mathcal{L} \in \mathcal{E}(\Omega)\), as well as the work by one of us and Zacharopoulos \cite{MZ} where Varopoulos' extensions are constructed and used to obtain similar duality results for elliptic systems with complex coefficients.
	
	While recent advances have led to a pretty good understanding of the Dirichlet and regularity problems in rough domains, this is not the case for the Neumann problem with data in \(L^p\). Indeed, for a Lipschitz domain with connected boundary, in 1987 Dahlberg and Kenig proved in \cite{DaKe} that \((N_p)_{\Delta}\) is solvable for \(p \in (1, 2+\varepsilon)\). This range is optimal since for any \(p > 2\), there exists a Lipschitz domain such that \((N_p)_{\Delta}\) is not solvable. Their proof uses the  solvability of \((N_2)_{\Delta}\), achieved previously by Jerison and Kenig in \cite{JeK1b} via the so-called two-sided Rellich inequality in \(L^2\), i.e., \(\|\partial_\nu u\|_2 \approx \|\nabla_t u\|_2\).
	
	Moreover,  in \cite{DPR} it was shown that in a Lipschitz domain in \(\mathbb{R}^2\) with a sufficiently small Lipschitz constant and for DKP operators  with coefficients with sufficiently small oscillation, \((N_p)_{\mathcal{L}}\) is solvable for any \(p \in (1, \infty)\). More recently, it was proved in \cite{DHP} that if \(\Omega\) is a Lipschitz graph domain in \(\mathbb{R}^2\) with \(\mathcal{L}\) satisfying the DKP condition, then there exists \(q \in (1, \infty)\) such that \((N_q)_{\mathcal{L}}\) is solvable. This result follows from a reduction to the solvability of a relevant regularity problem \((R_q)_{\widetilde{\mathcal{L}}}\), where \(\widetilde{\mathcal{L}} = -\text{div} (\widetilde{A}(\cdot) \nabla)\) and \(\widetilde{A} = A / \det A\) (an idea originated from work of Kenig and Rule \cite{KR}). On the other hand, for any \(p > 2\) there are examples of Lipschitz domains in \(\mathbb{R}^{n+1}\), with \(n \geq 2\), such that \((R_p)_{\Delta} \nRightarrow (N_p)_{\Delta}\) (see \cite[Lemma 3.1]{KP95}). Nevertheless, it is not clear if one should expect that \((R_p)_{\Delta} \Rightarrow (N_p)_{\Delta}\) for \(p \in (1,2]\).
	
	An open problem posed by Kenig in 1991  \cite[Problem 3.2.2]{Ke} and reintroduced by Toro at the ICM in 2010 \cite[Question 2.5]{To} is the following:
	\begin{prob}
		In a bounded chord-arc domain $\Omega \subset \R^{n+1}$, $n \geq 2$, does there exist $p>1$ such that the Neumann problem for the Laplacian with  boundary data in $L^{p}(\pom)$ is solvable?
	\end{prob}
	If $\Omega \subset \R^2$ is a bounded chord-arc domain, Jerison and Kenig \cite{JeK3} showed that $(N_p)_{\Delta} \Leftrightarrow (R_p)_{\Delta} \Leftrightarrow (D_{p'})_{\Delta}$. For a result that applies to rougher but also flatter domains than Lipschitz, we refer to the work of Hofmann, Mitrea, and Taylor \cite[Section 7]{HMT}, where they prove that for every $p \in (1,\infty)$, there exists an $\varepsilon > 0$ such that for every $\varepsilon$-regular SKT domain (see \cite[Definition 4.8]{HMT} for the definition), $(N_p)_{\Delta}$ is solvable. 
	By Semmes' decomposition (see \cite[Theorem 4.16]{HMT}), the boundaries of such domains have very big pieces of sufficiently flat Lipschitz graphs with the Lipschitz constant depending on $\varepsilon$.  Note that for general chord-arc domains, 
	the lack of flatness does not allow one to prove invertibility of layer potentials as in \cite{HMT}, thus necessitating different 
	methods.

	In a recent  interesting work,  Feneuil and Li \cite{FL} proved the Neumann counterpart of \cite[Theorem 1.22]{MPT}, exploring the connections between the solvability of Poisson Neumann problems with interior data in  appropriate Carleson-type spaces and the solvability of $(N_p)_{\mathcal{L}}$. They also demonstrated the extrapolation of the solvability of the Poisson Neumann problem, which, in turn, led to the extrapolation of the solvability of $(N_p)_{\mathcal{L}}$ under the assumption of the solvability of $(D_{p'})_{\mathcal{L}^*}$. This result enhances \cite[Theorem 6.3]{Kenig-Pipher}, even in the context of the ball, where the extrapolation of $(N_p)_{\mathcal{L}}$ was proved assuming the solvability of $(R_p)_{\mathcal{L}}$ (which,  by \cite[Theorem A.2]{MT},  in corkscrew domains with Ahlfors regular boundaries,  implies the solvability of $(D_{p'})_{\mathcal{L}^*}$).  It is noteworthy that, simultaneously and independently,  \cite[Theorem 6.3]{Kenig-Pipher} was extended to chord-arc domains by Hofmann and Sparrius \cite{HS} adapting the method in \cite{Kenig-Pipher},  which itself was an interesting  achievement.

	The Neumann problem with \(L^p\) data is a notoriously difficult problem. Unlike the Dirichlet problem, one cannot use the maximum principle to {transfer} solvability from the subdomains to the original domain. Note that the solution of the Dirichlet problem has a representation via the ``Poisson kernel,'' which is a positive function, allowing one to split the data into its positive and negative parts and work with positive solutions. One of the major challenges in the Neumann problem is to find a way to achieve the transference of solvability between domains and subdomains {(or superdomains).}
	
	The appropriate  dual version of the Neumann problem is a ``rough'' Neumann problem with data \(f\) in a ``negative'' Sobolev space.  Loosely speaking, for a function $f:\pom\to\R$ with $\int_\pom f\,d\sigma=0$,
	there exists a vector field \(\vec{g} \in L^p(\partial \Omega; \mathbb{R}^{N_n})\)\footnote{If $\om=\R^{n+1}_+$ and $\partial \Omega = \mathbb{R}^n\), this is exactly the space we are interested in, with \(N_n = n\).} such that \(f = -\text{div}_t \vec{g}\) in a certain sense, where \(\text{div}_t\) denotes the tangential divergence along the boundary. Then
	the solution \(u\) of the variational Neumann with boundary data $f$ problem can be expressed via the Neumann function as (see
	Theorem \ref{teoneumann2}):
	\begin{align*}
		u(x) &= \int_{\partial \Omega} N(x, \xi) f(\xi) \,d\sigma(\xi) = \int_{\partial \Omega} \nabla_t N(x, \xi) \cdot \vec{g}(\xi) \,d\sigma(\xi) \\
		&= c \sum_{j=1}^{N_n} \int_{\partial \Omega} \tilde{\partial_j} N(x, \xi) g_j(\xi) \,d\sigma(\xi),
	\end{align*}
	where \(\tilde{\partial}_j\) represents an appropriate version of the tangential partial derivatives. Now, it is clear that \(\nabla_t N(x, \xi)\) plays the role that the Poisson kernel plays in the Dirichlet problem, which, in nice domains, coincides with \(\partial_{\nu_A} G(x, \xi)\) (here \(G(\cdot, \cdot)\) stands for the Green function).  Due to the lack of information about the sign of the partial derivatives of the Neumann function on the boundary, we cannot hope to be able to introduce some kind of positive Neumann harmonic measure and apply the maximum principle.
	We remark that in Section 
	\ref{sec-roughneumann} we will introduce another more
	hands-on dual version of the Neumann problem which will be better suited for our purposes.

	\vv

	Our proof of Theorem \ref{teobigpi} uses a good $\lambda$ type argument applied to the rough Neumann problem together with a bootstrapping 
	procedure which requires to work with a weak $(p',p')$ version of the solvability of the rough Neumann problem. By duality,  in turn, this leads us to study the solvability of the Neumann problem from the Lorentz space $L^{p,1}(\sigma)$ to $L^p(\sigma)$.
	Altogether, we get an estimate of the form
	\begin{equation}\label{eqboot51}
		C_\LL(N^R_{L^{p'},L^{p',\infty}})\leq K\, \big(1+ \ve^{a}\,C_\LL(N^R_{L^{p'},L^{p',\infty}})\big),
	\end{equation}
	where $C_\LL(N^R_{L^{p'},L^{p',\infty}})$ is the constant of the weak $(p',p')$ solvability of the rough Neumann problem
	(see Section \ref{sec-roughneumann}), and $a,K$ are positive constants.
	So assuming that $C_\LL(N^R_{L^{p'},L^{p',\infty}})<\infty$, for $\ve$ small enough, we deduce that $C_\LL(N^R_{L^{p'},L^{p',\infty}})\lesssim K$ and then the theorem follows by interpolation between different values of $p$. To ensure that $C_\LL(N^R_{L^{p'},L^{p',\infty}})<\infty$, for $\rho>0$ we introduce $\rho$-smooth versions of the Neumann and rough Neumann problems and instead of proving \rf{eqboot51} directly, we prove this for that $\rho$-smooth version, uniformly on $\rho$.
	
	By the results of David and Jerison \cite{DJ} and \cite{Semmes} mentioned above, it is known that any chord-arc domain $\Omega$ has
	big pieces of Lipschitz superdomains. That is, there exists some $\ve\in (0,1)$ such that for all $\xi\in\pom$ and $0<r\leq\diam(\pom)$
	there exists a Lipschitz domain $U_{\xi,r}$ (with uniform Lipschitz character) such that $\Omega\cap B(\xi,r)\subset U_{\xi,r}$ and
	\begin{equation}\label{eqbigpiece'}
		\HH^n(B(\xi,r)\cap \pom\setminus \partial U_{\xi,r})\leq \ve\,\HH^n(B(\xi,r)\cap \pom).
	\end{equation} 
	Unfortunately, our proof of Theorem \ref{teobigpi} requires to choose the parameter $\ve>0$ in \rf{eqbigpiece'} small enough, so that
	$K\ve^{a}<1$ in \rf{eqboot51} and the term $K\ve^{a}\,C_\LL(N^R_{L^{p'},L^{p',\infty}})$ can be absorbed by the left hand side.
	{Informally speaking, one requires $\Omega$ to have {\em very} big pieces of Lipschitz superdomains (or chord-arc superdomains in case the domains $U_{\xi,r}$ are allowed to be chord-arc, as in Theorem \ref{teobigpi}).}
	One may imagine that perhaps an iterative application of Theorem~\ref{teobigpi} might be used to allow for values of $\ve$ close to $1$ in 
	\rf{eqbigpiece'}. By Corollary \ref{corobigpi}, for the Laplacian this would imply the solvability of $(N_p)_\Delta$  for some $p>1$ in chord-arc domains $\Omega$ whose boundaries support a suitable Poincar\'e inequality (in particular, in two-sided chord-arc domains),   since  $(R_p)_\Delta$
	is solvable in such domains $\Omega$, by \cite{MT}. We do not discard that an approach of this type might work, although
	this might present important difficulties, such as the dependence of the constant $K$ in \rf{eqboot51} on the solvability constant of 
	$(N_q)_\Delta$ in the superdomains $U_{\xi,r}$, which might increase in an iterative application of Theorem \ref{teobigpi}.
	\vv
	
	The reason why we assume $A$ to have Dini mean oscillation in Theorem \ref{teobigpi} is because this ensures that, for
	any function $u$ such that $\LL u=0$ in an open ball $B$ with radius $r(B)\leq C$, it holds that $\nabla u$ is continuous in $B$ and
	\begin{equation}\label{eqgradpunt}
		\sup_{x\in \frac12 B} |\nabla u(x)|\lesssim \avint_{B}|\nabla u(y)|\,dy.
	\end{equation} 
	See \cite{DK}. 
	Notice that \rf{eqgradpunt} implies that, if $\LL u =0$ in $\Omega$, then
	\begin{equation}\label{eqgradpunt'}
		\wt\cN_\Omega(\nabla u)(x)\lesssim \cN_\Omega(\nabla u)(x)\quad \mbox{ for all $x\in\Omega$.}
	\end{equation}

	\vv
	
	\noi {\bf Acknowledgement.} We are grateful to Steve Hofmann for the discussions we had with him about the topic of this paper and for sharing with us  a preliminary version of his work with Sparrius \cite{HS}.
	
	\vv

	\noindent{\bf Data availability statement:} Data sharing is not applicable to this article as no new data were created or analyzed in this study.
	
	\vv
	
	\section{Preliminaries}
	
	In the paper, constants denoted by $C$ or $c$ depend just on the dimension and perhaps other fixed
	parameters, such as the ellipticity of the operator $\LL$. Constants with subindices, such as $C_0$, retain their value at different occurrences.
	We write $a\lesssim b$ if there is $C>0$ such that $a\leq Cb$, and we write $a\approx b$ if $a\lesssim b\lesssim a$. The notation  $a\lesssim_\gamma b$ means that  $a\lesssim b$, with the implicit constant depending on $\gamma$.

	\vv

	\subsection{Measures, rectifiability, and dyadic lattices}\label{subsec:measures}
	All measures in this paper are assumed to be Borel measures.
	We denote by $\HH^n$ the $n$-dimensional Hausdorff measure, and we assume $\HH^n$ to be normalized so that it coincides with $n$-dimensional Lebesgue measure in
	$\R^n$. Sometimes we will denote the $(n+1)$-dimensional Lebesgue measures in $\R^{n+1}$ by $m$, and integration with respect to $dx$ or $dy$ also means integration with respect to Lebesgue measure.
	For a measure $\mu$, a Borel set $A$, and a function $f\in L^1_{loc}(\mu)$, we use the notation
	$$m_{\mu,A}(f) = \avint_A f\,d\mu = \frac1{\mu(A)}\int_A f\,d\mu.$$
	In case $\mu$ is the Lebesgue measure, we just write $m_A(f)$.
	
	A measure $\mu$ in $\R^{n+1}$ is called 
	\textit{$C_0$-$n$-Ahlfors regular} (or just \textit{$n$-Ahlfors regular or Ahlfors regular} or Ahlfors-David regular) if there exists some
	constant $C_{0}>0$ such that
	$$C_0^{-1}r^n\leq \mu(B(x,r))\leq C_0\,r^n\quad \mbox{ for all $x\in
		\supp\mu$ and $0<r\leq \diam(\supp\mu)$.}$$
	The measure $\mu$ is  \textit{uniformly  $n$-rectifiable} if it is 
	$n$-Ahlfors regular and
	there exist constants $\theta, M >0$ such that for all $x \in \supp\mu$ and all $0<r\leq \diam(\supp\mu)$ 
	there is a Lipschitz mapping $g$ from the ball $B_n(0,r)$ in $\R^{n}$ to $\R^{n+1}$ with $\text{Lip}(g) \leq M$ such that
	$
	\mu (B(x,r)\cap g(B_{n}(0,r)))\geq \theta r^{n}.$
	
	A set $E\subset \R^{n+1}$ is called $n$-{\textit {rectifiable}} if there are Lipschitz maps
	$f_i:\R^n\to\R^{n+1}$, $i=1,2,\ldots$, such that 
	$
	\HH^n\Bigl(E\setminus\bigcup_i f_i(\R^n)\Bigr) = 0.
	$
	The set $E$ is $C_0$-$n$-Ahlfors regular if $\HH^n|_E$ is $C_0$-$n$-Ahlfors regular.
	Also, 
	$E$ is uniformly $n$-rectifiable if $\HH^n|_E$ is uniformly $n$-rectifiable.
	The notion of uniform rectifiability is a quantitative version of rectifiability which was introduced by David and Semmes in the pioneering works \cite{DS1} and \cite{DS2}.

	Given an $n$-Ahlfors measure $\mu$ in $\R^{n+1}$, we consider 
	the dyadic lattice of ``cubes'' built by David and Semmes in \cite[Chapter 3 of Part I]{DS2}. The properties satisfied by $\DD_\mu$ are the following. 
	Assume first, for simplicity, that $\diam(\supp\mu)=\infty$). Then for each $j\in\Z$ there exists a family $\DD_{\mu,j}$ of Borel subsets of $\supp\mu$ (the dyadic cubes of the $j$-th generation) such that:
	\begin{itemize}
		\item[$(a)$] each $\DD_{\mu,j}$ is a partition of $\supp\mu$, i.e.\ $\supp\mu=\bigcup_{Q\in \DD_{\mu,j}} Q$ and $Q\cap Q'=\varnothing$ whenever $Q,Q'\in\DD_{\mu,j}$ and
		$Q\neq Q'$;
		\item[$(b)$] if $Q\in\DD_{\mu,j}$ and $Q'\in\DD_{\mu,k}$ with $k\leq j$, then either $Q\subset Q'$ or $Q\cap Q'=\varnothing$;
		\item[$(c)$] for all $j\in\Z$ and $Q\in\DD_{\mu,j}$, we have $2^{-j}\lesssim\diam(Q)\leq2^{-j}$ and $\mu(Q)\approx 2^{-jn}$;
		\item[$(d)$] there exists $C>0$ such that, for all $j\in\Z$, $Q\in\DD_{\mu,j}$, and $0<\tau<1$,
		\begin{equation}\label{small boundary condition}
			\begin{split}
				\mu\big(\{x\in Q:\, &\dist(x,\supp\mu\setminus Q)\leq\tau2^{-j}\}\big)\\&+\mu\big(\{x\in \supp\mu\setminus Q:\, \dist(x,Q)\leq\tau2^{-j}\}\big)\leq C\tau^{1/C}2^{-jn}.
			\end{split}
		\end{equation}
		This property is usually called the {\em small boundaries condition}.
		From (\ref{small boundary condition}), it follows that there is a point $x_Q\in Q$ (the center of $Q$) such that $\dist(x_Q,\supp\mu\setminus Q)\gtrsim 2^{-j}$ (see \cite[Lemma 3.5 of Part I]{DS2}).
	\end{itemize}
	We set $\DD_\mu:=\bigcup_{j\in\Z}\DD_{\mu,j}$. 
	
	In case that $\diam(\supp\mu)<\infty$, the families $\DD_{\mu,j}$ are only defined for $j\geq j_0$, with
	$2^{-j_0}\approx \diam(\supp\mu)$, and the same properties above hold for $\DD_\mu:=\bigcup_{j\geq j_0}\DD_{\mu,j}$.
	
	Given a cube $Q\in\DD_{\mu,j}$, we say that its side length is $2^{-j}$, and we denote it by $\ell(Q)$. Notice that $\diam(Q)\leq\ell(Q)$. 
	We also denote 
	\begin{equation}\label{defbq}
		B(Q):=B(x_Q,c_1\ell(Q)),\qquad B_Q := B(x_Q,\ell(Q)),
	\end{equation}
	where $c_1>0$ is some fixed constant such that $B(Q)\cap\supp\mu\subset Q$ for all $Q\in\DD_\mu$. Clearly, we have $Q\subset B_Q$.
	
	For $\lambda>1$, we write
	$$\lambda Q = \bigl\{x\in \supp\mu:\, \dist(x,Q)\leq (\lambda-1)\,\ell(Q)\bigr\}.$$
	
	
	The side length of a ``true cube'' $P\subset\R^{n+1}$ is also denoted by $\ell(P)$. On the other hand, given a ball $B\subset\R^{n+1}$, its radius is denoted by $r(B)$. For $\lambda>0$, the ball $\lambda B$ is the ball concentric with $B$ with radius $\lambda\,r(B)$.
	
	\vv
	
	\subsection{The Whitney decomposition of $\Omega$}\label{secwhitney}
	
	For any open set $\Omega \subsetneq \R^{n+1}$,  there is a family $\WW(\Omega)$ (the Whitney cubes of $\Omega$) of dyadic cubes in $\R^n$ with disjoint interiors contained in $\Omega$ such that
	$\bigcup_{P\in\WW(\Omega)} P = \Omega,$
	and moreover there are
	some constants $\Lambda>20$ and $D_0\geq1$ such that the following holds for every $P \in\WW(\Omega)$:
	\begin{itemize}
		\item[(i)] $10P \subset \Omega$;
		\item[(ii)] $\Lambda P \cap \partial\Omega \neq \varnothing$;
		\item[(iii)] there are at most $D_0$ cubes $P'\in\WW(\Omega)$
		such that $10P \cap 10P' \neq \varnothing$. Further, for such cubes $P'$, we have $\frac12\ell(P')\leq \ell(P)\leq 2\ell(P')$.
	\end{itemize}
	From the properties (i) and (ii) it is clear that $\dist(P,\partial\Omega)\approx\ell(P)$. We assume that
	the Whitney cubes are small enough so that
	\begin{equation}\label{eqeq29}
		\diam(P)< \frac1{20}\,\dist(P,\partial\Omega).
	\end{equation}
	The arguments to construct a Whitney decomposition satisfying the properties above are
	standard.
	\vv
	
	Suppose that $\pom$ is $n$-Ahlfors regular and consider the dyadic lattice $\DD_\sigma$ defined above, for $\sigma=\HH^n|_\pom$.
	Then, for each Whitney cube $P\in \WW(\Omega)$ there is some cube $Q\in\DD_\sigma$ such that $\ell(Q)=\ell(P)$
	and $\dist(P,Q)\approx \ell(Q)$, with the implicit constant depending on the parameters of $\DD_\sigma$ and on the Whitney decomposition. We denote this by $Q=b(P)$ (``b'' stands for ``boundary''). Conversely,
	given $Q\in\DD_\sigma$, we let 
	\begin{equation}\label{eqwq00}
		w(Q) = \bigcup_{P\in\WW(\Omega):Q=b(P)} P.
	\end{equation}
	It is immediate to check that $w(Q)$ is made up at most of a {uniformly} bounded number of cubes $P$, but it may happen
	that $w(Q)=\varnothing$.

	\vv
	\subsection{Chord-arc domains}\label{subsec-chordarc}
	
	A domain is a connected open set.
	In the whole paper, $\Omega$ will be an open set in $\R^{n+1}$, with $n\geq 1$.
	Often we will denote the $n$-Hausdorff measure on $\pom$ by $\sigma$.
	
	The open set $\Omega$ satisfies the $C_1$-corkscrew condition if
	there exists some $C_1>0$ such that for all
	$x\in\pom$ and all $r\in(0, 2\diam(\Omega))$ there exists a ball $B\subset B(x,r)\cap\Omega$ such that
	$r(B)\geq C_1^{-1}\,r$. We also say that $\Omega$ is a $C_1$-corkscrew domain, or just a corkscrew domain if we do not want to mention the constant $C_1$.
	
	Given two points $x,x' \in \Omega$, and a pair of numbers $M,N\geq1$, 
	an $(M,N)$-{\it Harnack chain connecting $x$ to $x'$},  is a chain of
	open balls
	$B_1,\dots,B_N \subset \Omega$, 
	with $x\in B_1,\, x'\in B_N,$ $B_k\cap B_{k+1}\neq \varnothing$
	and $M^{-1}\diam (B_k) \leq \dist (B_k,\partial\Omega)\leq M\diam (B_k).$
	For $C_2\geq1$, we say that $\Omega$ satisfies the {\it $C_2$-Harnack chain condition} if
	for any two points $x,x'\in\Omega$,
	there is an $(M,N)$-Harnack chain connecting them, with $M\leq C_2$ and $N$ such that
	$$N\leq C_2\,\bigg(1+\log^+\frac{|x-x'|}{\min(\delta_\Omega(x),\delta_\Omega(x'))}\bigg).$$
	Recall that $\delta_{\Omega}(x)=\dist(x, \pom)$.

	We  say that a domain $\Omega$ is $C_3$-uniform, if it satisfies the
	$C_3$-Harnack chain condition and the $C_3$-corkscrew condition.  
	Following \cite{JeK2}, we say that a
	domain $\Omega\subset \mathbb{R}^{n+1}$ is $C_3$-NTA ({\it non-tangentially accessible})  if it is uniform\footnote{In \cite{JeK2} the Harnack chain condition is formulated in a different but equivalent form.}  and $\Omega_{\textup{ext}}:= \R^{n+1}\setminus \overline{\Omega}$ satisfies the $C_3$-corkscrew condition. We also say that a connected open set $\Omega \subset \R^{n+1}$
	is a $C_3$-CAD ({\it chord-arc domain}), if it is $C_3$-NTA and  $\pom$ is $C_3$-$n$-Ahlfors regular. In this case, we say that $C_3$ is the CAD constant of $\Omega$. Additionally, if a domain $\Omega$ and its exterior $\R^{n+1} \setminus \overline\Omega$  are $C$-CAD, then we say that $\Omega$ is a two-sided $C$-CAD. To simplify notation, we may write NTA or CAD in place of $C$-NTA or $C$-CAD.

	It was shown independently by David and Jerison \cite{DJ} and \cite{Semmes} that if $\Omega$ is a CAD, then its boundary $\pom$ is uniformly
	$n$-rectifiable.

	\vv
	
	\subsection{Finite perimeter sets and  reduced boundary,} \label{subsec:finite perimeter}

	An open set $\Omega\subset\R^{n+1}$ has {\it finite perimeter} if {the distributional gradient $\nabla \chi_\Omega$ of $\chi_\Omega$}  is a locally finite $\R^{n+1}$-valued measure.  From results of De Giorgi and Moser it follows that $\nabla \chi_\Omega=
	-\nu_{\Omega}\,\HH^n_{\pom^*}$, where $\partial^*\Omega\subset\pom$ is the reduced boundary of $\Omega$ and
	$|\nu_{\Omega}(x)|=1$ $\HH^n$-a.e.\ in $\partial^*\Omega$.  By  \cite[Theorem 5.15]{EG},  $\partial^*\Omega$ can be written, up to  a set of $\HH^n$-measure zero,  as a countable union of  compact sets $\{K_j\}_{j=1}^\infty$ where $K_j \subset S_j$ for a  $C^1$ hypersurface $S_j$  and $\nu_{\Omega}|_{S_j}$ is  normal to $S_j$.  Moreover, the following Green's formula is satisfied: for every $\vphi \in C^\infty_c(\R^{n+1};\R^{n+1})$, 
	\begin{equation}\label{eq:Green-finiteperim}
		\int_\Omega \dv \varphi(x)\,dx = \int_{\partial^*\Omega} \nu_{\Omega}(\xi) \cdot \vphi(\xi)\,\HH^n(\xi).
	\end{equation}
	More generally, given a unit vector $\nu_\Omega$ and $x \in \pom$, we define the (closed) half-spaces 
	$$
	H^\pm_{\nu_\Omega}(x)= \{ y \in \R^{n+1}: \pm\nu_\Omega \cdot (y-x) \geq 0\}.
	$$
	Then, for $x \in \partial^*\Omega$, it holds
	\begin{equation}\label{eq:normal-halfspace}
		\lim_{r \to 0} r^{-(n+1)} m \left(B(x,r)\cap \Omega^\pm \cap H^\pm_{\nu_\Omega}\right) >0,
	\end{equation}
	where $\Omega^+=\om$ and $\Omega^-=\R^{n+1} \setminus \Omega$ (see for instance \cite[p. 230]{EG}).  A unit vector $\nu_\Omega$  satisfying \eqref{eq:normal-halfspace} 
	is called the {\it measure theoretic outer unit normal} to $\Omega$ at $x$ and we denote by $\partial_0 \Omega$ all the points of $\pom$ for which \eqref{eq:normal-halfspace} holds.  It is clear that $\partial^* \Omega \subset \partial_0 \Omega$.

	The {\it measure theoretic boundary} $\partial_*\Omega$ consists of the points $x\in\pom$ such that
	$$\limsup_{r\to0}\frac{m(B(x,r)\cap\Omega)}{r^{n+1}}>0\quad \mbox{ and }\quad
	\limsup_{r\to0}\frac{m(B(x,r)\setminus \overline\Omega)}{r^{n+1}}>0.$$
	When $\Omega$ has finite perimeter, it holds that $\partial_*\Omega\subset \partial_0 \om \subset \partial^*\Omega\subset \pom$ and $\HH^n(\partial^*\Omega\setminus\partial_*\Omega)=0$. A good reference for those results is either the book of Evans and Gariepy \cite{EG} or the book of Maggi \cite{Maggi}.

	\begin{rem}\label{rem:reduced-boundary-domains}
		If $\om \subset \R^{n+1}$ is a bounded open set with  Ahlfors regular boundary, then it has finite perimeter.  If it also  satisfies the two-sided corkscrew condition,   then it holds that  $\pom = \partial_* \Omega$, (see \cite[(3.1.25),  p. 52]{HMT}).  Therefore, for such domains,   $\HH^n(\partial \Omega \setminus \partial^* \Omega)=0$. 
	\end{rem}

	
	\vv
	
	\subsection{The Haj\l asz Sobolev space and the HMT Sobolev space}\label{sec24.}

	Let $\Sigma$  be a metric space equipped with a doubling
	measure $\sigma$ on $\Sigma$, which means that there is a uniform constant $C_\sigma\geq1$ such that $\sigma(B(x,2r))\leq C_\sigma\, \sigma(B(x,r))$, for all $x\in \Sigma$ and $ r>0$. We will now  define  the {\it Haj\l{}asz's Sobolev spaces}  $\dot{M}^{1,p}(\Sigma)$ and  ${M}^{1,p}(\Sigma)$, which were introduced in \cite{Hajlasz}. For  more information on those spaces and, in general, Sobolev spaces in metric measure  spaces, the reader may consult \cite{Heinonen}.
	
	For a Borel function $f:\Sigma\to\R$, we say that a non-negative Borel function $g:\Sigma \to \R$ is a {\it Haj\l asz upper gradient of  $f$} if  
	\begin{equation}\label{eq:H-Sobolev}
		|f(x)-f(y)| \leq |x-y| \,(g(x)+g(y))\quad \mbox{ for $\sigma$-a.e. $x, y \in \Sigma$.} 
	\end{equation}
	We denote the collection of all the Haj\l asz upper gradients of $f$ by $D(f)$.
	
	For $p\geq1$, we denote by $\dot{M}^{1,p}(\Sigma)$ the space of Borel functions $f$ which have 
	a Haj\l asz upper gradient in $L^p(\sigma)$, and we let $M^{1,p}(\Sigma)$ be the space of functions $f\in L^p(\sigma)$ which have a Haj\l asz upper gradient in $L^p(\sigma)$, i.e.,  $M^{1,p}(\Sigma)=  \dot M^{1,p}(\Sigma) \cap L^p(\sigma)$.
	We  define the semi-norm (as it annihilates constants)
	\begin{equation}\label{eqseminorm}
		\| f \|_{ \dot M^{1.p}(\Sigma)} = \inf_{g \in D(f)} \| g\|_{L^p(\Sigma)}
	\end{equation}
	and the scale-invariant norm
	\begin{equation}\label{eqnorm}
		\| f\|_{M^{1,p}(\Sigma)} = \diam(\Sigma)^{-1} \|f\|_{L^p(\Sigma)} +   \inf_{g \in D(f)} \| g\|_{L^p(\Sigma)}.
	\end{equation}
	Remark that, 
	for any a metric space $\Sigma$, in the case $p\in (1,\infty)$, from the uniform convexity of $L^p(\sigma)$, one easily deduces
	that the infimum in the definition of the norm $\|\cdot\|_{M^{1,p}(\Sigma)}$ and $\|\cdot\|_{\dot M^{1,p}(\Sigma)}$ in \eqref{eqseminorm} and \eqref{eqnorm} is attained and is unique. We denote by $\nabla_{H,p} f$ the function $g$ which attains the infimum.

	\vv
	
	In \cite{HMT}, Hofmann, Mitrea, and Taylor have introduced some tangential derivatives and another tangential gradient which
	are  well suited for arguments involving integration by parts in chord-arc and other more general domains.

	Let $\om \subset \R^{n+1}$ be  a set of finite perimeter.  The HMT-derivatives we will introduce below  are defined on $\partial^* \Omega$.  To this end, consider a $C_c^1$ function $\vphi:\R^{n+1}\to \R$ and $1\leq j,k\leq n+1$. Then, one  defines the tangential derivatives of $\vphi$ on ${\partial^* \Omega}$  by
	\begin{equation}\label{eqparts099}
		\partial_{t,j,k} \vphi := \nu_j\,(\partial_k \vphi)|_{\partial^* \Omega} - \nu_k\,(\partial_j \vphi)|_{\partial \Omega},
	\end{equation}
	where $\nu_i$, $i=1,\ldots,n+1$ are the components of the outer unit normal $\nu$. Remark that, by integration by parts, if and $\vphi,\psi$ are $C^1$ in a neighborhood of $\pom$, the arguments in \cite[p.  2676]{HMT} show that
	\begin{equation}\label{eqparts09}
		\int_{\partial^* \Omega} \partial_{t,j,k}\psi\,\vphi\,d\HH^n = \int_{\partial^* \Omega} \psi\,\partial_{t,k,j} \vphi\,d\HH^n.
	\end{equation}
	
	Let $\sigma_*:=\HH^n|_{\partial^* \om}$.  For $1<p<\infty$, one defines the Sobolev type space $W^{1,p}(\sigma_*)\equiv W^{1,p}(\partial^* \om)$ (see \cite[display (3.6.3)]{HMT}) as the subspace of functions in $L^p(\sigma_*)$
	for which there exists some constant $C(f)$ such that
	\begin{equation}\label{eqlp1HMT}
		\sum_{1\leq j,k\leq n+1} \left|\int_{\partial \Omega} f\, \partial_{t,k,j} \vphi\,d\sigma_* \right|\leq C(f)\,\|
		\vphi\|_{L^{p'}(\sigma_*)}
	\end{equation}
	for all $\vphi\in C_c^\infty(\R^{n+1})$. By the Riesz representation theorem, for
	each $f\in W^{1,p}(\partial^* \om)$ and each $j,k=1,\ldots,n+1$, there exists some function $h_{j,k}\in L^p(\sigma_*)$ such that 
	$$
	\int_{\partial^* \Omega}  h_{j,k}\,\vphi\,d\sigma_*= \int_{\partial^* \Omega}  f\, \partial_{t,k,j} \vphi\,d\sigma_*$$
	and we set 
	$\partial_{t,j,k} f:=h_{j,k}$, so that this is coherent with \eqref{eqparts09}. It is easy to check that Lipschitz functions with compact support are contained in $W^{1,p}(\partial^* \Omega)$.

	If $f:\pom\to\R$ is Lipschitz on $\pom$, it is shown in \cite[Lemma 6.4]{MT} that
	\begin{equation*}\label{eqnab720}
		\partial_{t,j,k} f(x) =-\nu_j \,(\nabla_{t} f)_k(x) + \nu_k\,(\nabla_{t} f)_j(x)\quad\mbox{ for $\HH^n|_{\partial^* \Omega}$-a.e.\ $x\in\partial^* \Omega$,}
	\end{equation*}
	where $(\nabla_{t} f)_k$ denotes the $k$ component of the tangential gradient $\nabla_{t} f$ at $x$.
	
	\vv
	
	The proof of  the following lemma is given in Subsection \ref{sec:ApproxIdentity}.
	
	\begin{lemma}\label{lem:HMT vs Hajlasz}
		Let $\Omega \subset \R^{n+1}$ be a bounded open set with an Ahlfors regular boundary.  Then,   it holds that $ M^{1,p}(\pom) \subset  W^{1,p}(\partial^* \om)$ and, for any $f \in  M^{1,p}(\pom)$ and  $ j ,k \in \{1, \dots, n+1\}$,  
		\begin{equation}\label{eq:M emb W}
			\| \partial_{t,j,k} f\|_{L^p(\partial^*\om)} \lesssim \| \nabla_H f\|_{L^p(\pom)},
		\end{equation}
		where the implicit constant depends on $n$ and the Ahlfors regularity constants.
	\end{lemma}
	
	\vv
	
	For $\Omega$ such that $\HH^n(\pom)=\HH^n(\partial^* \om)$,  we say that $\pom$ supports a weak $p$-Poincar\'e inequality if there are constants $C$, $C'$ such that  for any Lipschitz function $f:\pom \to\R$ and any ball $B$ centered in $\pom$, the following holds:
	$$\avint_{B \cap \pom} |f- m_{\sigma,B}f|\,d\sigma \leq C\,r(B)\avint_{C'B\cap\pom }|\nabla_t f|^p\,d\sigma.$$
	If this weak $p$-Poincar\'e inequality holds, then the space $W^{1,p}(\partial\Omega)$ of Hofmann, Mitrea, Taylor
	coincides with the Haj\l asz-Sobolev space $M^{1,p}(\pom)$ (in particular, this happens when $\Omega$ is a two-sided CAD). The Haj\l asz-Sobolev space is specially useful when studying the regularity problem in rather general domains, as shown in \cite{MT}.  

	\vv
	
	\subsection{Lorentz spaces on Ahlfors regular sets}
	Here we recall some basic facts about Lorentz spaces, which will play an important role in this paper. One can find proofs
	of the results described in this subsection in \cite[Chapter 4]{BS}.

	Let $\mu$ be an $n$-Ahlfors regular finite measure on $\R^{n+1}$ (for the purposes of this paper,  we may think that $\mu=\HH^n|_\pom$, where $\Omega\subset\R^{n+1}$ is a bounded  CAD).  
	For $X=\supp\mu$ and a $\mu$-measurable function $f:X\to \R$, we denote by $d_f$ its distribution function, given by
	$$d_f(\lambda) = \mu\big(\big\{x\in X:|f(x)|>\lambda\big\}\big),\quad \mbox{ for $\lambda\geq0$.}$$
	The decreasing rearrangement of $f$ is the function $f^*:[0,\infty)\to\R$ defined by
	$$f^*(t) = \inf\{\lambda\geq 0:d_f(\lambda)\leq t\},\quad \mbox{ for $t\geq0$.}$$
	For $0< p,q \leq\infty$, the Lorentz space $L^{p,q}(\mu)$ consists of the $\mu$-measurable functions $f:X\to\R$ such that the quantity
	$$\|f\|_{L^{p,q}(\mu)} = \left\{
	\begin{array}{ll}
		\left(\int_0^\infty (t^{1/p}f^*(t))^q\,\frac{dt}t\right)^q & \quad\mbox{for $1\leq q<\infty$,}\\
		&\\
		\sup_{t>0}(t^{1/p}f^*(t))& \quad\mbox{for $q=\infty$,}
	\end{array}
	\right.
	$$
	is finite. 
	
	Remark the space $L^{p,p}(\mu)$ coincides with the Lebesgue space $L^p(\mu)$, while $L^{p,\infty}(\mu)$ coincides with the usual space weak-$L^p(\mu)$.
	From the definition above, it easily follows  that for $0<p\leq\infty$ and $0<q\leq r\leq\infty$,
	\begin{equation}\label{eqordre4}
		\|f\|_{L^{p,r}(\mu)}\lesssim_{p,q,r} \|f\|_{L^{p,q}(\mu)}.
	\end{equation}

	The spaces $L^{p,q}(\mu)$ are quasi-Banach spaces (assuming the functions in the spaces to be defined modulo sets of zero measure, as usual), and 
	$\|\cdot\|_{L^{p,q}(\mu)}$ is a quasinorm. Further, for $1<p<\infty$ and $1\leq q\leq\infty$, the space $L^{p,q}(\mu)$ is normable. That is, $\|\cdot\|_{L^{p,q}(\mu)}$ is comparable to a norm.
	Additionally, for $1<p<\infty$ and $1\leq q<\infty$, the dual of $L^{p,q}(\mu)$ can identified with $L^{p',q'}(\mu)$, with equivalence of norms.
	More generally, for $1<p<\infty$, $1\leq q\leq\infty$, we have
	\begin{equation}\label{eqclau29}
		\|f\|_{L^{p.q}(\mu)} \approx_{p,q} \sup_{\|g\|_{L^{p',q'}(\mu)}\leq 1} \int |f\,g|\,d\mu.
	\end{equation} 
	
	In this paper we will make use of the spaces $L^{p,1}(\mu)$, $L^{p,p}(\mu)=L^p(\mu)$, and $L^{p,\infty}(\mu)$, with $1<p<\infty$.
	Clearly, from \rf{eqordre4} it follows that
	$$L^{p,1}(\mu) \subset L^p(\mu)\subset  L^{p,\infty}(\mu).$$
	
	Recall that, by Kolmogorov's inequality, if $1\leq p_1<p_2<\infty$,
	$$\|f\|_{L^{p_1}(\mu)}\lesssim_{p_1,p_2}\|f\|_{L^{p_2,\infty}(\mu)}\,\mu(X)^{\frac1{p_1}- \frac1{p_2}}.$$
	Then, from \rf{eqclau29}, it follows easily that, for $1\leq p_1<p_2$, 
	\begin{equation}\label{eqclau31}
		\|f\|_{L^{p_1,1}(\mu)}\lesssim_{p_1,p_2}\|f\|_{L^{p_2}(\mu)}\,\mu(X)^{\frac1{p_1}- \frac1{p_2}}.
	\end{equation}

	We also recall that (a special case of) the Marcinkiewicz interpolation theorem asserts that, for
	$1\leq p_0<p_1<\infty$,
	if $T$ is a quasilinear operator bounded from $L^{p_i,1}(\mu)$ to $L^{p_i,\infty}(\mu)$ for $i=0,1$,
	then $T$ is also bounded in $L^{p,q}(\mu)$ for all $p_0<p<p_1$ and $1\leq q\leq \infty$.
	
	\vv

	We denote by $\Lip(X)$ the space of Lipschitz functions on $X$. We also define
	$$
	{\Lip}_0(X):=\{ f \in \Lip(X): \int_X f\,d\mu=0\} \qquad L^{p,q}_0(X):=\{ f \in L^{p,q}(X): \int_X f\,d\mu=0\} 
	$$
	
	\vv

	Note that the space $(X,\mu)$ is a strongly resonant space; see \cite[p.45, Definition 2.3]{BS} and \cite[p.49, Theorem 2.6]{BS}.  Therefore,  by  \cite[p. 23, Corollary 4.3]{BS} and \cite[p. 221, Corollary 4.8]{BS}, we have that $L^{p,q}(\mu)$ has absolutely continuous norm. 
	{In fact, for every $f\in L^{p,q}(\mu)$, with $p \in (1,\infty)$, $q \in[1,\infty)$, and every 
		$\ve>0$, there exists $\delta=\delta(f,\ve)>0$ such that 
		\begin{equation}\label{eq:absolutecontinuityLorentz}
			\textup{if}\,\, E \subset X\,\,\textup{with}\,\, \mu(E)<\delta,\,\,\textup{then}\,\, \| f\, \chi_E \|_{L^{p,q}(\mu)} < \ve.
		\end{equation}
	}
	
	\vv
	\begin{lemma}\label{lem:denseLorentz}
		Let $\mu$ be an Ahlfors $n$-Ahlfors regular measure in $\R^{n+1}$ and $X=\supp\mu$.
		For   $1<p <\infty$, and $1\leq q <\infty$,  $\Lip(X)$  is dense in $L^{p,q}(\mu)$. Moreover,  $\Lip_0(X)$ is dense in $L_0^{p,q}(\mu)$ for $p \in (1,\infty)$ and $q \in[1,p]$. 
	\end{lemma}
	
	We remark that this lemma holds for more general measures $\mu$. However, we only need this for the $n$-Ahlfors regular ones.

	\begin{proof}
		We first record that, by \cite[Theorem 2.20]{CC},  the set of simple functions is dense in $L^{p,q}(\mu)$.  We will then show that for every Borel set $E \subset X$ with $\mu(E)<\infty$,  the function $\chi_E$ can be approximated by continuous functions in the $L^{p,q}(\mu)$ norm, which would imply the same result for any simple function.  To this end,  since $\mu$ is a Radon measure,  for fixed $\delta>0$ to be chosen momentarily,  there exist a compact set $K$ and an open set $U$ such that $K \subset E \subset U$ and $\mu(U \setminus K) <\delta$.  By Urysohn's lemma,  there exists $f \in C(X)$ such that $\chi_K \leq f \leq \chi_U$.  Therefore,  for fixed $\ve>0$,  if  $\delta=\delta(\ve)>0$ is the one for which \eqref{eq:absolutecontinuityLorentz} holds,  we have that 
		$$
		\| f - \chi_E\|_{L^{p,q}(\mu)} \leq \| \chi_U - \chi_E\|_{L^{p,q}(\mu)} = \| \chi_{U \setminus E}\|_{L^{p,q}(\mu)} \leq \| \chi_{U \setminus K}\|_{L^{p,q}(\mu)}<\ve.
		$$
		We have proved that any $L^{p,q}(\mu)$ function can be approximated by continuous functions.  Since any continuous function on a compact set $X$ can be approximated by Lipschitz functions on $X$ in the uniform norm,  we readily infer that $\Lip(X)$  is dense in $L^{p,q}(\mu)$.  
		
		Let now $f \in L_0^{p,q}(\mu)$.   By the first part of the lemma, there exists $f_k \in \Lip(X)$ is such that $f_k \to f$ in the $L^{p,q}(\mu)$  norm.  Since    $\| f_k -f\|_{L^{p}(\mu)}  \lesssim \| f_k-f\|_{L^{p,q}(\mu)} $, for any $p \in (1,\infty)$ and $q \in[1,p]$, we get that $\lim_{k \to \infty} f_k \to f$ in $L^1(\mu)$ and so $\fint_X f_k \,d\mu\to \fint_X f \,d\mu=0$ as $k \to \infty$. Thus,  if we set $g_k:=f_k -\fint_X f_k\,d\mu$, it is easy to see that $g_k \in \Lip_0(X)$ and $g_k \to f$ in the $L^{p,q}(\mu)$-norm.
	\end{proof}

	\vv
	\subsection{Approximations of the identity on Ahlfors regular sets}\label{sec:ApproxIdentity}

	Let $\phi:\R^{n+1}\to\R$ be a smooth radial function such that $\chi_{B(0,1/2)}\leq \phi\leq\chi_{B(0,1)}$ and, for $\rho>0$, denote
	$\phi_\rho(x) = \phi(\rho^{-1}x)$.
	
	Let $\sigma$ be an $n$-dimensional Ahlfors regular measure on $\R^{n+1}$ and let $\Sigma:=\supp(\sigma)$.  For a function $g\in L^1_{loc}(\sigma)$, $\rho\in(0,\diam \Sigma)$, and $x\in\ \Sigma$, we denote
	$$\wt S_\rho  g(x) = \frac{\phi_\rho *(g\sigma)(x)}{\phi_\rho *\sigma(x)}.$$
	We denote by $\wt s_\rho (x,y)$ the kernel of $\wt S_\rho $ with respect to $\sigma$. That is, 
	$$\wt s_\rho (x,y) = \frac1{\phi_\rho *\sigma(x)}\,\phi_\rho (x-y)\quad \mbox{ for $x,y\in \Sigma$},$$
	so that  $\wt S_\rho  g(x)  =\int \wt s_\rho (x,y)\,g(y)\,d\sigma(y)$.  It is easy to prove that,  $ \wt s_\rho (\cdot,y) \in \Lip(\Sigma)$ with $\Lip( \wt s_\rho (\cdot,y)) \lesssim \rho^{-n-1} \|\sigma\|$, uniformly in $y \in \Sigma$.

	Let $\wt S_\rho ^*$ the dual operator of $\wt S_\rho $. That is, 
	$$
	\wt S_\rho^*  g(x)  =\int \wt s_\rho (y,x)\,g(y)\,d\sigma(y).
	$$
	Notice that $\wt S_\rho ^*1\approx 1$, but  $\wt S_\rho ^*1\not\equiv 1$, in general.
	To solve this drawback, let $W_\rho $ be the operator of multiplication by $1/\wt S_\rho ^*1$. Then we consider the operator
	$$S_\rho  = \wt S_\rho \,W_\rho \,\wt S_\rho ^*,$$
	Notice that $S_\rho $ is self-adjoint. Moreover
	$S_\rho 1 = S_\rho ^*1\equiv1$. 
	On the other hand, the kernel of $S_\rho $ is the following:
	$$s_\rho (x,y) = \int \wt s_\rho (x,z)\,\frac1{\wt S_\rho ^*1(z)}\,\wt s_\rho (y,z)\,d\sigma(z).$$
	We remark that the above construction of the operators $S_\rho$ appeared first in the proof of David, Journ\'e, and Semmes  
	of the $Tb$ theorem in homogeneous spaces \cite{DJS}. In this work the authors attribute the construction of these operators to Coifman.

	The proof of the following lemma is standard and we will omit it.

	\begin{lemma}\label{lemsr}
		Let $\sigma$ be an $n$-dimensional Ahlfors regular measure on $\R^{n+1}$ and $\Sigma:=\supp(\sigma)$.  For $\rho\in(0,\diam(\Sigma))$, let $S_\rho$ be defined as above.
		The following holds:
		\begin{itemize}
			\item[(a)] For every $x\in\Sigma$, the kernel $s_\rho(x,\cdot)$ is supported in $\bar B(x,2\rho)$ and it holds
			$s_\rho(x,\cdot)\gtrsim \chi_{B(x,\frac14\rho)}$.
			\item[(b)] For every $x\in\Sigma$,  the kernels $s_\rho(x,\cdot)$ and $s_\rho(\cdot, x)$ are  Lipschitz on $\Sigma$, with Lipschitz constants ${\rm Lip}(s_\rho(x,\cdot))+{\rm Lip}(s_\rho(\cdot,x))\lesssim \rho^{-n-1}.$
			\item[(c)] For $1\leq p \leq\infty$, $S_\rho$ is bounded in $L^p(\sigma)$ with norm at most $1$.
			\item[(d)] For $1\leq p<\infty$ and  $f\in L^p(\sigma)$,  then $S_\rho(f)$ converges to $f$ in $L^p(\sigma)$ as $\rho\to0$.
			\item[(e)] If  $f\in L^1(\sigma)$,  then $S_\rho(f) \in \Lip(\Sigma)$ and  ${\rm Lip}(S_\rho g) \lesssim \rho^{-n-1} \|g\|_{L^1(\sigma)}$.
		\end{itemize}
	\end{lemma}

	\vv
	
	\begin{lemma}\label{lemlip}
		Let $\sigma$ be an $n$-dimensional Ahlfors regular measure on $\R^{n+1}$  and $\Sigma:=\supp(\sigma)$.  For $\rho\in(0,\diam(\Sigma))$, let $S_\rho$ be defined as above. The following hold:
		\begin{itemize}
			\item[(a)] If $g \in \dot M^{1,p}(\sigma)$, then $S_\rho g \in \dot M^{1,p}(\sigma)$ satisfying
			$$
			\|\nabla_H S_\rho g\|_{L^p(\sigma)} \lesssim \|\nabla_H  g\|_{L^p(\sigma)}.
			$$
			\item[(b)] If $g\in \Lip(\Sigma)$,  then $S_\rho g \in \Lip(\Sigma)$ satisfying
			$${\rm Lip}(S_\rho g) \lesssim {\rm Lip}(g).$$
		\end{itemize}
		The implicit constants depend only on $n$ and the Ahlfors regularity constants of $\sigma$.
	\end{lemma}
	
	\begin{proof}
		Fix $\rho\in(0,\diam(\Sigma))$.
		Let $x,x'\in\Sigma$ and suppose first that $|x-x'|\geq \rho$.
		Then we write
		$$|S_\rho g(x) - S_\rho g(x')| \leq |S_\rho g(x) -  g(x)| + |g(x) -  g(x')| + |g(x') - S_\rho g(x')|.$$
		Since $S_\rho1=1$ and $\supp(s_\rho(x,\cdot))\subset \bar B(x,2\rho)$,
		we have
		\begin{align*}
			|S_\rho g(x) -  g(x)| & = \bigg|\int_\pom s_\rho(x,y) (g(y)-g(x)))\,d\sigma(y)\bigg| \\
			& \lesssim \rho^{-n}  \int_{B(x,2\rho)} |x-y| \,\left( \nabla_H g(y) +  \nabla_H g(x) \right) \,d\sigma(y)\\
			& \lesssim \rho\, \left( \cM_\sigma(\nabla_H  g)(x) +  \nabla_H g(x)\right),
		\end{align*}
		where $\cM_\sigma$ is the Hardy-Littlewood maximal function with respect to $\sigma$. The same estimate holds with $x'$ in place of $x$. Thus,
		\begin{align*}
			|S_\rho g(x) &- S_\rho g(x')| \\
			&\lesssim  \left( \cM_\sigma(\nabla_H  g)(x) +  \nabla_H g(x)+   \cM_\sigma(\nabla_H  g)(x') +  \nabla_H g(x')\right)\,\rho  + |g(x) -  g(x')|\\
			& \lesssim \left( \cM_\sigma(\nabla_H  g)(x) +  \nabla_H g(x)+   \cM_\sigma(\nabla_H  g)(x') +  \nabla_H g(x')\right) \,|x-x'|.
		\end{align*}
		
		In case that $|x-x'|< \rho$, we write
		\begin{align*}
			S_\rho g(x) - S_\rho g(x') &= \int_\pom (s_\rho(x,y) - s_\rho(x',y))\,g(y)\,d\sigma(y)\\
			& =
			\int_\pom (s_\rho(x,y) - s_\rho(x',y))\,(g(y)-g(x))\,d\sigma(y).
		\end{align*}
		Notice that the support of the integrand is contained in
		$$\supp(s_\rho(x,\cdot))\cup \supp(s_\rho(x',\cdot))\subset 
		\bar B(x,2\rho)\cup \bar B(x',2\rho)\subset \bar B(x,3\rho).$$
		Using also that ${\mathrm Lip}(s_\rho(\cdot,y))\lesssim \rho^{-n-1}$, we deduce
		\begin{align*}
			|S_\rho g(x) - S_\rho g(x')| & \leq 
			\rho\,\int_{\pom\cap \bar B(x,3\rho)} |s_\rho(x,y) - s_\rho(x',y)|\,\,\left( \nabla_H g(y) +  \nabla_H g(x) \right) \,d\sigma(y)\\
			&\lesssim \rho \, \frac{|x-x'|}{\rho^{n+1}}\int_{\pom\cap \bar B(x,3\rho)}\left( \nabla_H g(y) +  \nabla_H g(x) \right)  \,\,d\sigma(y)\\
			&\lesssim |x-x'| \,\left( \cM_\sigma(\nabla_H  g)(x) +  \nabla_H g(x)\right).
		\end{align*}
		Hence, $C(\cM_\sigma(\nabla_H  g) +  \nabla_H g)$ is an upper gradient of $g$ with $L^p(\sigma)$ norm bounded above by $C'\|\nabla_Hg\|_{L^p(\sigma)}$.
		This completes the proof of (a). 
		
		To prove (b), one readily checks that in the case that $g \in \Lip(\Sigma)$, the above estimates
		hold replacing $\nabla_H g$ and $\cM_\sigma(\nabla_H  g)$ by $\Lip(g)$.
	\end{proof}
	
	\vv
	
	We are now ready to prove Lemma \ref{lem:HMT vs Hajlasz}.
	\begin{proof}[Proof of Lemma \ref{lem:HMT vs Hajlasz}]
		If $f \in \Lip(\pom)$, then Lemma \ref{lem:HMT vs Hajlasz}  follows from  \cite[eq. (4.15)]{MT} and  \cite[Lemma 6.3]{MT},  where the implicit constants are independent of the Lipschitz constant of $f$.  Let us now assume that $f \in M^{1,p}(\pom)$ and let $\vphi \in C^\infty_c(\R^{n+1})$. By Lemma \ref{lemsr} (d) and (e),  H\"older's inequality,  \eqref{eq:M emb W} for Lipschitz functions,  and  Lemma \ref{lemlip} (a), we obtain 
		\begin{align*}
			\int_{\partial^* \om} f \,\partial_{t, k, j} \vphi \,d\sigma_*&= \lim_{\rho \to 0}\int_{\partial^* \om} S_\rho f \,\partial_{t, k, j} \vphi \,d\sigma_*=\lim_{\rho \to 0}\int_{\partial^* \om} \partial_{t, j, k} S_\rho f\, \vphi \,d\sigma_*\\
			& \leq \lim_{\rho \to 0} \|\partial_{t, j, k} S_\rho f\|_{L^p(\partial^* \om)} \, \| \vphi \|_{L^{p'}(\partial^* \om)} \\
			& \leq \lim_{\rho \to 0} \|\nabla_H S_\rho f\|_{L^p(\partial \om)} \, \| \vphi \|_{L^{p'}(\partial ^*\om)} \lesssim  \|\nabla_H  f\|_{L^p(\partial \om)} \, \| \vphi \|_{L^{p'}(\partial^* \om)},
		\end{align*}
		which, by the definition of $W^{1,p}(\partial^* \om)$,  concludes the proof of the lemma.
	\end{proof}
	
	\vv
	
	\subsection{The variational Neumann problem in bounded uniform domains with Ahlfors regular boundaries} 
	
	Let us first introduce the function spaces in the interior of the domain and the boundary which are necessary in order to solve the variational Neumann problem. 
	We consider the Sobolev space 
	$$
	W^{1,p}(\om):= \{ u \in L^2(\om): \nabla u \in L^p(\om)\},\quad p \in [1,\infty],
	$$ 
	with norm 
	$\|u \|_{W^{1,p}(\om)}:= \|u \|_{L^{p}(\om)}+\|\nabla u \|_{L^{p}(\om)}$.

	We say that a domain $\om\subset \R^{n+1}$ is a {\it $p$-Sobolev extension domain} if there exists a linear and bounded ({\it extension}) operator $\textup{Ext}_{\om \to \R^{n+1}}: W^{1,p}(\om) \to W^{1,p}(\R^{n+1})$ and a constant $\wt  C_E>0$ such that
	\begin{equation}\label{eq:extension*}
		\|\textup{Ext}_{\om \to \R^{n+1}}(u)\|_{W^{1,p}(\R^{n+1}) } \leq \wt  C_E\,\|u\|_{W^{1,p}(\om)}.
	\end{equation}
	By the work of Jones \cite{Jo}, we know that any uniform domain is a $p$-Sobolev extension domain for every $p\in[1,\infty]$.

	If $\mu$ is an Ahlfors $d$-regular measure in $\R^{n+1}$ with $F:=\supp \mu$,   then, for $s \in (0,1)$, we define the fractional Sobolev space 
	$$
	H^{s}(F):=\left\{ f \in L^2(F): \|f\|_{\dot H^{s}(F)}:= \int_F \,\,\int_F \frac{|f(x)-f(y)|^2}{|x-y|^{d-2s}} \,d\mu(x)\,d\mu(y)< \infty. \right\}
	$$
	equipped with norm $\|f\|_{H^{s}(F)}:= \|f\|_{L^2(F)}+ \|f\|_{\dot H^{s}(F)}$,  which is a Banach space.  Thus, we define via duality
	$$
	H^{-s}(F)=(H^{s}(F))^*.
	$$
	We also define by $H_0^{s}(F)$ the subspace of $H^{s}(F)$ of those functions satisfying $\int_{F} f \,d\mu=0$.

	Let $F \subset \R^{n+1}$ be a $d$-Ahlfors regular closed  set, with $s \in(0,n+1]$.  Then, by \cite[Theorem 1,  p.182]{JW},  it holds that there exists a linear and bounded ({\it trace}) operator 
	$$
	\textup{Tr}_{\R^{n+1} \to F} : W^{1,2}(\R^{n+1}) \to H^{\beta}(F)
	$$ such that, for every $u \in W^{1,2}(\R^{n+1})$,  
	$$
	\| \textup{Tr}_{\R^{n+1} \to F}  (u) \|_{H^{\beta}(F)} \lesssim \| u\|_{W^{1,2}(\R^{n+1})}.
	$$
	where $\beta=1-\frac{n+1-d}{2} >0$.
	
	Notice that if $\pom$ which is $n$-Ahlfors regular, then we may take $F=\pom$ and $\mu:=\HH^n|_\pom$, and get that $\beta=1/2$ and $\textup{Tr}_{\R^{n+1} \to \pom} : W^{1,2}(\R^{n+1}) \to H^{1/2}(\pom)$.  If $\om \subset \R^{n+1}$ is a corkscrew domain, then we may take $F=\overline{\om}$ and equip it with the Lebesgue measure $m|_\om$, which is $(n+1)$-Ahlfors regular (by the corkscrew condition). { In that case, we have  $d=n+1$ and so $\beta= 1$. Therefore, there exists a linear and bounded  trace operator $\textup{Tr}_{\R^{n+1} \to \overline{\om}} : W^{1,2}(\R^{n+1}) \to H^{1}(\overline{\om})$.  Note that since  $m|_\om(\partial \om)=0$, we have that $H^{1}(\overline \om)=H^{1}(\om)\equiv W^{1,2}(\om)$.  Consequently,  if  $\om$ is a uniform domain with $n$-Ahlfors regular boundary, 
		we can  define $\textup{Tr}_{\R^{n+1} \to \om} : W^{1,2}(\R^{n+1}) \to W^{1,2}( \om)$.  }
	
	Assuming still $\pom$ to be $n$-Ahlfors regular, by \cite[Theorem 3,  p.155]{JW},  there exists a linear and bounded ({\it extension}) operator $\textup{Ext}_{\pom \to \R^{n+1}}: H^{s}(\pom) \to W^{1,p}(\R^{n+1})$ such that, for every $f \in H^{1/2}(\pom)$,
	$$
	\| \textup{Ext}_{\pom \to \R^{n+1}} (f)\|_{W^{1,2}(\R^{n+1})}  \lesssim \| f \|_{H^{1/2}(\pom)}.
	$$

	\vv

	Combining all the results above one can show the following:
	\begin{theorem}\label{thm:trace-extension}
		If  $\om\subset \R^{n+1}$ is a uniform domain with $n$-Ahlfors regular boundary, there exists a bounded linear trace operator $\textup{Tr}_{\om \to \pom}: W^{1,2}(\om) \to H^{1/2}(\pom)$ and a bounded linear extension operator $\textup{Ext}_{\pom \to \om}: H^{1/2}(\pom) \to W^{1,2}(\om) $ such that 
		\begin{align}\label{eq:trace}
			\| \textup{Tr}_{\om \to \pom}  (u) \|_{H^{1/2}(\om)} &\leq C_T\, \| u\|_{W^{1,2}(\om)}\\
			\| \textup{Ext}_{\pom \to \om} (f)\|_{W^{1,2}(\om )}  &\leq  C_E\,  \| f \|_{H^{1/2}(\pom)}.\label{eq:extension}
		\end{align} 
		Moreover, $\textup{Ext}_{\pom \to \om} \circ \textup{Tr}_{\om \to \pom} = \textup{Id}$, the identity on $H^{1/2}(\pom)$.
	\end{theorem}   
	
	\begin{proof}
		This follows from \cite[Theorem 1,  p.208]{JW} and by defining $\textup{Ext}_{\pom \to \om} :=\textup{Tr}_{\R^{n+1} \to \om} \circ  \textup{Ext}_{\pom \to \R^{n+1}}$. 
	\end{proof}
	
	\vv
	Now,  by \cite[Lemma 7.2.1]{MMM},  if $\pom$ is $n$-Ahlfors regular,  we have that the following inclusions are  well-defined continuous with dense ranges if $n \geq 2$:
	\begin{align*}
		L^{\frac{2n}{n+1}}(\pom) \hookrightarrow H^{-1/2}(\pom) \qquad \textup{and} \qquad H^{1/2}(\pom)  \hookrightarrow L^{\frac{2n}{n-1}}(\pom). 
	\end{align*}
	Moreover,  when $n=1$,  if $\pom$ is compact,  the following inclusions are also  well-defined continuous with dense ranges for each $p \in (1,\infty)$:
	$$
	H^{1/2}(\pom) \hookrightarrow   L^{p}(\pom)    \hookrightarrow  H^{-1/2}(\pom).
	$$
	Let us also mention that the set of bounded Lipschitz functions $\Lip_b(\partial \Omega)$ equipped with the norm $\| \vphi \|_{\Lip_b(\partial \Omega)}:= \|\vphi \|_{\Lip_b(\partial \Omega)}  + \| \vphi\|_{L^\infty(\partial \Omega)}$ is a Banach space.  Moreover,  if $\pom$ is compact,  then the space $\Lip_0(\pom)$ endowed with the homogeneous norm $\|\vphi \|_{\Lip(\partial \Omega)} $  is a Banach space.   It is easy to see that for $p \geq 1$,  $\Lip_b(\pom) \hookrightarrow L^p(\pom)$ (resp. $\Lip_0(\pom)\hookrightarrow L_0^p(\pom)$ when $\pom$ is compact) is a well-defined continuous inclusion with dense range.  In fact, when $\pom$ is compact,  $\Lip(\pom)$  endowed with  $\| \cdot \|_{\Lip_b(\partial \Omega)}$ is continuously embedded in $L^p(\pom)$. 
	
	Therefore,  if $\pom$ is compact and $n$-Ahlfors regular,  for $n \geq 2$,  we have that 
	\begin{align}
		{\Lip}_0(\pom) \hookrightarrow L_0^{\frac{2n}{n+1}}(\pom) &\hookrightarrow H_0^{-1/2}(\pom) \label{eq:embed-fractional-}\\
		H^{1/2}(\pom)  &\hookrightarrow L^{\frac{2n}{n-1}}(\pom). \label{eq:embed-fractional+}
	\end{align}
	are well-defined continuous inclusions with dense ranges.

	\vv

	Let $\om \subset \R^{n+1}$ be a bounded uniform domain with $n$-Ahlfors regular boundary. We consider the Sobolev space with zero trace
	$$
	\widehat W^{1,2}(\om):=\left\{u \in W^{1,2}(\om) : \int_\pom  \textup{Tr}_{\om \to \pom}  (u) \,d\sigma=0 \right\}.
	$$
	By abusing notation, we will  write $u|_\pom$ instead of $ \textup{Tr}_{\om \to \pom}  (u)$.  Note that $\widehat W^{1,2}(\om)$ becomes a Hilbert space with the inner product 
	$$
	\langle u, v\rangle=\int_\om  u \, v + \int_\om \nabla u \cdot \nabla v.
	$$
	
	If we define the bilinear form associated with the operator \(L\) as
	\[
	B(u,v) := \int_\Omega A \nabla u\cdot\nabla  v,
	\]
	then by \eqref{eqelliptic1} and \eqref{eqelliptic2}, the bilinear form \(B\) becomes coercive and bounded on \(\widehat W^{1,2}(\om)\).

	Let $g \in L^{\frac{2n}{n+1}}(\partial \Omega)$ satisfy the ``compatibility” condition 
	$\int_{\partial \Omega} g = 0.$  Then, by \eqref{eq:trace}, \eqref{eq:embed-fractional-}, and \eqref{eq:embed-fractional+},   we find that
	\[
	\ell(u) :=  \int_{\partial \Omega} g \,  u \,d\sigma
	\]
	is a bounded linear functional on  $\widehat W^{1,2}(\om)$. Therefore, the Lax-Milgram theorem implies that there exists a unique \(u \in \widehat W^{1,2}(\om)\) such that \(B(u, v) = \ell(v)\) for all \(v \in  \widehat{W}^{1,2}(\Omega)\). 
	
	Observe now that any function \(v \in W^{1,2}(\Omega)\) can be written as the sum of a function in \(\widehat W^{1,2}(\om)\) and a constant \(c\). Indeed,
	\[
	v = \tilde{v} + c, \quad \text{where} \quad \tilde{v} = v - \frac{1}{\sigma(\partial \Omega)} \int_{\partial \Omega} v \quad \text{and} \quad c = \frac{1}{\sigma(\partial \Omega)} \int_{\partial \Omega} v.
	\]
	Notice that the condition $\int_\pom g \,d\sigma=0$ implies that \(\ell(c) = 0\). Then the identity \(B(u, \tilde{v}) = \ell(\tilde{v})\) yields
	\[
	\int_\Omega A\nabla u \cdot \nabla v = \int_{\partial \Omega} g \, v, \qquad \textup{for all}\,\, v \in W^{1,2}(\Omega),
	\]
	and so,  there exists a unique solution \(u\) in \( \widehat{W}^{1,2}(\Omega)\) of the variational Neumann problem \eqref{eqneumann1}.

	\vv

	\subsection{The Neumann function} 
	
	The following result is already known. 
	
	\begin{theorem}[Moser estimates]\label{teomoser}
		Let $\Omega\subset\R^{n+1}$ be a  bounded uniform  domain with Ahlfors regular boundary and assume that  $\LL \in \EE(\Omega)$.  Let also $\xi\in \pom$ and $0<r\leq \diam(\Omega)$.   Let $B=B(\xi,r)$ and $u\in W^{1,2}(2B\cap \Omega)$ be
		such that $\LL u=0$ in $2B$ with vanishing Neumann data on $2B\cap \pom$. Then 
		\begin{equation}\label{eq:moser1}
			{\rm osc}_B(u) \leq C\,\avint_{2B\cap\Omega} |u-m_B (u)|\,dm,
		\end{equation}
		where $m_B(u) = \avint_{B\cap\Omega} u\,dm$.
		Further, there exists some $\alpha>0$ such that for $0<\ve\leq 1$,
		\begin{equation}\label{eq:moser2}
			{\rm osc}_{\ve B}(u) \leq C\,\ve^\alpha \,{\rm osc}_B(u).
		\end{equation}
		The constants $C$ and $\alpha$ in \rf{eq:moser1} and \rf{eq:moser2} depend only on $n$, the ellipticity of $A$, the Ahlfors regularity of $\pom$, and  the uniformity  constant of $\Omega$.
	\end{theorem}

	The proof of \rf{eq:moser1} can be found in \cite{Kim} or \cite{HS},  the proof of \rf{eq:moser2} is in \cite{HS}. The  following result is shown in \cite{HS} and \cite[Theorem 2.29]{FL}.

	\begin{theorem}\label{teoneumann1}
		Let $\Omega\subset\R^{n+1}$ be a bounded uniform  domain with Ahlfors regular boundary and let $\LL \in \EE(\Omega)$.
		There exists a unique function $N:\overline\Omega\times \overline\Omega\to \R\cup\{+\infty\}$ (called the Neumann function) such that  the following holds:
		\begin{itemize}
			\item[(i)] $N(x,\cdot)\in W^{1,2}(\Omega\setminus\{x\}) \cap C(\overline\Omega\setminus\{x\})$,  for any fixed $x\in\Omega$, 
			\item[(ii)] $\int_\pom N(x,\xi)\,d\sigma(\xi)=0$, 
			\item[(iii)] For any $y\in\overline \Omega$ and $v\in W^{1,2}(\Omega)$,
			$$\int_\Omega A(x)\,\nabla_1N(x,y)\cdot \nabla v(x)\,dx = v(y) - \avint_\pom v\,d\sigma,$$
			where we identified $v$ on $\pom$ with the trace of $v$.
			\item[(iv)] For any $x,y\in\Omega$, $N^T(x,y) = N(y,x)$, where $N^T$ is the Neumann function for the adjoint operator  $\LL^*$ on $\Omega$.
			\item[(v)] For all $x,y\in\overline\Omega$,
			$$|N(x,y)|\leq \frac{C}{|x-y|^{n-1}}.$$
			\item[(vi)] There exists some $\alpha\in (0,1)$ such that for all $x,x'\in\overline\Omega$, $y\in\Omega$,
			$$|N(x,y) - N(x',y)| \leq \frac{C\,|x-x'|^\alpha}{|x-y|^{n-1+\alpha} + |x'-y|^{n-1+\alpha} }.$$
		\end{itemize}
		The constants $C$ in (v) and (vi) only depend on the uniformity constants of $\Omega$.
	\end{theorem}
	
	Recall that $\wh W^{1,2}(\Omega)$ denotes the subspace of the functions $u\in W^{1,2}(\Omega)$ such that the trace of $u$   on  $\pom$ satisfies (abusing notation)
	$\int_\pom u\,d\sigma=0.$
	
	\vv
	
	Integrating by parts, one derives the following result (see \cite{FL} and \cite{HS}):
	\begin{theorem}\label{teoneumann2}
		Under the assumptions of Theorem \ref{teoneumann1}, for any $g\in L^\infty_0(\pom)$, there exists a unique function  $u\in \wh W^{1,2}(\Omega)$ which solves the Neumann problem \rf{eqneumann1} with the following representation:
		$$u(x) = \int_\pom N(x,\xi)\,g(\xi)\,d\sigma(\xi),\quad\mbox{ for all $x\in\Omega$.}
		$$
	\end{theorem}
	
	
	\vv
	
	The following result will be useful when proving a localization result.  Although this is a standard calculation that follows from (iii) in Theorem \ref{teoneumann1}, we show the details for the reader's convenience.
	\begin{lemma}\label{propoloc0}
		Under the assumptions of Theorem \ref{teoneumann1}, let $\vphi\in C_c^\infty(\R^{n+1})$, let $g\in L^\infty_0(\pom)$, and  let  $u\in \widehat W^{1,2}(\Omega)$ solve the Neumann problem \rf{eqneumann1}.  For all $x\in \Omega$, we have
		\begin{align*}
			\vphi(x)\,u(x) - \avint_\pom \vphi\,u\,d\sigma &= \int_\pom N(x,\xi)\,\vphi(\xi)\,g(\xi)\,d\sigma(\xi) + \int_\Omega u(y)\,A(y)\,\nabla \vphi(y)\cdot \nabla_2 N(x,y)\,dy \\
			& \quad- \int_\Omega N(x,y)\,A(y)\nabla u(y)\cdot\nabla\vphi(y)\,dy.
		\end{align*}
	\end{lemma}

	\begin{proof}
		Set $v=\vphi\,u$. By   Theorem \ref{teoneumann1} (iii) applied to  $\LL^*$ and $N^T$, and (iv), we have  that
		$$v(x) - \avint_\pom v\,d\sigma= \int_\Omega A^T(y)\,\nabla_2N(x,y)\cdot \nabla v(y)\,dy =\int_\Omega A(y)\nabla v(y)\cdot\nabla_2N(x,y)\,dy.$$
		Writing $N^x = N(x,\cdot)$, we have
		\begin{align}\label{eqalk471}
			\int_\Omega A(y)&\nabla v(y)\cdot\nabla_2N(x,y)\,dy  =   \int_\Omega u\,A\nabla \vphi\cdot\nabla N^x\,dm + \int_\Omega \vphi\,A\nabla u\cdot\nabla N^x\,dm\\
			& = 
			\int_\Omega u\,A\nabla \vphi\cdot\nabla N^x\,dm + \int_\Omega A\nabla u\cdot\nabla (\vphi\,N^x)\,dm - \int_\Omega N^xA\nabla u\cdot\nabla\vphi\,dm.\nonumber
		\end{align}
		By \rf{eqweak93} with $\vphi$ replaced by $\vphi\,N^x$, Theorem \ref{teoneumann1} (i),  and a standard approximation argument, we infer that 
		$$\int_\Omega A\nabla u\cdot\nabla (\vphi\,N^x)\,dm = 
		\int_\pom N^x\,\vphi\,g\,d\sigma.$$
		The lemma readily  follows once we plug this identity into \rf{eqalk471}.
	\end{proof}

	\vvv
	
	
	\section{The Neumann problem and the rough Neumann problem}\label{sec-roughneumann}
	
	In this section we introduce several variants of the solvability of the Neumann problem which are required for the proof of our main theorem.
	First, recall that for $1<p<\infty$, we say that the {\it Neumann problem for $\LL$  
		is solvable in $L^p$} if the variational solution $u:\Omega\to\R$ of 
	\rf{eqneumann1}  satisfies
	\begin{equation}\label{eq:neu Lp}
		\|\wt \cN_\Omega(\nabla u)\|_{L^p(\pom)} \lesssim \|g\|_{L^{p}(\pom)} \quad \mbox{ for all $g\in  L^p(\pom) \cap L_0^{\frac{2n}{n+1}}(\pom)$.}
	\end{equation}
	Additionally, we say that the {\it  Neumann problem is solvable
		from the Lorentz space $L^{p,1}$ to $L^p$} if the variational solution \rf{eqneumann1}  satisfies
	\begin{equation}\label{eq:neu lp1}
		\|\wt \cN_\Omega(\nabla u)\|_{L^p(\pom)} \lesssim \|g\|_{L^{p,1}(\pom)}\quad \mbox{ for all $g\in  L^{p,1}(\pom) \cap L_0^{\frac{2n}{n+1}}(\pom)$.}
	\end{equation}
	To be brief,  we will write that $(N_{L^p})_\LL$ (or $(N_p)_\LL$) and $(N_{L^{p,1},L^p})_\LL$ are solvable, respectively. 
	
	For $\rho>0$, we say that the {\it $\rho$-smooth Neumann problem is solvable in $L^p$} if  the variational solution of 
	\rf{eqneumann1} satisfies
	\begin{equation}\label{eq:smooth neu lp}
		\|S_\rho(\partial_{t,j,k} u)\|_{L^p(\pom)}\lesssim \|g\|_{L^{p}(\pom)}\quad \mbox{ for all $g\in  L^p(\pom) \cap L_0^{\frac{2n}{n+1}}(\pom)$ and $1\leq j,k\leq
			n+1$.}
	\end{equation}
	On the other hand, we say that the {\it $\rho$-smooth Neumann problem is solvable from $L^{p,1}$ to $L^p$} if
	\begin{equation}\label{eq:smooth neu lp1}
		\|S_\rho(\partial_{t,j,k} u)\|_{L^p(\pom)}\lesssim \|g\|_{L^{p,1}(\pom)}\quad \mbox{ for all $g\in  L^p(\pom) \cap L_0^{\frac{2n}{n+1}}(\pom)$  and $1\leq j,k\leq
			n+1$.}
	\end{equation}
	We  will write that $(N_p(\rho))_\LL$ and $(N_{L^{p,1},L^p}(\rho))_\LL$ are solvable, respectively.

	\vv

	We say that the {\it rough Neumann problem is solvable for $\LL$ in $L^{p'}$} (and we write $(N^R_{L^{p'}})_\LL$ or $(N^R_{p'})_\LL$ is solvable) if, for every
	$1\leq j,k\leq n+1$, and every $g\in L^{p'}(\pom) \cap \Lip(\pom)$, the variational solution $u\in\wh W^{1,2}(\Omega)$ of 
	\begin{equation}\label{eqdualneumann}
		\left\{
		\begin{array}{ll}
			\LL u=0 & \quad \text{ in $\Omega$,}\\
			\partial_{\nu_A} u = \partial_{t,j,k} g & \quad \text{ in $\pom$}
		\end{array}
		\right.
	\end{equation}
	satisfies
	\begin{equation}\label{eq:rouneu Lp}
		\|\cN_\Omega(u)\|_{L^{p'}(\pom)} \lesssim \|g\|_{L^{p'}(\pom)}.
	\end{equation}
	On the other hand, we say that the {\it rough Neumann problem is solvable from $L^{p'}$ to $L^{p',\infty}$} (and we write 
	$(N^R_{L^{p'},L^{p',\infty}})_\LL$ is solvable) if 
	\begin{equation}\label{eq:rouneu Lpinfty}
		\|\cN_\Omega(u)\|_{L^{p',\infty}(\pom)} \lesssim \|g\|_{L^{p'}(\pom)},
	\end{equation}
	for all $j,k,g$ as above.

	\vv
	
	For $\rho>0$, we say that the {\it $\rho$-smooth rough Neumann problem is solvable in $L^{p'}$} (and we write $(N^R_{L^{p'}}(\rho))_\LL$ or $(N^R_{L^{p'}}(\rho))_\LL$ is solvable) if 
	for every function $g\in L^{p'}(\pom)$
	the variational 
	solution $u\in\wh W^{1,2}(\Omega)$ of
	\begin{equation}\label{eqdualneumannsmooth}
		\left\{
		\begin{array}{ll}
			\LL u=0 & \quad \text{ in $\Omega$,}\\
			\partial_{\nu_A} u = \partial_{t,j,k} S_\rho g & \quad \text{ in $\pom$}
		\end{array}
		\right.
	\end{equation}
	satisfies
	\begin{equation}\label{eq:smooth rouneu Lp}
		\|\cN_\Omega(u)\|_{L^{p'}(\pom)} \lesssim \|g\|_{L^{p'}(\pom)}.
	\end{equation}
	On the other hand, we say that the {\it $\rho$-smooth rough Neumann problem is solvable from $L^{p'}$ to $L^{p',\infty}$} (and we write 
	$(N^R_{L^{p'},L^{p',\infty}}(\rho))_\LL$ is solvable) if 
	\begin{equation}\label{eq:smooth rouneu Lpinfty}
		\|\cN_\Omega(u)\|_{L^{p',\infty}(\pom)} \lesssim \|g\|_{L^{p'}(\pom)},
	\end{equation}
	for all $j,k,g$ as above.

	\vv
	We denote by $C_\LL(N_p)$, $C_\LL(N^R_{p'})$, $C_\LL(N_{L^{p,1},L^p})$, and $C_\LL(N^R_{L^{p'},L^{p',\infty}})$ stand for the solvability constants
	of $(N_p)_\LL$, $(N^R_{p'})_\LL$, $(N_{L^{p,1},L^p})_\LL$, and $(N^R_{L^{p'},L^{p',\infty}})_\LL$, respectively. Their $\rho$-smooth versions are denoted by  $C_\LL(N_p(\rho))$, $C_\LL(N^R_{p'}(\rho))$, $C_\LL(N_{L^{p,1},L^p}(\rho))$, and $C_\LL(N^R_{L^{p'},L^{p',\infty}}(\rho))$.
	
	\vv
	The proposition below and its proof is partially inspired by the results in \cite[section 3]{AM}.
	\begin{propo}\label{propodual1}
		Let $\Omega$ be a bounded chord-arc domain and $\LL \in \EE(\om)$.  Suppose that $1 < p < \infty$ and that $(R_q)_\LL$ is solvable for some $q > p$. The following hold:
		\begin{itemize}
			\item[(a)] If $(N_p)_\LL$ (resp. $(N_{L^{p,1},L^p})_\LL$) is solvable, then $(N^R_{p'})_{\LL^*}$ (resp. $(N^R_{L^{p'},L^{p',\infty}})_{\LL^*}$) is solvable.
			\item[(b)] If $(N^R_{p'})_{\LL^*}$ (resp. $(N^R_{L^{p'},L^{p',\infty}})_{\LL^*}$) is solvable and we assume,  in addition,  that $\partial \Omega$ supports a weak $p$-Poincaré inequality, then  $(N_p)_\LL$ (resp. $(N_{L^{p,1},L^p})_\LL$) is also solvable.
		\end{itemize}
		Further, 
		$$C_\LL(N_p)\approx C_{\LL^*}(N^R_{p'})\qquad \text{and}\qquad C_\LL(N_{L^{p,1},L^p})\approx C_{\LL^*}(N^R_{L^{p'},L^{p',\infty}}).$$
	\end{propo}
	
	{Remark that in (b) one asks $\partial \Omega$ to support a weak $p$-Poincaré inequality, not a weak $p'$-Poincaré inequality.}
	
	\begin{proof}
		We will only prove the  statements involving the Lorentz spaces,  as the others are similar and slightly easier.
		Suppose first that $(N_{L^{p,1},L^p})_\LL$ is solvable with constant $C_4$ and let us check that then $(N^R_{L^{p'},L^{p',\infty}})_{\LL^*}$ is solvable with constant $\lesssim C_4$.
		Let $j,k,u,g$ be as in \rf{eqdualneumann} with $\LL^* $ and $A^T$  in place of $\LL$ and $A$ and consider an arbitrary function $\vphi\in \Lip(\pom) \cap L^{p,1}(\pom)$.
		Let $v:\Omega\to\R$ be the solution of the Neumann problem with boundary data $\vphi-m_{\sigma,\pom}\vphi$. Using that $m_{\sigma,\pom}(u)=0$, we get
		\begin{align*}
			\int_\pom u\,\vphi\,d\sigma &= \int_\pom u\,(\vphi-m_{\sigma,\pom}\vphi)\,d\sigma =
			\int_\pom u\,\partial_{\nu_A} v\,d\sigma = \int_\Omega A\nabla u\cdot \nabla v\,dm\\
			& = \int_\pom \partial_{\nu_{A^T}} u\, v\,d\sigma
			= \int_\pom \partial_{t,j,k} g\, v\,d\sigma = \int_\pom g\, \partial_{t,k,j}  v\,d\sigma.
		\end{align*} 
		As usual, abusing notation we denoted by $u$ the  trace of $u$ on $\pom$.
		Therefore,  by Lemma \ref{lem:HMT vs Hajlasz} and a standard argument originating from \cite{Kenig-Pipher} (see for instance \cite[Eq. (4.3.11)]{HMT}), we have that
		\begin{align}\label{eq:tang.der.bound n.t. maximal}
			\left|\int_\pom u\,\vphi\,d\sigma\right| & \lesssim \|g\|_{L^{p'}(\pom)}\, \|\partial_{t,k,j}  v\|_{L^p(\pom)} \lesssim\|g\|_{L^{p'}(\pom)}\, \|\nabla_H v\|_{L^p(\pom)}\\
			&\lesssim \|g\|_{L^{p'}(\pom)}\, \|\cN_\Omega(\nabla  v)\|_{L^p(\pom)} \notag \\
			& \leq C_4\|g\|_{L^{p'}(\pom)}\, \|\partial_{\nu_A} v\|_{L^{p,1}(\pom)} = C_4\|g\|_{L^{p'}(\pom)}\, \|\vphi\|_{L^{p,1}(\pom)} \notag
		\end{align}
		By Lemma \ref{lem:denseLorentz} and  \rf{eqclau29}, we infer that
		$$\|u\|_{L^{p',\infty}(\pom)} \lesssim C_4 \|g\|_{L^{p'}(\pom)}.$$
		Next,  since $(R_q)_{\LL}$  is solvable for $\LL$ some $q>p$,  by Theorem \cite[Theorem 1.6]{MT}, we deduce that  the Dirichlet problem is solvable for $\LL^*$ in $L^{q'}(\pom)$. Hence, by interpolation,  the Dirichlet problem is also solvable from $L^{p',\infty}(\pom)$ to $L^{p',\infty}(\pom)$ and so
		$$\|\cN_\Omega(u)\|_{L^{p',\infty}(\pom)} \lesssim \|u\|_{L^{p',\infty}(\pom)} \lesssim C_4 \|g\|_{L^{p'}(\pom)},$$
		concluding that $(N^R_{L^{p'},L^{p',\infty}})_{\LL^*}$ is solvable.

		Suppose now that $(N^R_{L^{p'},L^{p',\infty}})_{\LL^*}$ is solvable with constant $C_*$ and let us prove that  $(N_{L^{p,1},L^p})_\LL$ is solvable with constant $\lesssim C_*$.
		Let $u$ be solution of the Neumann problem in $\Omega$ with boundary data $\partial_{\nu_A} u=f$ for some $f\in L^{p,1}(\pom) \cap L_0^{2n/(n+1)}(\pom)$.  Since $\pom$ satisfies a weak $p$-Poincar\'e inequality,  then,   by \cite[Theorem 1.6]{GMT} and \cite[Lemma 1.3]{MT},  it also holds that $(R_p)_\LL$\footnote{We mean the regularity problem in terms of the tangential derivatives. Notice however that, since $\pom$ supports a Poincar\'e inequality, this is equivalent to the regularity problem in the Haj\l asz space $\dot M^{1,p}(\pom)$. See \cite[Lemma 1.3]{MT} for more details.} is solvable. Thus, it suffices to show that for all $1\leq j,k\leq n+1$,
		\begin{equation}\label{eqdual1}
			\|\partial_{t,j,k} u\|_{L^p(\pom)} \leq C_*\,\|f\|_{L^{p,1}(\pom)},
		\end{equation}
		where again $u$ on $\pom$ should be understood as a trace.
		To this end, consider an arbitrary function  $\psi\in L^{p'}(\pom) \cap \Lip(\pom)$ and let $w:\Omega\to\R$ be the solution of the
		rough Neumann problem for $\LL^*$ with data $\partial_{\nu_{A^T}} w= \partial_{t,k,j}\psi$.
		Then we have
		\begin{align*}
			\int_\pom \partial_{t,j,k} u\,\psi\,d\sigma &= \int_\pom  u\,\partial_{t,k,j}\psi\,d\sigma   
			=\int_\pom u\, \partial_{\nu_{A^T}} w\,d\sigma = \int_\pom \partial_{\nu_A} u\,  w\,d\sigma 
			\\
			&\leq \| \partial_{\nu_A} u\|_{L^{p,1}(\Omega)}\,\|w\|_{L^{p',\infty}(\pom)}\leq C_*\,\| f\|_{L^{p,1}(\Omega)}\,\|\psi\|_{L^{p'}(\pom)},
		\end{align*}
		which gives \rf{eqdual1} by density and duality. 
	\end{proof}
	
	\vv
	\begin{propo}
		Let $1<p<\infty$, $\rho\in (0,\diam(\pom))$, and let $\Omega$ be a bounded chord-arc domain.
		Suppose that the Dirichlet problem for $\LL^*$ is solvable in $L^{p'}$ and in $L^{p',\infty}$.
		Then $(N_p(\rho))_\LL$ is solvable if and only if $(N^R_{p'}(\rho))_{\LL^*}$ is solvable. Also, 
		$(N_{L^{p,1},L^p}(\rho))_{\LL}$ is solvable if and only if $(N^R_{L^{p'},L^{p',\infty}}(\rho))_{\LL^*}$ is solvable.
		Further, 
		$$C_\LL(N_p(\rho))\approx C_{\LL^*}(N^R_{p'}(\rho) )\qquad \text{and}\qquad C_\LL(N_{L^{p,1},L^p}(\rho))\approx C_{\LL^*}(N^R_{L^{p'},L^{p',\infty}}(\rho)).$$
	\end{propo}
	
	\begin{proof}
		The arguments are very similar to the ones for Proposition \ref{propodual1}.  However, we show the details for completeness.
		We will only prove the second statement since the first one is similar.

		Suppose  that $(N_{L^{p,1},L^p}(\rho))_\LL$ is solvable with constant $C_4$.  We will check that $(N^R_{L^{p'},L^{p',\infty}}(\rho))_{\LL^*}$ is solvable with constant $\lesssim C_4$ as well.
		Let $j,k,u,g,\rho$ be as in \rf{eqdualneumannsmooth} with $\LL^*$ and $A^T$ in place of $\LL$ and $A$ and consider an arbitrary function $\vphi\in L^{p,1}(\pom)\cap \Lip(\pom)$.
		Let $v:\Omega\to\R$ be the solution of the Neumann problem for $\LL$ with boundary data $\vphi-m_{\sigma,\pom}\vphi$. Then, we have that
		\begin{align*}
			\int_\pom u\,\vphi\,d\sigma &= \int_\pom u\,(\vphi-m_{\sigma,\pom}\vphi)\,d\sigma =
			\int_\pom u\,\partial_{\nu_A} v\,d\sigma \\
			& = \int_\pom \partial_{\nu_{A^T}} u\, v\,d\sigma
			= \int_\pom \partial_{t,j,k} S_\rho(g)\, v\,d\sigma = \int_\pom g\, S_\rho(\partial_{t,k,j}  v)\,d\sigma.
		\end{align*} 
		Therefore,
		\begin{align*}
			\left|\int_\pom u\,\vphi\,d\sigma\right| & \lesssim \|g\|_{L^{p'}(\pom)}\, \|S_\rho(\partial_{t,k,j}  v)\|_{L^p(\pom)}
			\\
			& \leq C_4\|g\|_{L^{p'}(\pom)}\, \|\partial_{\nu_A} v\|_{L^{p,1}(\pom)} = C_4\|g\|_{L^{p'}(\pom)}\, \|\vphi\|_{L^{p,1}(\pom)},
		\end{align*}
		where $C_4$ is the constant of $(N_{L^{p,1},L^p})$.   By density and duality,  we deduce that
		$$\|u\|_{L^{p',\infty}(\pom)} \lesssim C_4 \|g\|_{L^{p'}(\pom)}.$$
		Using  that the  Dirichlet problem is solvable from $L^{p',\infty}(\pom)$ to $L^{p',\infty}(\pom)$,  we infer that
		$$\|\cN_\Omega(u)\|_{L^{p',\infty}(\pom)} \lesssim \|u\|_{L^{p',\infty}(\pom)} \lesssim C_4 \|g\|_{L^{p'}(\pom)}.$$
		That is, $(N^R_{L^{p'},L^{p',\infty}}(\rho))_{\LL^*}$ is solvable with constant $\lesssim C_4$.
		
		Suppose now that $(N^R_{L^{p'},L^{p',\infty}}(\rho))_{\LL^*}$ is solvable with constant $C_*$ and let us prove that then $(N_{L^{p,1},L^p}(\rho))_\LL$ is solvable with constant $\lesssim C_*$.
		Let $u$ be solution of $(N_{L^{p,1},L^p}(\rho))_\LL$ with boundary data $\partial_{\nu_A} u=f$, for some $f\in L^{p,1}(\pom) \cap \Lip_0(\pom)$.  We aim to show that for all $1\leq j,k\leq n+1$,
		\begin{equation}\label{eqdual1'}
			\|S_\rho(\partial_{t,j,k} u)\|_{L^p(\pom)} \lesssim C_*\,\|f\|_{L^{p,1}(\pom)}.
		\end{equation}
		To this end, consider an arbitrary Lipschitz function $\psi\in L^{p'}(\pom)$ and let $w:\Omega\to\R$ be the solution of the
		rough Neumann problem for $\LL^*$ with $\partial_{\nu_{A^T}} w= \partial_{t,k,j}S_\rho(\psi)$.
		Then we have
		\begin{align*}
			\int_\pom S_\rho(\partial_{t,j,k} u)\,\psi\,d\sigma &= \int_\pom  u\,\partial_{t,k,j}S_\rho(\psi)\,d\sigma   
			=\int_\pom u\, \partial_{\nu_{A^T}} w\,d\sigma = \int_\pom \partial_{\nu_A} u\,  w\,d\sigma 
			\\
			&\leq \| \partial_{\nu_A} u\|_{L^{p,1}(\Omega)}\,\|w\|_{L^{p',\infty}(\pom)}\leq C_*\,\| f\|_{L^{p,1}(\Omega)}\,\|\psi\|_{L^{p'}(\pom)},
		\end{align*}
		which gives \rf{eqdual1'} by  duality.
	\end{proof}

	\vv

	\begin{lemma}\label{lemfinit}
		Let $\Omega\subset\R^{n+1}$ be a bounded chord-arc domain and let $\rho\in (0,\diam(\pom)]$.  Then $(N_{p}^R(\rho))_\LL$ is solvable for any  $1<p<\infty$  with constant
		$$ C_\LL(N^R_{p}(\rho))\lesssim \frac{\diam(\pom)}\rho.$$
	\end{lemma}
	
	\begin{proof}
		Let $u:\Omega\to\R$ be the variational solution of  \rf{eqdualneumannsmooth}. Then
		$$u(x) = \int_\pom N(x,y)\,\partial_{t,j,k} (S_\rho g)(y)\,d\sigma(y).$$
		For brevity, we write
		$$u= N_\sigma(\partial_{t,j,k} (S_\rho g)).$$
		From the properties of the kernel of $S_\rho$ in Lemma \ref{lemsr} and Schur's lemma, it follows that $\partial_{t,j,k} S_\rho$ is an operator bounded in $L^{p}(\sigma)$ with norm $\lesssim \rho^{-1}$, for $1\leq p\leq\infty$. On the other hand, since
		$$\int_{\pom}N(x,y)\,d\sigma(y)\lesssim \int_{\pom}\frac1{|x-y|^{n-1}}\,d\sigma(y)\lesssim \diam(\pom)$$
		and the same estimate holds for $N(y,x)$, by Schur's criterion again, we deduce that the operator $N_\sigma$ is bounded in $L^{p}(\sigma)$ with norm
		$\lesssim \diam(\pom)$.
		Therefore,
		$$\|u\|_{L^{p}(\sigma)} = \|N_\sigma(\partial_{t,j,k} (S_\rho g))\|_{L^{p}(\sigma) }\lesssim \diam(\pom)\,
		\|\partial_{t,j,k} (S_\rho g))\|_{L^{p}(\sigma)} \lesssim \frac{\diam(\pom)}\rho\,\|g\|_{L^{p}(\pom)}.$$ 
	\end{proof}
	\vv

	\begin{lemma}\label{lemuniformrho}
		Let $\Omega\subset\R^{n+1}$ be a bounded chord-arc domain and suppose that $(D_{p})_{\LL}$ is solvable for some $p\in(1,\infty)$.   Then $(N^R_{p})_\LL$ is solvable if and only if $(N^R_{p}(\rho))_{\LL}$ is solvable uniformly on $0<\rho\leq\diam(\pom)$.
	\end{lemma}
	
	\begin{proof}
		Suppose that  $(N^R_{p})$ is solvable and let $0<\rho\leq\diam(\pom)$.
		For $g\in L^{p}(\pom)$, let $u_\rho:\Omega\to\R$ be the variational solution of \rf{eqdualneumannsmooth}.
		Then, by the solvability of $(N^R_{p})$ and the $L^{p}(\sigma)$ boundedness of $S_\rho$ with norm $1$,
		$$\|\cN_\Omega(u_\rho)\|_{L^{p}(\pom)} \lesssim \|S_\rho g\|_{L^{p}(\pom)}\leq\| g\|_{L^{p}(\pom)}
		.$$
		So $(N^R_{p}(\rho))_{\LL}$ is solvable uniformly on $0<\rho\leq\diam(\pom)$.
		
		Conversely, suppose that $(N^R_{p}(\rho))_{\LL}$ is solvable uniformly on $0<\rho\leq\diam(\pom)$.
		For $g\in L^{p}(\pom)\cap \Lip(\pom)$,  let $u$ and $u_\rho:\Omega\to\R$ be the respective variational solutions of 
		\rf{eqdualneumann} and \rf{eqdualneumannsmooth}. 
		We claim that $u$ and $u_\rho$ are continuous in $\overline \Omega$ and
		$u_\rho$ converges to $u$ in $L^{p}(\sigma)$ as $\rho\to0$. Together with the $L^{p}$ solvability of the Dirichlet problem, this implies that 
		$$\|\cN_\Omega(u)\|_{L^{p}(\pom)}\lesssim\|u\|_{L^{p}(\sigma)} = \lim_{\rho\to 0}\|u_\rho\|_{L^{p}(\sigma)}\leq C\,\|g\|_{L^p(\sigma)},$$
		which completes the proof of the lemma, modulo our claim.
		
		To prove the claim, observe first that, from the fact that $g$ is Lipschitz, by Lemma \ref{lemlip}, it follows that
		$S_\rho g$ is also Lipschitz uniformly on $\rho$. Hence, $\partial_{t,j,k} g$ and  $\partial_{t,j,k} S_\rho g$ are 
		in $L^\infty(\sigma)$ uniformly in $\rho$. Arguing as in the proof of Lemma \ref{lemfinit},  it follows that
		$N_\sigma$ is bounded in 
		$L^{\infty}(\sigma)$ (with norm depending on $\diam(\pom)$), and so $u = N_\sigma(\partial_{t,j,k} g)$ and $u_\rho 
		= N_\sigma(\partial_{t,j,k} S_\rho g)$ are uniformly in $L^\infty(\sigma)$. By the H\"older continuity of the Neumann function away from the diagonal given by Theorem  \ref{teoneumann1} (vi) and the local integrability implied by (v) in the same theorem, using standard arguments, it follows that 
		in fact, $u$ and $u_\rho$ are continuous in $\overline\Omega$.

		By the dominated convergence theorem, since $u_\rho$ and $u$ are uniformly in $L^\infty(\sigma)$, to prove the convergence of $u_\rho$ to $u$ in $L^{p}(\sigma)$, it suffices to check that $u_\rho$ converges to $u$ pointwise in $\pom$.
		First we check the weak convergence of $\partial_{t,j,k} S_\rho g$ to $\partial_{t,j,k} g$ in $L^{q}(\sigma)$, for all $q\in (1,\infty)$. Indeed, if $\vphi$ is a  $C^1$ function in a neighborhood of $\pom$, then
		\begin{equation}\label{eqconv61}
			\int_\pom \partial_{t,j,k}(S_\rho g)\,\vphi\,d\sigma = \int_\pom S_\rho(g)\,\partial_{t,k,j}\vphi\,d\sigma\to 
			\int_\pom g\,\partial_{t,k,j}\vphi\,d\sigma=\int_\pom \partial_{t,j,k}g\,\vphi\,d\sigma \quad \mbox{ as $\rho\to0$,}
		\end{equation}
		because $S_\rho(g)$ converges to $g$ in $L^q(\sigma)$. Since the $L^q(\sigma)$ norms of $\partial_{t,j,k}(S_\rho g)$ and 
		$\partial_{t,j,k}g$ are bounded uniformly in $L^q(\sigma)$, \rf{eqconv61} also holds for any $\vphi\in L^{q'}(\sigma)$, by density.
		That is, $\partial_{t,j,k} S_\rho g$ converges to $\partial_{t,j,k} g$ weakly in $L^{q}(\sigma)$ as $\rho\to0$.
		
		Finally, from the fact that $N^x\equiv N(x,\cdot)\in L^q(\sigma)$ for $1<q<n/(n-1)$, we deduce that, for any $x\in\overline\Omega$,
		$$
		N_\sigma(\partial_{t,j,k} (S_\rho g))(x) = \int_\pom \partial_{t,j,k}S_\rho(g)\,N^x\,d\sigma \to  \int_\pom \partial_{t,j,k} g\,N^x\,d\sigma=
		N_\sigma(\partial_{t,j,k} g)(x),\,\,\textup{as}\,\,\rho \to 0.
		$$
		This finishes the proof of the claim.
	\end{proof}

	\vvv

	
	\section{The localization lemmas}
	
	\begin{theorem}[Poincar\'e inequality]\label{teotrace**}
		Let \(\Omega \subset \mathbb{R}^{n+1}\) be a  uniform domain and let  \(B := B(x_0, R)\) be a ball of radius \(R > 0\) centered at \(x_0 \in \partial \Omega\).  There exists a constant \(C > 2\), depending only on the uniformity constants of \(\Omega\),  such that if  \(u \in W^{1,p}(C B \cap \Omega)\), for \(1 < p <n+1\),  then
		\begin{equation}\label{eq:Poincare-interior}
			\bigg(\avint_{B\cap\Omega} |u-u_B|^{p+\ve_p}\,dm\bigg)^{\frac1{p+\ve_p}}\\
			\lesssim   R\left(\avint_{CB\cap\Omega}|\nabla u|^p \,dm\right)^{\frac1p},
		\end{equation}
		where  \(u_B := \avint_{B \cap \Omega} u \, dm\) and $\ve_{p}$ is some positive constant depending only on $p$ and $n$.
	\end{theorem}

	\begin{proof} 
		Since \eqref{eq:Poincare-interior} is scale invariant, we may assume that  \(R = 1\). If \(T_\Delta \subset \Omega\) is the Carleson box associated with the surface ball \(\Delta := B \cap \partial \Omega\) (see \cite[Eq. 3.59]{HM} for its definition), then by \cite[Eq. (3.60)]{HM}, we have \(\frac{5}{4} B \cap \Omega \subset T_\Delta \subset CB \cap \Omega\) for some large constant \(C > 2\) depending only on the uniformity constants of \(\Omega\).   Furthermore, according to \cite[Lemma 3.61]{HM}, \(T_\Delta\) is a uniform domain, where the uniformity constants depend only on those of \(\Omega\) and are uniform with respect to \(B\). It is evident that \(\text{diam}(T_\Delta) \approx 1\).

		By \cite[Theorem 2]{Jo},  any uniform domain is a Sobolev extension domain for the homogeneous Sobolev space $\dot W^{1,p}$ for any  $p \in (1, \infty)$.  Therefore,  since $u \in \dot W^{1,p}(T_\Delta)$,  we may extend it to a function $ \wt u\in \dot W^{1,p}(\R^n)$, satisfying   $\| \nabla \wt u \|_{L^{p}(\R^{n+1})} \lesssim \|\nabla  u\|_{L^{p}(T_\Delta)}$.   If  $q:= \frac{(n+1)\,p}{n+1-p}$, then by the Sobolev-Poincar\'e inequality,  we have that 
		\begin{align*}
			\bigg(\int_{B\cap\Omega} |u-u_B|^{q}\,dm\bigg)^{1/q} &\leq 2\, \bigg( \int_{CB} |\wt u-\wt u_{CB}|^{q}\,dm\bigg)^{1/q}
			\lesssim   \bigg(\int_{CB}|\nabla \wt u|^p \,dm\bigg)^{1/p}\\
			& \lesssim   \bigg(\int_{T_\Delta}|\nabla u|^p \,dm\bigg)^{1/p} \leq \bigg(\int_{CB\cap \om}|\nabla u|^p \,dm\bigg)^{1/p}.
		\end{align*}
		It is trivial to see that \eqref{eq:Poincare-interior} follows by rescaling the  inequalities above.
	\end{proof}

	\vv
	
	We define the {\it truncated  non-tangential maximal operator}  $\cN_{\Omega,r}$ by
	$$\cN_{\Omega,r}v(\xi) =  \sup_{x\in\gamma_\Omega(\xi)\cap B(\xi,r)} |v(x)|,\quad \mbox{ for $\xi\in\pom$.}$$

	\begin{lemma}\label{lemfutur}
		Let \(\Omega \subset \mathbb{R}^{n+1}\) be an open set  with a $n$-Ahlfors regular boundary satisfying the interior corkscrew condition. Then, for all \(p\) and \(q\) such that \(1 \leq q \leq p < q\left(1 + \frac{1}{n}\right)\), any function \(v: \Omega \to \mathbb{R}\), and any ball \(B\) centered at \(\partial \Omega\), we have
		\begin{equation}\label{eqRHI}
			\left(\avint_{B \cap \Omega} |v|^p \, dm \right)^{1/p} \lesssim_{p,q} \left(\avint_{2B \cap \partial \Omega} |\mathcal{N}_{\Omega,4 r(B)} v|^q \, d\sigma \right)^{1/q},
		\end{equation}
		assuming the aperture of the cones associated with \(\mathcal{N}_{\Omega}\) is large enough, depending only on \(n\).
	\end{lemma}
	
	\begin{proof}
		Consider first the case $q=1$.
		Denote by $W_{B}$ the family of (Euclidean) Whitney cubes of $\Omega$ that intersect $B \cap \Omega$.
		Adjusting suitably the parameters of the construction of the Whitney cubes, we can ensure that 
		the cubes $P\in W_{B}$ satisfy
		$$\diam(P)\leq r(B)/2\quad \text{ and }\quad P\subset 2B\cap \Omega.$$
		Denote $\Delta=B\cap\pom$. Then we write
		\begin{align*}
			\left(\avint_{B \cap \Omega} |v|^p\,dm\right)^{1/p}\!\!\sigma(\Delta) & \lesssim \bigg( \sum_{P\in W_B} m_P(|v|^p)\,\ell(P)^{n+1}\bigg)^{1/p}\frac{\sigma(\Delta)}{r(B)^{(n+1)/p}}\\
			& \lesssim 
			\bigg( \sum_{P\in W_B} \inf_{\xi\in \wh P} \cN_{\Omega,4 r(B)}(v)(\xi)^p\,\ell(P)^{n+1}\bigg)^{1/p}\frac{\sigma(\Delta)}{r(B)^{(n+1)/p}},
		\end{align*}
		where $m_P$ denotes the mean with respect to Lebesgue measure and $\wh P\in\DD_\pom$ is a boundary cube associated with $P$ such that $\dist(\wh P,P)\approx \ell(P)$ and $\ell(\wh P) = \ell(P)$.
		Observe that we can assume that $\wh P\subset 2B$.
		Thus, taking into account that $\frac{n+1}{p} - n>0$ (because $p<1+\frac1n$ as $q=1$), 
		\begin{align}\label{eqfutur}
			\left(\frac1{r(B)^{n+1}}\int_{\Omega\cap B}
			|v|^p\,dm\right)^{1/p}\!\!\sigma(\Delta) & \lesssim 
			\sum_{P\in W_B}
			\inf_{\xi\in \wh P} \cN_{\Omega,4 r(B)}(v)(\xi)\,\ell(P)^{(n+1)/p}\frac{\sigma(\Delta)}{r(B)^{(n+1)/p}}\\
			& \lesssim
			\sum_{\substack{\wh P\in \DD_\pom:\wh P\subset 2B\\ \diam(P)\leq r(B)}} \int_{\wh P} \cN_{\Omega,4r(B)}(v)\,d\sigma \,\frac{\ell(P)^{\frac{n+1}{p} - n}}{r(B)^{\frac{n+1}{p} - n}}\nonumber\\
			&\lesssim \int_{2\Delta}  \cN_{\Omega,4 r(B)}(v)\,d\sigma.\nonumber
		\end{align}
		
		In the case $q>1$, we apply \rf{eqfutur} to the function $u=|v|^q$. Then, for $1\leq s< 1+\frac1n$, we get
		$$\bigg(\avint_{B \cap \Omega} |v|^{sq}\,dm\bigg)^{1/s}\lesssim \avint_{2B \cap \pom} |\cN_{\Omega,4 r(B)}(|v|^q)\,d\sigma =  \avint_{2B \cap \pom} |\cN_{\Omega,4 r(B)}(v)|^q\,d\sigma.$$
		Writing $p=sq$, we are done.
	\end{proof}
	\vv

	We will need the following localization lemma.
	
	\begin{lemma}[Smooth Localization Lemma in $L^{p,1}$]\label{lemlocLp**}
		Let $\Omega\subset\R^{n+1}$ be a chord-arc domain and $1<p\leq2$, $0<\rho\leq \diam(\pom)$. Suppose that $(N_{L^{p,1},L^p}(\rho))_\LL$ 
		and $(D_{L^{p',\infty}})_{\LL^*}$ are solvable in $\Omega$. Let $R\geq 4\rho$,  
		let $B=B(x_0,R)$ be a ball centered in $\pom$, and denote $\Delta = \pom\cap B$ and $\Omega_R =\Omega\cap A(x_0,2R,C_5R)$, for some constant $C_5>4$ depending only on the chord-arc character of $\Omega$.
		Let $u\in W^{1,2}(\Omega)$ be $\LL$-harmonic in $\Omega\cap B(x_0,C_5R)$ and such that $\partial_{\nu_A} u\in L^{p,1}(3\Delta)$ (with
		$\partial_{\nu_A} u$ understood in the weak sense).
		Then, for $1\leq j,k\leq n+1$, we have
		$$\avint_\Delta |S_\rho(\partial_{t,j,k} u)|^p \,d\sigma \lesssim \frac1{\sigma(\Delta)}\,\|\partial_{\nu_A} u\|_{L^{p,1}( 3\Delta)}^p + \bigg(\avint_{\Omega_R}
		|\nabla u |^{2}\,dm\bigg)^{p/2},$$
		with the implicit constant bounded by $1+C_\LL(N_{L^{p,1},L^p}(\rho))$ times some constant depending only on $n$, $p$, the chord-arc character of $\Omega$, and the $(D_{L^{p',\infty}})_{\LL^*}$ constant.
	\end{lemma}
	
	\begin{proof}
		Let $\vphi$ be a smooth function which equals $1$ on $2.5B$ and vanishes on  $\R^{n+1} \setminus 3B$. Denote $u_B = \avint_{\Omega_R} u$ and $v=u-u_B$,
		so that $S_\rho(\partial_{t,j,k} u)=S_\rho(\partial_{t,j,k} (\vphi\,v))$ on $\Delta=B\cap \pom$.
		Hence it suffices to estimate $\|S_\rho(\partial_{t,j,k} (\vphi\,v))\|_{L^p(\sigma|_\Delta)}$ (recall that $\supp(s_\rho(x,\cdot)\subset \bar B(x,2\rho)$ and $\rho\leq R/4$)).
		To this end, we consider an arbitrary smooth function $\psi$ supported on $\Delta$, and for every $j,k$ we write
		\begin{equation}\label{eqwyu7**a}
			\langle S_\rho(\partial_{t,j,k} (\vphi\,v)),\,\psi\rangle_\sigma = \langle \vphi\,v,\,\partial_{t,k,j} S_\rho\psi\rangle_\sigma = 
			\langle \vphi\,v- c_{\vphi v},\,\partial_{t,k,j} S_\rho\psi\rangle_\sigma,
		\end{equation}
		where 
		$$c_{\vphi v} =\avint_\pom \vphi\,v\,d\sigma$$
		and where we took into account that $\int_\pom\partial_{t,k,j} S_\rho\psi\,d\sigma=0$. Then,  by Lemma \ref{propoloc0}, for all $x\in \Omega$, we have
		\begin{align*}
			\vphi(x)\,v(x) - c_{\vphi v} &= \int_\pom N^x\,\vphi\,\partial_{\nu_A}v\,d\sigma + \int_\Omega v\,A\nabla \vphi\cdot \nabla N^x\,dm - \int_\Omega N^x\,A\nabla v\cdot\nabla\vphi\,dm.
		\end{align*}
		In other words, if we denote 
		\begin{equation}\label{defNmu}
			N_\mu (f)(x) = \int N(x,y)\,f(y)\,d\mu(y),
			\quad (\nabla_2 N)_\mu (F)(x) = \int_\Omega  \nabla N^x(y)\cdot F(y)\,d\mu(y),
		\end{equation}
		we have
		$$\vphi(x)\,v(x) - c_{\vphi v} = N_\sigma(\vphi\,\partial_{\nu_A} v)(x) - N_m\big(A\nabla v\cdot \nabla\vphi\big)(x)
		+ (\nabla_2 N)_m (v\,A\nabla \vphi)(x).$$   
		Plugging this identity into \rf{eqwyu7**a}, we obtain
		\begin{align}\label{eqgh34**a}
			\langle S_\rho(\partial_{t,j,k} (\vphi\,v)),\,\psi\rangle_\sigma & = \langle N_\sigma(\vphi\,\partial_{\nu_A} v),\,\partial_{t,k,j} S_\rho\psi\rangle_\sigma -
			\langle N_m\big(A\nabla v\cdot \nabla\vphi\big),\,\partial_{t,k,j} S_\rho\psi\rangle_\sigma \\
			&\quad +
			\langle (\nabla_2 N)_m (v\,A\nabla \vphi),\,\partial_{t,k,j} S_\rho\psi\rangle_\sigma
			\nonumber\\
			& =
			\langle \vphi\,\partial_{\nu_A} v,\,N_\sigma^*(\partial_{t,k,j} S_\rho\psi)\rangle_\sigma + \langle A \nabla v\cdot \nabla\vphi,\,N_\sigma^*(\partial_{t,k,j} S_\rho\psi)\rangle_m\nonumber\\
			&\quad + \langle  v\,A\nabla \vphi,\,(\nabla_2 N)_\sigma^*(\partial_{t,k,j} S_\rho\psi)\rangle_m,\nonumber
		\end{align}
		where
		\begin{equation}\label{defNmu*}
			N_\mu^* (f)(y) = \int N(x,y)\,f(x)\,d\mu(x),
			\quad (\nabla_2 N)^*_\mu (f)(y) = \int_\Omega  \nabla_2N(x,y) \,f(x)\,d\mu(x).
		\end{equation}
		Notice that $(\nabla_2 N)_\sigma^*(\partial_{t,k,j} S_\rho\psi)$ is a vector field.

		Next we claim that
		\begin{equation}\label{claim*1**}
			\|N_\sigma^*(\partial_{t,k,j} S_\rho\psi)\|_{L^{p',\infty}(\sigma)}\lesssim C_p\,\| \psi\|_{L^{p'}(\sigma)},
		\end{equation}
		where we wrote $C_p:=C_\LL(N_{L^{p,1},L^p}(\rho))$ to shorten notation.
		Indeed, consider a Lipschitz function $\theta\in L^{p,1}(\sigma)$. By the solvability of $(N_{L^{p,1},L^p}(\rho))_\LL$, we have
		\begin{align*}
			\big|\langle \theta,\,N_\sigma^*(\partial_{t,k,j} S_\rho\psi)\rangle_\sigma\big| & = \big|\langle S_\rho(\partial_{t,j,k}(N_\sigma(\theta))),\, \psi\rangle_\sigma\big|\\
			&\leq \|S_\rho(\partial_{t,j,k}(N_\sigma(\theta)))\|_{L^p(\sigma)} \,\|\psi\|_{L^{p'}(\sigma)} \leq C_p
			\|\theta\|_{L^{p,1}(\sigma)}\, \|\psi\|_{L^{p'}(\sigma)},
		\end{align*}
		which proves \eqref{claim*1**}.
		Consequently, we can estimate the first term in \rf{eqgh34**a} as follows
		\begin{align*}
			\big|
			\langle \vphi\,\partial_{\nu_A} v,\,N_\sigma^*(\partial_{t,k,j} S_\rho\psi)\rangle_\sigma\big| 
			& \lesssim  
			\|\vphi\,\partial_{\nu_A} v\|_{L^{p,1}(\pom)}\,\|N_\sigma^*(\partial_{t,k,j} S_\rho\psi)\|_{L^{p',\infty}(\pom)}\\
			&\lesssim
			C_p\,\|\partial_{\nu_A} v\|_{L^{p,1}(\sigma|_{3\Delta})}\,\|\psi\|_{L^{p'}(\sigma)}.
		\end{align*}
		
		Now we turn our attention to the second term on the right hand side of \rf{eqgh34**a}. First we apply H\"older's inequality to obtain
		$$\big|\langle A \nabla v\cdot \nabla\vphi,\,N_\sigma^*(\partial_{t,k,j} S_\rho\psi)\rangle_m\big| 
		\leq \|A \nabla v\cdot \nabla\vphi\|_{L^1} \,\|N_\sigma^*(\partial_{t,k,j} S_\rho\psi)\|_{L^{\infty}(\Omega\cap A(x_0,2.5R,3R))}.$$
		Notice now that 
		$$\supp(\partial_{t,k,j} S_\rho\psi) \subset \bar B(x_0,R+2\rho)\subset \bar B(x_0,\tfrac32R).$$
		So
		$N_\sigma^*(\partial_{t,k,j} S_\rho\psi)$ is an $\LL^*$-harmonic function which is continuous in $\overline\Omega$ and 
		such that its conormal derivative vanishes on $\pom\setminus  \bar B(x_0,\tfrac32R)$. Then, by Moser type estimates and Lemma \ref{lemfutur}, it follows that
		\begin{align*}
			\|N_\sigma^*(\partial_{t,k,j} S_\rho\psi)\|_{L^{\infty}(\Omega\cap A(x_0,2.5R,3R))} & \lesssim \avint_{\Omega\cap A(x_0,2R,4R)} |N_\sigma^*(\partial_{t,k,j} S_\rho\psi)|\,dm\\
			& \lesssim \avint_{6\Delta} \cN_\Omega(N_\sigma^*(\partial_{t,k,j} S_\rho\psi))\,d\sigma.
		\end{align*}

		By Kolmogorov's inequality (recall that $p'\geq2$), the $L^{p',\infty}$ solvability of the Dirichlet problem for $\LL^*$,  and \rf{claim*1**}:
		\begin{align*}
			\|N_\sigma^*(\partial_{t,k,j} S_\rho\psi)&\|_{L^{\infty}(\Omega\cap A(x_0,2.5R,3R))}
			\lesssim 
			\sigma(\Delta)^{-1/p'}\, \|\cN_\Omega(N_\sigma^*(\partial_{t,k,j} S_\rho\psi))\|_{L^{p',\infty}(\sigma)} \\
			&\qquad\lesssim  \sigma(\Delta)^{-1/p'}\, \|N_\sigma^*(\partial_{t,k,j} S_\rho\psi)\|_{L^{p',\infty}(\sigma)}\lesssim C_p\,\sigma(\Delta)^{-1/p'}\,\| \psi\|_{L^{p'}(\sigma)}.
		\end{align*}
		Therefore,
		\begin{align*}
			\big|\langle A \nabla v\cdot \nabla\vphi,\,N_\sigma^*(\partial_{t,k,j}  S_\rho\psi)\rangle_m\big| &
			\lesssim C_p\, \sigma(\Delta)^{-1/p'}\|A \nabla v\cdot \nabla\vphi\|_{L^1} \,\| \psi\|_{L^{p'}(\sigma)}\\
			& \lesssim C_p\,\sigma(\Delta)^{-1/p'} R^{-1}\|\nabla v\|_{L^1(A(x_0,2.5R,3R)\cap\Omega)} \,\| \psi\|_{L^{p'}(\sigma)}.
		\end{align*}
		
		Finally we deal with the third term on the right hand side of \rf{eqgh34**a}. First we write
		\begin{align}\label{eqal5bx}
			\big|\langle  v\,A\nabla \vphi,(\nabla_2 N)_\sigma^*(\partial_{t,k,j} S_\rho\psi)\rangle_m\big|  & \leq \bigg(\int_\Omega |v\,A\nabla \vphi|^2\,dy\bigg)^{1/2}\|(\nabla_2 N)_\sigma^*(\partial_{t,k,j} S_\rho\psi)\|_{L^2(A(x_0,2.5R,3R)}.
		\end{align}
		Notice that 
		\begin{equation}\label{eqint581}
			\int_\Omega |v\,A\nabla \vphi|^2\,dy\lesssim \frac1{R^2}\,\int_{A(x_0,2.5R,3R)} |v|^2\,dy.
		\end{equation}
		Regarding the last term on the right hand side of \rf{eqal5bx}, observe that $N_\sigma^*(\partial_{t,k,j} S_\rho\psi)$ is $\LL^*$-harmonic in $\Omega$ and it has vanishing conormal derivative in $\pom\setminus \bar B(x_0,\tfrac32R)$. Thus, 
		by the Caccioppoli inequality and Lemma \ref{lemfutur}, for some $q\in (1,2)$ (depending on $n$),
		\begin{align*}
			\|\nabla N_\sigma^*(\partial_{t,k,j} S_\rho\psi)\|_{L^2(A(x_0,2.5R,3R)} &\lesssim \frac1R\,\|N_\sigma^*(\partial_{t,k,j} S_\rho\psi)\|_{L^2(A(x_0,2R,4R))}\\ & 
			\lesssim R^{\frac{n-1}2-\frac nq}\,\|\cN_\Omega(N_\sigma^*(\partial_{t,k,j} S_\rho\psi))\|_{L^q(8\Delta)}.
		\end{align*}
		Then, by Kolmogorov's inequality, the $L^{p',\infty}$ solvability of the Dirichlet problem, and \rf{claim*1**},
		\begin{align*}
			\|\nabla N_\sigma^*(\partial_{t,k,j} S_\rho\psi)\|_{L^2(A(x_0,2.5R,3R)} &\lesssim R^{\frac{n-1}2-\frac n{p'}}\,\|\cN_\Omega(N_\sigma^*(\partial_{t,k,j} S_\rho\psi))\|_{L^{p',\infty}(8\Delta)}\\
			& \lesssim  
			R^{\frac{n-1}2-\frac n{p'}}\,\|N_\sigma^*(\partial_{t,k,j} S_\rho\psi)\|_{L^{p',\infty}(8\Delta)}\\
			& \leq C_p\,R^{\frac{n-1}2-\frac n{p'}}\,\|\psi\|_{L^{p'}(\sigma)}.
		\end{align*}
		Plugging this estimate and \rf{eqint581} into \rf{eqal5bx}, we obtain
		$$\big|\langle  v\,A\nabla \vphi,\,(\nabla_2 N)_\sigma^*(\partial_{t,k,j} S_\rho\psi)\rangle_m\big|
		\lesssim 
		C_p\, R^{\frac{n-3}2-\frac n{p'}}\,\bigg(\int_{A(x_0,2.5R,3R)} |v|^2\,dy\bigg)^{1/2}\,\|\psi\|_{L^{p'}(\sigma)}.$$
		
		Gathering the estimates obtained above, we derive
		\begin{align*}
			\big|\langle S_\rho(\partial_{t,j,k} (\vphi\,v)),\,\psi\rangle_\sigma\big| 
			& \lesssim C_p\,\Big(\|\partial_{\nu_A} v\|_{L^{p,1}(\sigma|_{3\Delta})}
			+ \sigma(\Delta)^{-1/p'} R^{-1}\|\nabla v\|_{L^1(A(x_0,2.5R,3R)\cap\Omega)} \\
			&\quad + R^{\frac{n-3}2-\frac n{p'}}\,\| v\|_{L^2(A(x_0,2.5R,3R)\cap\Omega)}\Big)\,\|\psi\|_{L^{p'}(\sigma)}.
		\end{align*}
		Therefore,
		\begin{align*}
			\| S_\rho(\partial_{t,j,k} (\vphi\,v))\|_{L^p(\sigma)}
			& \lesssim C_p\,\Big(\|\partial_{\nu_A} v\|_{L^{p,1}(\sigma|_{3\Delta})}
			+ \sigma(\Delta)^{-1/p'} R^{-1}\|\nabla v\|_{L^1(A(x_0,2.5R,3R)\cap\Omega)} \\
			&\quad + R^{\frac{n-3}2-\frac n{p'}}\,\| v\|_{L^2(A(x_0,2.5R,3R)\cap\Omega)}\Big).
		\end{align*}
		Then, denoting $\sigma_\Delta = \frac1{\sigma(\Delta)}\sigma$, we get
		\begin{align*}\| S_\rho(\partial_{t,j,k} (\vphi\,u))\|_{L^p(\sigma_\Delta)} &\lesssim C_p\,
			\bigg(\|\chi_{3\Delta}\,\partial_{\nu_A} u\|_{L^{p,1}(\sigma_\Delta)} + \avint_{A(x_0,2.5R,3R)\cap\Omega}|\nabla u|\,dm \\
			&\quad + \frac1R
			\bigg(\avint_{A(x_0,2.5R,3R)\cap\Omega}|u-u_B|^2\,dm\bigg)^{1/2}.
		\end{align*}
		By the Poincar\'e inequality in Theorem \ref{teotrace**}, we have\footnote{In fact, we can obtain
			$$
			\frac1R \bigg(  \avint_{A(x_0,2.5R,3R)\cap\Omega}|u-u_B|^2\,dm    \bigg)^{1/2}\lesssim \bigg(\avint_{\Omega_R}|\nabla u|^{q_n}\,dm\bigg)^{1/{q_n}}, 
			$$
			with $q_n= \frac{2n+2}{n+3}$, so that at the end we get 
			$$\| S_\rho(\partial_{t,j,k} (\vphi\,u))\|_{L^p(\sigma_\Delta)} \lesssim C_p\,
			\bigg(\|\chi_{3\Delta}\,\partial_{\nu_A} u\|_{L^{p,1}(\sigma_\Delta)} + \bigg(\avint_{\Omega_R}|\nabla u|^{q_n}\,dm\bigg)^{1/q_n}\bigg).$$
		}
		$$\frac1R\bigg(\avint_{A(x_0,2.5R,3R)\cap\Omega}|u-u_B|^2\,dm\bigg)^{1/2}\lesssim \bigg(\avint_{\Omega_R}|\nabla u|^{2}\,dm\bigg)^{1/2}, 
		$$
		so that at the end we get 
		$$\| S_\rho(\partial_{t,j,k} (\vphi\,u))\|_{L^p(\sigma_\Delta)} \lesssim C_p\,
		\bigg(\|\chi_{3\Delta}\,\partial_{\nu_A} u\|_{L^{p,1}(\sigma_\Delta)} + \bigg(\avint_{\Omega_R}|\nabla u|^{2}\,dm\bigg)^{1/2}\bigg).$$
	\end{proof}

	\vv
	
	We remark that by quite similar, but somewhat simpler arguments, we could get the following more classical localization result. Since this will not be used in this paper, we skip the detailed proof.

	\begin{lemma}[Localization Lemma]\label{lemlocLp}
		Let $\Omega\subset\R^{n+1}$ be a chord-arc domain and $1<p\leq2$. Suppose that the Neumann problem for $\Omega$ is solvable in $L^p$ and the
		Dirichlet problem for $\Omega$ is solvable in $L^{p'}$.
		Let $g\in L^p(\sigma)$ and let $u$ be the solution of the Neumann problem with boundary data $g$.
		Let $B=B(x_0,R)$ be a ball centered in $\pom$, $\Delta = \pom\cap B$, and $\Omega_R =\Omega\cap A(x_0,2R,C_5R)$, where $C_5>4$ is some constant depending just on the chord-arc character of $\Omega$. Then, for $1\leq j,k\leq n+1$, we have
		$$\avint_\Delta |\partial_{t,j,k} u|^p \,d\sigma \lesssim \avint_{3\Delta} |\partial_{\nu_A} u|^p\,d\sigma + \bigg(\avint_{\Omega_R}
		|\nabla u |^2\,dm\bigg)^{p/2}.$$
	\end{lemma}

	\vvv

	
	\section{Proof of the main theorem}
	
	
	This section is devoted to the proof of Theorem \ref{teobigpi} by means of a suitable good $\lambda$ inequality.
	To this end, we need some auxiliary lemmas. The first one is the following.
	
	\begin{lemma}\label{lemaux2}
		Let $\Omega$ be a domain with $n$-Ahlfors regular boundary, $B_0$ a ball centered at  $\pom$, and $\mu$ a Borel measure in $ B_0\cap\Omega$ such that
		\begin{equation}\label{eqtyp4500}
			\mu(B(x,r))\leq \bar C_0\,r^n\quad \mbox{ for all $x\in \R^{n+1}$ and $r>0$}
		\end{equation}
		and
		\begin{equation}\label{eqtyp4501}
			\mu(B(x,r))\geq \bar C_0^{-1}\,r^n\quad \mbox{ for all $x\in \supp\mu$ and $0<r\leq \delta_\Omega(x)$}
		\end{equation}
		Then, for any Borel function $u:\Omega\to\R$ such that $u\in L^1_{loc}(\mu)$ and any Borel function $\vphi:\supp\mu\to \R$,
		\begin{equation}\label{eqtyp450}
			\int |u\,\vphi|\,d\mu \lesssim \int_{2B_0}\cN_{\Omega,4r(B_0)}(u)\,\cM_\mu(\vphi)\,d\sigma.
		\end{equation}
		Also,
		\begin{equation}\label{eqtyp45}
			\int |u|\,d\mu \lesssim \int_{2B_0}\cN_{\Omega,4r(B_0)}(u)\,d\sigma,
		\end{equation}
		assuming in both estimates the aperture  of the cones associated with $\cN_{\Omega,4r(B_0)}$ to be large enough (depending only on $n$).
		The implicit constant above depends only on $n$, $\bar C_0$, and the Ahlfors regularity of $\pom$.
	\end{lemma}
	
	In the lemma $\cM_\mu$ stands for the non-centered maximal Hardy-Littlewood operator, with the supremum taken with respect to balls centered in $\supp\mu$.
	
	\begin{proof}
		Notice that \rf{eqtyp45} follows from \rf{eqtyp450} setting $\vphi=1$.
		
		To prove \rf{eqtyp450}, let $E=\supp\mu$, consider a decomposition of $\Omega$ into Whitney cubes as in Section \ref{secwhitney}, and  denote by $\WW_0$ the family of the Whitney cubes that intersect $B_0\cap E$. Reducing the size of the Whitney cubes if necessary, we can assume that
		$P\subset 1.5 B_0$ for each $P\in \WW_0$.
		By monotone convergence and the inner regularity of $\mu$,  we can also assume that 
		$\dist(E,\pom)>0$, which implies that the family $\WW_0$ is finite. 
		
		By the lower Ahlfors regularity of $\pom$ and the properties of Whitney cubes, we can choose positive constants $C_6$ (depending on $n$) and
		$\tau$ (depending on $n$ and the Ahlfors regularity of $\pom$) such that
		for each $P\in\WW_0$, the ball $\wt B(P) :=B(x_P,C_6\ell(P))$ (where $x_P$ is the center of $P$) satisfies 
		$3P\subset \wt B(P)$ and 
		\begin{equation}\label{eqsig840}
			\sigma(\wt B(P)\cap B_0\cap\pom) \geq \tau\,\ell(P)^n
		\end{equation}
		(in particular, this implies that $\wt B(P)\cap B_0\cap\pom\neq\varnothing$).
		For later reference, notice that the growth conditions on $\mu$ ensures that
		\begin{equation}\label{eqsig841}
			\mu(\wt B(P))\lesssim \mu(3P) \leq C_0'\,\ell(P)^n,
		\end{equation}
		with $C_0'$ depending only on $\bar C_0$ and $n$.
		
		We claim that for each \( P \in \WW_0 \), we can choose a Borel subset \( F_P \subset \wt B(P) \cap 2B_0 \cap \partial \Omega \) such that there exist constants \( c_2 > 0 \) and \( A > 1 \), depending on \( n \), the Ahlfors regularity of \( \partial \Omega \), and \( \bar{C}_0 \), for which the following holds:
		\begin{itemize}
			\item[(a)] $c_2\mu(3P)\leq \sigma(F_P) \leq \mu(3P)$, and
			\item[(b)] $\sum_{P\in\WW_0}\chi_{F_P}\leq (A+1)\,\chi_{2B_0}.$
		\end{itemize}
		
		Assume the claim for the moment, and let us see how the lemma follows.  From the fact that $F_P\subset \wt B(P)\cap 2B_0\cap\pom$, we have
		$$P\subset \gamma_\Omega(\xi)\cap 2B_0 \quad\mbox{ for $\xi\in F_P$,}$$
		where $\gamma_\Omega(\xi)$ is a non-tangential cone associated with $\cN_\Omega$ with vertex at $\xi$, with aperture  large enough. Thus
		$$u(x)\leq \inf_{\xi\in F_P} \cN_{\Omega,4r(B_0)}(u)(\xi)\quad\mbox{ for all $x\in P\cap E$.}$$
		From this fact, the properties (a), (b) claimed above, and \rf{eqsig841}, we deduce 
		\begin{align*}
			\int |u\,\vphi|\,d\mu & = \sum_{P\in\WW_0}\int_{P} |u\,\vphi|\,d\mu 
			\leq \sum_{P\in\WW_0}\inf_{\xi\in F_P} \cN_{\Omega,4r(B_0)}(u)(\xi)\,\int_{\wt B(P)}|\vphi|\,d\mu\\
			&  \lesssim \sum_{P\in\WW_0}\inf_{\xi\in F_P} \big( \cN_{\Omega,4r(B_0)}(u)(\xi)\,\cM_\mu(\vphi)(\xi)\big)\,\mu(3P)\\
			& \lesssim
			\sum_{P\in\WW_0}\int_{F_P} \cN_{\Omega,4r(B_0)}(u)\,{M}_\mu(\vphi)\,d\sigma \lesssim A\int_{2B_0} \cN_{\Omega,4r(B_0)}(u)\,{\cM}_\mu(\vphi)\,d\sigma,
		\end{align*}
		which proves \rf{eqtyp450}.
		\vv
		
		Next we prove the claim. To this end, since $\WW_0$ is a finite family, we can order it so that
		$\WW_0=\{P_1,P_2,\ldots P_{N}\}$, with $\ell(P_i)\leq \ell(P_{i+1})$ for all $i=1,\ldots,N-1$.
		We will construct the sets $F_i\equiv F_{P_i}$ inductively, 
		checking that for each $i=1,\ldots,N$ it holds
		\begin{equation}\label{eqind497}
			F_i\subset \wt B(P_i)\cap 2B_0\cap\pom =:S_i,
		\end{equation}
		and moreover,
		\begin{equation}\label{eqind498}
			c_2\,\mu(3P_i)\leq \sigma(F_i) \leq \mu(3P_i)
		\end{equation}
		and
		\begin{equation}\label{eqind499}
			\sum_{k=1}^i\chi_{F_k}\leq (A+1)\,\chi_{2B_0},
		\end{equation}
		with $c_2=\tau/(2C_0')$ and some sufficiently large $A>1$.
		
		To start, we choose an arbitrary Borel set $F_1\subset S_1= \wt B(P_1)\cap 2B_0\cap\pom$ satisfying \rf{eqind498}. The existence of
		$F_1$ is ensured by the fact that, by \rf{eqsig840} and 
		\rf{eqsig841}, for each $P_i$ it holds
		\begin{equation}\label{eqsigh23}
			\sigma(S_i) \geq \tau\,\ell(P_i)^n \geq \tau\,\mu(3P_i)/C_0'.
		\end{equation}
		Recall that $S_i=\wt B(P_i)\cap B_0\cap\pom$.
		Obviously, \rf{eqind499} also holds.
		
		Suppose now that we have already constructed sets $F_1,\ldots,F_i$ satisfying 
		\rf{eqind497}, \rf{eqind498}, and \rf{eqind499}, and let us construct $F_{i+1}$ for $i\leq N$.
		For any $\lambda>0$, by Chebyshev's inequality,  we have
		\begin{align*}
			T:=\sigma\Big(\Big\{x\in S_{i+1}:\sum_{1\leq k\leq i} \chi_{F_k}(x)>\lambda\Big\}\Big)
			\leq \frac1\lambda\sum_{1\leq k\leq i} \sigma(F_k \cap S_{i+1}).
		\end{align*}
		Notice now that if $F_k \cap S_{i+1}\neq\varnothing$, then $\wt B(P_k)\cap \wt B(P_{i+1})\neq\varnothing$, and then
		$3P_k\subset \wt B(P_k)\subset 3 \wt B(P_{i+1})$, since $r(\wt B(P_k))\leq r(\wt B(P_{i+1}))$ for $k<i+1$. So, if we denote by $I_{i+1}$ the subset of indices
		$k$ with $1\leq k\leq i$ such that $F_k \cap S_{i+1}\neq\varnothing$, using the upper estimate in \rf{eqind498}, we get
		$$T
		\leq \frac1\lambda\sum_{k\in I_{i+1}} \sigma(F_k \cap S_{i+1}) \leq
		\frac1\lambda\sum_{k\in I_{i+1}}  \mu(3P_k)\leq \frac C\lambda  \mu(3 \wt B(P_{i+1}))
		.$$
		By the upper growth condition on $\mu$, we deduce
		$$T\leq \frac{c\,\bar C_0}\lambda\,\ell(P_{i+1})^n,
		$$
		with $\bar C_0$ as in \rf{eqtyp4500}.
		
		So choosing $A=\lambda=\frac{2c\bar C_0}\tau$, by \rf{eqsig840}, we derive
		$$\sigma\Big(\Big\{x\in S_{i+1}:\sum_{1\leq k\leq i} \chi_{F_k}(x)>A\Big\}\Big) \leq
		\frac\tau2\,\ell(P_{i+1})^n\leq \frac12\,\sigma(S_{i+1}).$$
		Then we can take a subset 
		$$F_{i+1}\subset \Big\{x\in S_{i+1}:\sum_{1\leq k\leq i} \chi_{F_k}(x)\leq A\Big\}$$
		satisfying
		$$\sigma(F_{i+1}) =\min\big(\mu(3P_{i+1}),\,\tfrac12\sigma(S_{i+1})\big).$$
		So either $\sigma(F_{i+1}) =\mu(3P_{i+1})$ or 
		$$\sigma(F_{i+1}) =\frac12\sigma(S_{i+1}) \geq \frac{C_0'^{-1}\tau}2\,\mu(3P_{i+1}),
		$$
		using \rf{eqsigh23} in the last inequality. So in any case $F_{i+1}$ satisfies the required properties 
		\rf{eqind497}, \rf{eqind498}, \rf{eqind499}, and then the claim follows.
	\end{proof}
	\vv
	
	The second lemma that we need is an immediate consequence of the assumptions in Theorem~\ref{teobigpi}:
	
	\begin{lemma}\label{lemg2q}
		Under the assumptions of Theorem \ref{teobigpi}, let $K\geq10$ and $Q\in\DD_\pom$ and denote $\wt \Omega=U_{x_Q,10C_5K\ell(Q)}$, where $C_5>4$ depends on the chord-arc character of $\Omega$, as in Lemma \ref{lemlocLp**}.
		Then 
		$$\sigma(10C_5KQ\setminus \partial\wt\Omega)\lesssim \ve\,\sigma(4KQ)$$
		and
		there exists a compact subset $G_{2Q}\subset 2Q\cap\wpom$ such that
		$$\sigma(2Q\setminus G_{2Q})\lesssim K^n\ve\,\sigma(2Q).$$
	\end{lemma}
	
	Remark that, by definition, $2Q\subset 4KQ\subset\pom$.
	Below, for the application of this lemma, we will choose $\ve$ small enough so that $K^n\ve\ll1$.
	
	\vvv

	\vvv
	\begin{proof}[\bf Proof of Theorem \ref{teobigpi}]
		It is enough to show that, for $1<p<q$, $(N_{L^{p'},L^{p',\infty}}^R(\rho))_{\LL^*}$ is solvable uniformly on $\rho\in (0,\diam(\pom))$. Indeed, then by interpolation, for any $1<p<q$, $(N_{L^{p'}}^R(\rho))_{\LL^*}$ is solvable uniformly on $\rho\in (0,\diam(\pom))$, which implies the solvability of $(N_{L^{p'}}^R)_{\LL^*}$ and of $(N_{L^p})_\LL$, by Lemma \ref{lemuniformrho} (notice that the solvability of
		$(R_q)_{\LL}$ for some $q>p$ implies the solvability of $(D_{q'})_{\LL^*}$ and thus of $(D_{p'})_{\LL^*}$. See \cite{MT}).

		{We will prove that $(N_{L^{p'},L^{p',\infty}}^R(\rho))_{\LL^*}$ is solvable with some finite constant $C_*:=C_*(\rho)$ (i.e.,  $C_*$ depends on $\rho$; see Lemma \ref{lemfinit}), with $ C_*(\rho)$ bounded uniformly on $\rho$. To this end, we will prove
			an estimate of the form $C_*(\rho)\leq C + \delta\,C_*(\rho)$, for some fixed small $\delta>0$.
			by means of a good $\lambda$ inequality. }
		By the solvability of the Dirichlet problem for $\LL^*$ in $L^{p',\infty}$ (which follows from the one of $(D_{L^{q'}})_{\LL^*}$ by interpolation), it suffices to prove that
		\begin{equation}\label{eqproof1}
			\|u\|_{L^{p',\infty}(\pom)} \lesssim  C_*(\rho)\|g\|_{L^{p'}(\pom)},
		\end{equation}
		for $u:\Omega\to\R$ and a Lipschitz function $g\in L^{p'}(\pom)$ satisfying \rf{eqdualneumannsmooth}, for any fixed $1\leq j,k\leq n+1$ and $0<\rho\leq \diam(\pom)$, and $C_*(\rho)$ as above.
		To shorten notation, we write $\partial_t=\partial_{t,j,k}$.
		Notice that the function $u$ can be written as follows:
		\begin{equation}\label{eqproof2}
			u(x) = \int_\pom N(y,x)\,\partial_t S_\rho g(y) \,d\sigma(y) = - \langle\partial_tN_x^T,\, S_\rho g\rangle_\sigma,
		\end{equation}
		where $N_x^T(y) = N(y,x)$, with $N$ equal to the Neumann function for $\LL$ in $\Omega$.

		To prove \rf{eqproof1}, for any $\lambda>0$, let 
		\begin{equation}\label{eqdefvlam}
			V_\lambda = \{x\in \pom:\cM_\sigma u(x)>\lambda\},
		\end{equation}
		where $\cM_\sigma$ is the non-centered Hardy-Littlewood operator with respect to balls centered at $\pom$.
		For some fixed constant $A>2$ to be chosen below, we will estimate $\sigma(V_{A\lambda})$.
		To this end, we consider a partition of $V_\lambda$ into Whitney cubes from $\DD_\pom$, and we denote by $\WW_\lambda$ this family of Whitney cubes. 
		We choose the parameters in the Whitney decomposition so that $10Q\subset V_\lambda$ for all $Q\in\WW_\lambda$. 
		We denote
		$$E_Q = Q\cap V_{A\lambda}\quad\mbox{ for $Q\in\WW_\lambda$.}$$
		Clearly, 
		$$V_{A\lambda} = \bigcup_{Q\in\WW_\lambda} E_Q \subset V_\lambda.$$
		
		Fix a cube $Q\in\DD_\pom$ and consider the ball $B(x_Q,2K\ell(Q))$ and the associated domain
		$$\wt\Omega\equiv U_{x_Q,10C_5K\ell(Q)}.$$
		Also, let $G_{2Q}$ be as in Lemma \ref{lemg2q}.

		From the properties of the Whitney decomposition, we know that $CQ\cap V_\lambda^c\neq\varnothing$ for some fixed constant $C>1$. So by the definition of $V_\lambda$  we deduce 
		that
		$$
		\avint_{G_{2Q}}|u|\,d\sigma\leq 2
		\avint_{2Q}|u|\,d\sigma \leq C_7 \lambda\quad\mbox{ for all $Q\in\WW_\lambda$,}$$
		for some fixed constant $C_7$.
		Analogously, for any ball $B$ centered in $\pom$ intersecting $Q\in\WW_\lambda$ such that $r(B)\geq \ell(Q)/4$,
		$$ \avint_{B}|u|\,d\sigma \leq C_8 \lambda.$$
		Consequently,  assuming $A>C_8$,  if $x\in E_Q$, it holds that  $\cM_\sigma u(x) >A\lambda$,  and so there exists some ball $B_x$ centered at $\pom$ such that $r(B_x)\leq  \ell(Q)/4$ and $x\in B_x$, satisfying
		$$ \avint_{B_x}|u|\,d\sigma > A \lambda.$$
		Thus,
		$$ \avint_{B_x}|u - m_{\sigma,G_{2Q}}(u)|\,d\sigma > (A -C_7) \lambda \geq \lambda,$$
		where we wrote $m_{\sigma,G_{2Q}}(u) = \avint_{G_{2Q}}u\,d\sigma$ to shorten notation and we assumed $A\geq C_7+1$.
		Therefore, 
		$$\cM_\sigma \big(\chi_{2Q}(u-m_{\sigma,G_{2Q}}(u))\big)(x)>\lambda\quad\mbox{ for all $x\in E_Q$ with $Q\in\WW_\lambda$.}$$
		By the weak $(1,1)$ boundedness of $\cM_\sigma$, it follows that
		$$\sigma(E_Q) \leq \sigma(\{x\in Q:\cM_\sigma(\chi_{2Q}(u-m_{\sigma,G_{2Q}}(u)))(x)>\lambda\}) \lesssim \frac1\lambda\int_{2Q}|u-m_{\sigma,G_{2Q}}(u)|\,d\sigma.$$
		
		For a fixed $Q\in\WW_\lambda$ and for some big constant $K>10$ to be chosen below, 
		let $j_0\geq 0$ be the least integer such that $\ell(2^{j_0}KQ)\geq \rho$. We  split
		\begin{multline*}
			u(x) = \int N(y,x)\, \partial_tS_\rho (\chi_{2^{j_0}KQ} \,g)(y) \,d\sigma(y) 
			+ \int N(y,x)\, \partial_tS_\rho (\chi_{\pom\setminus 2^{j_0}KQ} \,g)(y) \,d\sigma(y)\\
			=: u_l(x) + u_f(x).
		\end{multline*}
		The subindices  ``$l$" and ``$f$" above  stand for ``local" and ``far".
		Then,
		\begin{equation}\label{eqEQ2}
			\sigma(E_Q)  \lesssim \frac1\lambda\int_{2Q}|u_l-m_{\sigma,G_{2Q}}(u_l)|\,d\sigma + \frac1\lambda\int_{2Q}|u_f-m_{\sigma,G_{2Q}}(u_f)|\,d\sigma =:T_l + T_f.
		\end{equation}
		\vv
		
		\subsection*{\bf Estimate of $T_f$}
		We write\footnote{To estimate $T_f$ as we do,  it is important that we are working with 
			the rough Neumann problem.  Similar arguments for the usual Neumann problem do not work.}
		\begin{align}\label{eqalft15} 
			\lambda\,T_f &= \int_{2Q}|u_f(x)-m_{\sigma,G_{2Q}}(u_f)|\,d\sigma(x)\leq \int_{2Q}m_{\sigma,G_{2Q}}(|u_f(x)-u_f|)|\,d\sigma(x)\\
			& \lesssim \sigma(2Q)\,
			\sup_{x,x'\in 2Q}|u_f(x)-u_f(x')|.\notag
		\end{align}

		To bound $|u_f(x)-u_f(x')|$ for $x,x'\in 2Q$, we write
		$$u_f(x) =     \sum_{j\geq j_0} \int_{\pom} N(y,x)\,\partial_t S_\rho (\chi_{2^{j+1}KQ\setminus 2^jKQ}\,g)(y) \,d\sigma(y) =: \sum_{j\geq j_0}u_j(x).$$ Then, 
		$$|u_f(x)-u_f(x')|\leq \sum_{j\geq j_0} |u_j(x)-u_j(x')|.$$
		Denote by $B(2^jKQ)$ a ball with radius $\ell(2^jKQ)$ centered at $x_Q$, the center of $Q$.
		By standard estimates, Kolmogorov's inequality (recall that $p'>2$), and the solvability of $(N^R_{L^{p'},L^{p',\infty}}(\rho))_{\LL^*}$
		with constant $C_*=C_*(\rho)$,
		we have that
		\begin{align*}
			\left(\avint_{B(2^jKQ)\cap\Omega}|u_j|^2\,dm\right)^{1/2} & \lesssim \left(\avint_{2^{j+3}KQ}\cN_\Omega(u_j)^2\,d\sigma\right)^{1/2}\\
			&\lesssim \frac1{\sigma(2^{j+3}KQ)^{1/p'}}\,\|\cN_\Omega(u_j)\|_{L^{p',\infty}(\pom)} \\
			& \lesssim \frac{C_*}{\sigma(2^{j+3}KQ)^{1/p'}}\,\|\chi_{2^{j+1}KQ\setminus 2^jKQ} \,g\|_{L^{p'}(\pom)}\\
			& \lesssim C_*\,\inf_{y\in Q} \cM_{\sigma,p'}g(y),
		\end{align*}
		where $\cM_{\sigma,p'}$ stands for the maximal $p'$-Hardy-Littlewood operator, defined by
		$$\cM_{\sigma,p'} f(x) = \big(\cM_{\sigma} (|f|^{p'})(x)\big)^{1/p'}.$$
		Notice that $u_j$ is  $\LL$-harmonic in $\Omega$ and its conormal derivative vanishes $\sigma$-a.e.\ in $\pom\cap 2^{j-1}KQ$, since
		$$\supp (\partial_t S_\rho (\chi_{2^{j+1}KQ\setminus 2^jKQ}\,g))\subset 2^{j+2}KQ\setminus 2^{j-1}KQ,$$
		by the choice of $j_0$. Then, by Moser type
		estimates (see Theorem \ref{teomoser}) we infer that, for some fixed $\alpha>0$,
		$$|u_j(x) - u_j(x')| \lesssim \left(\frac{\ell(2Q)}{2^{j}K\ell(Q)}\right)^\alpha\,\left(\avint_{B(2^jKQ)\cap\Omega}|u_j|^2\,dm\right)^{1/2}\lesssim K^{-\alpha}\,2^{-j\alpha}\,
		C_*\,\inf_{y\in Q} \cM_{\sigma,p'}g(y).$$
		Thus, summing on $j$,
		$$|u_f(x)-u_f(x')|\leq \sum_{j\geq 1} K^{-\alpha}\,2^{-j\alpha}\,
		C_*\,\inf_{y\in Q} \cM_{\sigma,p'}g(y) \lesssim C_*\,K^{-\alpha}\,
		\inf_{y\in Q} \cM_{\sigma,p'}g(y).$$
		Consequently, by \rf{eqalft15},
		\begin{equation}\label{eqtf94}
			T_f \lesssim \frac1\lambda\,\sigma(Q)\,C_*\,K^{-\alpha}\,
			\inf_{y\in Q} \cM_{\sigma,p'}g(y) \lesssim \frac{C_*\,K^{-\alpha}}\lambda\int_{Q} \cM_{\sigma,p'}g\,d\sigma.
		\end{equation}
		\vv
		
		\subsection*{\bf Estimate of $T_l$ in the case $j_0>0$}
		Remark that in this case we have $\ell(KQ)\leq \rho\approx \ell(2^{j_0}KQ)$.
		Then, from the properties of the kernel of $S_\rho$ in Lemma \ref{lemsr}, it follows that
		$$|\partial_tS_\rho (\chi_{2^{j_0}KQ} \,g)(y)|\lesssim \frac{m_{\sigma,2^{j_0+1}KQ}(|g|)}\rho\,\chi_{2^{j_0+1}KQ}(y).$$ 
		Thus, for all $x\in 2Q$,
		\begin{align*}
			|u_l(x)| & = |N_\sigma^*(\partial_tS_\rho (\chi_{2^{j_0}KQ}  \,g))(x)|\lesssim \frac{m_{\sigma,2^{j_0+1}KQ}(|g|)}\rho\,\int_{y\in2^{j_0+1}KQ}\frac1{|x-y|^{n-1}}\,d\sigma(y)\\
			&\lesssim m_{\sigma,2^{j_0+1}KQ}(|g|)\lesssim
			\inf_{z\in Q} \cM_{\sigma,p'}g(z),
		\end{align*}
		where $N_\sigma^*$ is defined in \rf{defNmu*}.
		Therefore, 
		\begin{equation}\label{eqcas00}
			T_l=\frac1\lambda\int_{2Q}|u_l-m_{\sigma,G_{2Q}}(u_l)|\,d\sigma \lesssim \frac1\lambda\,\sigma(Q)\,\inf_{z\in Q} \cM_{\sigma,p'}g(z)\leq \frac{1}\lambda\int_{Q} \cM_{\sigma,p'}g\,d\sigma.
		\end{equation}

		\vv
		
		\subsection*{\bf Estimate of $T_l$ in the case $j_0=0$}
		In this case we have $\ell(KQ)> \rho$ and
		we write
		\begin{align}\label{eqtl*}
			\lambda\,T_l = \int_{2Q}|u_l-m_{\sigma,G_{2Q}}(u_l)|\,d\sigma &\lesssim  \int_{G_{2Q} }|u_l-m_{\sigma,G_{2Q}}(u_l)|\,d\sigma + 
			\int_{2Q\setminus G_{2Q} }|u_l|\,d\sigma=: I_1+ I_2.
		\end{align}
		To estimate $I_2$,  we use  Kolmogorov's inequality,  Lemma \ref{lemg2q}, and that $(N^R_{L^{p'},L^{p',\infty}}(\rho))_{\LL^*}$ is solvable with constant $C_*$,  and we obtain
		\begin{align}\label{eqI2}
			I_2 & \lesssim \|u_l\|_{L^{p',\infty}(\pom)}\,\sigma(2Q\setminus G_Q)^{1/p} \lesssim 
			K^{n/p}\,\ve^{1/p}\,\|u_l\|_{L^{p',\infty}(\pom)}\,\sigma(Q)^{1/p}\\
			& \lesssim C_*K^{n/p}\,\ve^{1/p}\,\|g\|_{L^{p'}(KQ)}\,\sigma(Q)^{1/p} \lesssim C_*\,C(K)\,\ve^{1/p}
			\,\inf_{y\in Q} \cM_{\sigma,p'}g(y)\,\sigma(Q)\notag \\
			&\lesssim C_*\,C(K)\,\ve^{1/p}\int_{Q} \cM_{\sigma,p'}g\,d\sigma,\notag
		\end{align}
		where $C(K)$ depends on $K$.
		
		To deal with the integral $I_1$, recall that $
		\wt\Omega = U_{x_Q,10 C_5K\ell(Q)}.$
		We denote by $\wt N$ the Neumann function of $\wt \Omega$ and by $\wt\sigma$ the surface measure on $\partial\wt\Omega$.
		We also consider the function $\wt u_l:\wt\Omega\to\R$ defined by
		$$\wt u_l(x) = \int \wt N(y,x)\, \partial_tS_\rho (\chi_{KQ} \,g)(y) \,d\sigma(y),$$
		where the tangential derivative is defined with respect to the tangent at $\pom$.  
		In a sense, $\wt u_l$ should be considered as an approximation of $u_l$. So we split 
		$$I_1 \leq \int_{G_{2Q} }\!|\wt u_l - m_{\sigma,G_{2Q}}(\wt u_l)|\,d\sigma + \int_{G_{2Q} }\!|(u_l-m_{\sigma,G_{2Q}}(u_l))
		-(\wt u_l - m_{\sigma,G_{2Q}}(\wt u_l))|\,d\sigma = I_{1,a} + I_{1,b}.$$
		
		\vv
		
		\subsection*{\bf Estimate of $I_{1,a}$}
		Denote
		$$\vphi=\chi_{G_{2Q}} \,\frac{\wt u_l - m_{\sigma,G_{2Q}}(\wt u_l)}{|\wt u_l - m_{\sigma,G_{2Q}}(\wt u_l)|},$$
		so that we have 
		\begin{align*}
			I_{1,a}= \int_{G_{2Q} }(\wt u_l - m_{\sigma,G_{2Q}}(\wt u_l))\,\vphi\,d\sigma &= \int_{G_{2Q} }(\wt u_l - m_{\sigma,G_{2Q}}(\wt u_l))\,(\vphi - m_{\sigma,G_{2Q}}(\vphi))\,d\sigma \\
			&= \int_{G_{2Q} }\wt u_l\,(\vphi - m_{\sigma,G_{2Q}}(\vphi))\,d\sigma.
		\end{align*}
		In view of the identity above,  Fubini, and the fact that $\sigma=\wt\sigma$ on $G_{2Q}\subset\pom\cap \partial\wt\Omega$, we obtain
		\begin{align*} 
			I_{1,a} &= \int_{x\in G_{2Q}}(\vphi(x) - m_{\sigma,G_{2Q}}(\vphi))\int  \wt N(y,x)\, \partial_t S_\rho(\chi_{KQ}\,g)(y) \,d\sigma(y)\,d\sigma(x)\\
			& =
			-\int S_\rho(\chi_{KQ}\,g)(y)\,\partial_{t_y}\bigg(\int_{x\in G_{2Q}}\wt N(y,x)\,(\vphi(x) - m_{\sigma,G_{2Q}}(\vphi))\,d\wt\sigma(x)\bigg)d\sigma(y).
		\end{align*}
		Denoting 
		$$F(\vphi)(y) = \int_{x\in G_{2Q}}\wt N(y,x)\,(\vphi(x) - m_{\sigma,G_{2Q}}(\vphi))\,d\wt\sigma(x),$$
		and using H\"older's inequality, we get
		\begin{align*}
			I_{1,a} & \leq \|S_\rho (g\,\chi_{KQ})\|_{L^{p'}(\pom)}\,\|\partial_{t}F(\vphi)\|_{L^p(2KQ)}\\ &\leq
			\|g\|_{L^{p'}(KQ)}\,\|\partial_{t}F(\vphi)\|_{L^p(2KQ)}\lesssim \|g\|_{L^{p'}(KQ)}\,\|\partial_{t}F(\vphi)\|_{L^q(2KQ)}\,\sigma(KQ)^{\frac1p-\frac1q}.
		\end{align*}
		{By Lemma \ref{lemaux2} applied to $|\nabla F(\vphi)|^q$ in $\wt\Omega$ and to $\mu=\HH^n|_{\pom\setminus\wpom}$ and the solvability of $(N_q)_\LL$  with constant $\wt C_q$ in $\wt\Omega$, we infer that
			\begin{align*}
				\|\partial_{t}F(\vphi)\|_{L^q(2KQ)} & \lesssim \|\partial_{t}F(\vphi)\|_{L^q(2KQ\setminus\wpom)} + \|\partial_{t}F(\vphi)\|_{L^q(2KQ\cap\wpom)}\\
				& \lesssim \|\cN_{\wt\Omega}(\nabla F(\vphi))\|_{L^q(\partial\wt\Omega)}+\|\partial_{t}F(\vphi)\|_{L^q(2KQ\cap\wpom)}
				\lesssim \wt C_q\,\|\partial_{\wt \nu_A} F(\vphi)\|_{L^q(\partial\wt\Omega)},
		\end{align*}}
		where we denoted by $\partial_{\wt \nu_A}$ the conormal derivative on $\partial\wt\Omega$. 
		Recall now that $\vphi$ is supported on $G_{2Q}\subset\pom\cap \partial\wt\Omega$ and notice that 
		$$
		\avint_{G_{2Q}} (\vphi - m_{\sigma,G_{2Q}}(\vphi))\,d\sigma=0.
		$$
		Since  $\partial_{\wt \nu_A} F(\vphi)=\chi_{G_{2Q}}\,(\vphi - m_{\sigma,G_{2Q}}(\vphi))$ and $\|\vphi\|_\infty\leq 1$,
		we have $\|\partial_{t}F(\vphi)\|_{L^q(2KQ)}\lesssim \wt C_q\,\sigma(Q)^{1/q}$. Thus,
		we obtain
		\begin{align*}
			I_{1,a}&\lesssim \wt C_q\,\|g\|_{L^{p'}(KQ)}\,\sigma(Q)^{\frac1q}\sigma(KQ)^{\frac1p-\frac1q} \lesssim
			\wt C_q\,K^{\frac np-\frac nq}\,\|g\|_{L^{p'}(KQ)}\,\sigma(Q)^{\frac1p}
			\\
			&\lesssim \wt C_q\,C(K)\,\inf_{y\in Q}\cM_{\sigma,p'}(g)(y)\,\sigma(Q) \lesssim \wt C_q\,C(K)\int_Q \cM_{\sigma,p'}(g)\,d\sigma.
		\end{align*}
		
		\vv
		
		\subsection*{\bf Estimate of $I_{1,b}$}
		Let 
		$$\psi= \chi_{G_{2Q}}\,\frac{(u_l-m_{\sigma,G_{2Q}}(u_l))
			-(\wt u_l - m_{\sigma,G_{2Q}}(\wt u_l))}{\big|(u_l-m_{\sigma,G_{2Q}}(u_l))
			-(\wt u_l - m_{\sigma,G_{2Q}}(\wt u_l))\big|}.$$
		Notice that $\psi$ is supported on $G_{2Q}$ and $|\psi(x)|\leq1$ for all $x\in G_{2Q}$.
		Then we have, as in the estimate fo $I_a$,
		\begin{align*}
			I_{1,b} & = \int_{G_{2Q} }\big((u_l-m_{\sigma,G_{2Q}}(u_l))
			-(\wt u_l - m_{\sigma,G_{2Q}}(\wt u_l))\big)\,\psi\,d\sigma\\
			&=\int_{G_{2Q} }\big((u_l-m_{\sigma,G_{2Q}}(u_l))
			-(\wt u_l - m_{\sigma,G_{2Q}}(\wt u_l))\big)\,\big(\psi-m_{\sigma,G_{2Q}}(\psi)\big)\,d\sigma\\
			& =\int_{G_{2Q} }(u_l-\wt u_l)\,\big(\psi-m_{\sigma,G_{2Q}}(\psi)\big)\,d\sigma.
		\end{align*}
		Using Fubini, as above, we get
		\begin{align*}
			I_{1,b} & = \int_{x\in G_{2Q}} \big(\psi(x)-m_{\sigma,G_{2Q}}(\psi)\big)\int ( N(y,x) - \wt N(y,x)) \,\partial_t S_\rho(\chi_{KQ}\,g)(y)\,d\sigma(y)\,d\sigma(x)\\
			& = \int S_\rho(\chi_{KQ}\,g)(y) \,\partial_{t_y} \bigg(\int_{x\in G_{2Q}} (\wt N(y,x) -  N(y,x)) \,\big(\psi(x)-m_{\sigma,G_{2Q}}(\psi)\big)\,d\sigma(x)\bigg)\,d\sigma(y).
		\end{align*}
		
		We denote
		$$v(y) = \int_{x\in G_{2Q}} (\wt N(y,x) -  N(y,x)) \,\big(\psi(x)-m_{\sigma,G_{2Q}}(\psi)\big)\,d\sigma(x)=: \wt h - h,$$
		{ 
			where to shorten notation, we set
			$$h(y) = \int_{x\in G_{2Q}}  N(y,x) \,\big(\psi(x)-m_{\sigma,G_{2Q}}(\psi)\big)\,d\sigma(x)$$
			and
			$$\wt h(y) = \int_{x\in G_{2Q}}  \wt N(y,x) \,\big(\psi(x)-m_{\sigma,G_{2Q}}(\psi)\big)\,d\sigma(x).$$
			Remark that both $\wt h$ and $h$ are solutions of $Lu=0$ in the respective domains $\wt \om$ and $\om$, and recall that $B(x_Q,10 C_5K\ell(Q)) \cap \om \subset \wt \om$.  
			
			To bound $I_{1,b}$, we will apply the localization Lemma \ref{lemlocLp**} and to do so,  we first need to prove that $\partial_{\nu_A} v \in L^{p,1}(4KQ)$.  Estimating $\|\partial_{\nu_A} v\|_{L^{p,1}(4KQ)}$ is one of the key points of this proof.
			To shorten notation, we write 
			$$\psi_0(x) = \chi_{G_{2Q}}\big(\psi(x)-m_{\sigma,G_{2Q}}(\psi)\big), $$
			and
			$$\wt N_\sigma(\psi_0)(y) = \int_{x\in G_{2Q}} \wt N(y,x)\,\psi_0(x)\,d\sigma(x).$$
			Notice that $|\psi_0| \leq 2\, \chi_{G_{2Q}}$. Since $\psi_0$ is supported in $G_{2Q}\subset \wpom$, it vanishes on the relatively open set $\pom\cap\wt\Omega$, and Lemma \ref{lemapp} gives $\partial_{\nu_A}v=\psi_0=0$ there.
			Thus,
			\begin{equation}\label{eqnatural1} 
				\partial_{\nu_A} v(y) = \psi_0(y) - \psi_0(y)=0\quad \mbox{ for $\sigma$-a.e.\ $y\in\pom\cap \partial\wt\Omega$},
			\end{equation}
			since $G_{2Q}\subset \pom\cap  \partial\wt\Omega$.
			Using also the solvability of $(N_{q})_\LL$ in 
			$\wt\Omega$ with constant $\wt C_q$ and Lemma~\ref{lemaux2},
			\begin{equation}\label{eqnatural2}
				\|\partial_{\nu_A} v\|_{L^q(\sigma|_{4KQ})}\lesssim\wt C_q
				\|\psi_0\|_{L^q(\wt\sigma)}.
			\end{equation}
			Although both \rf{eqnatural1} and \rf{eqnatural2} look very natural, they need a careful justification,
			since $\partial_{\nu_A} v$ is only defined in a weak sense. We defer the justification to Appendix \ref{appendix1}.
			
			We choose $\bar q=\frac{p+q}2$, so that $p<\bar q<q$.
			Then, by H\"older's inequality, 
			\begin{align}\label{eqPguai0.5}
				\|\partial_{\nu_A} v\|_{L^{p,1}(4KQ)}&\lesssim \|\partial_{\nu_A} v\|_{L^{\bar q}(4KQ)}\,\sigma(KQ)^{\frac1p-\frac1{\bar q}}\\
				&= \|\partial_{\nu_A} \wt N_\sigma(\psi_0)\|_{L^{\bar q}(4KQ\setminus \partial\wt\Omega)}\,\sigma(KQ)^{\frac1p-\frac1{\bar q}}.\notag
			\end{align}
			Since $4KQ\subset\pom$ and recalling that the conormal derivative $\partial_{\nu_A}$ is defined with respect to the conormal at $\pom$ and that
			\begin{equation*}\label{eqPguai**}
				\sigma(4KQ\setminus \wpom) \leq  C\ve\,\sigma(KQ),
			\end{equation*}
			we may apply  H\"older's inequality to deduce that
			\begin{align*}
				\|\partial_{\nu_A} \wt N_\sigma(\psi_0)\|_{L^{\bar q}(4KQ\setminus \partial\wt\Omega)} &\leq
				\|\partial_{\nu_A} \wt N_\sigma(\psi_0)\|_{L^{q}(4KQ\setminus \partial\wt\Omega)}\,\sigma(4KQ\setminus \partial\wt\Omega)^{\frac1{\bar q}-\frac1q}\\
				&\lesssim \ve^{a_0}\|\partial_{\nu_A} \wt N_\sigma(\psi_0)\|_{L^{q}(4KQ\setminus \partial\wt\Omega)}\,\sigma(KQ)^{\frac1{\bar q}-\frac1q},
			\end{align*}
			with $a_0: = \frac1{\bar q}-\frac1{q}>0.$ Plugging this estimate into \rf{eqPguai0.5} and using \rf{eqnatural2}, we obtain
			\begin{align}\label{eq:Ib-conormal}
				\|\partial_{\nu_A} v\|_{L^{p,1}(4KQ)} & \lesssim \ve^{a_0}\|\partial_{\nu_A} \wt N_\sigma(\psi_0)\|_{L^{q}(4KQ)}\,\sigma(KQ)^{\frac1{p}-\frac1q}\lesssim
				\wt C_q\,\ve^{a_0}\|\psi_0\|_{L^{q}(\partial\wt\Omega)} \,\sigma(KQ)^{\frac1{p}-\frac1q}\\
				&\lesssim \wt C_q\,\ve^{a_0}\,\sigma(KQ)^{\frac1q} \,\sigma(KQ)^{\frac1{p}-\frac1q} = \wt C_q\,\ve^{a_0}\,\sigma(KQ)^{\frac1p},\nonumber
			\end{align}
			demonstrating   that $\partial_{\nu_A} v \in L^{p,1}(4KQ)$.

			We will now show that 
			\begin{align}\label{eq:Ib-interior}
				\bigg(\avint_{A(x_Q,K\ell(Q),4C_5K\ell(Q))\cap\wt\Omega}
				|\nabla v|^2\,dm\bigg)^{1/2}  \!\!\lesssim K^{-n-\alpha}.
			\end{align}
			For the sake of brevity we denote 
			$$
			A_Q:=A(x_Q,K\ell(Q),4C_5K\ell(Q)) \qquad \textup{and}\qquad \wt A_Q:=A(x_Q,0.5K\ell(Q),5C_5K\ell(Q)),
			$$
			where $A(x,r,R)$ stands for the open annulus with center $x$, inner radius $r$, and outer radius $R$.
			We use the triangle inequality  and estimate the corresponding terms separately.   To this end,  by the zero mean of $\psi(x)-m_{\sigma,G_{2Q}}(\psi)$ on $G_{2Q}$ and Theorem~\ref{teoneumann1}, we infer that, for all $y\in \wt A_Q \cap \Omega$,
			\begin{align*}
				|h(y)| & =\bigg|\int_{x\in G_{2Q}}  (N(y,x) - N(y,x_Q)) \,\big(\psi(x)-m_{\sigma,G_{2Q}}(\psi)\big)\,d\sigma(x)\bigg|\\
				&\lesssim \int_{x\in G_{2Q}}  \frac{|x-x_Q|^\alpha}{|x-y|^{n-1+\alpha}} \,\big|\psi(x)-m_{\sigma,G_{2Q}}(\psi)\big|\,d\sigma(x)\\
				&\lesssim \frac{\ell(Q)^\alpha}{\ell(KQ)^{n-1+\alpha}}\,\|\psi\|_{L^1(G_{2Q})}\lesssim
				\frac{K^{-\alpha}}{\ell(KQ)^{n-1}}\,\sigma(Q).
			\end{align*}
			Then, since $\partial_{\nu_A} h$ vanishes on $\wt A_Q \cap \pom$,  by Caccioppoli's inequality,  we get
			\begin{equation*}
				\bigg(\avint_{ A_Q\cap\Omega}
				|\nabla h|^2\,dm\bigg)^{1/2} \lesssim  \bigg(\frac1{\ell(KQ)}\avint_{ \wt A_Q \cap\Omega}
				|h|^2\,dm\bigg)^{1/2} \lesssim\frac{K^{-\alpha}}{\ell(KQ)^{n}}\,\sigma(Q) \approx K^{-n-\alpha}.
			\end{equation*}
			Arguing analogously,   taking into account now that the conormal derivative of $\wt h$ vanishes on $ \wt A_Q \cap \wpom$, we also obtain 
			\begin{align*}
				\bigg(\avint_{  A_Q \cap \Omega} |\nabla \wt h|^2\,dm\bigg)^{1/2}  \leq \bigg(\avint_{  A_Q \cap\wt\Omega} |\nabla \wt h|^2\,dm\bigg)^{1/2}  \!\!\lesssim K^{-n-\alpha},
			\end{align*}
			which, combined with the same  estimate for $h$,  readily proves \eqref{eq:Ib-interior}.

			Finally, using the solvability of $(N_{L^{p,1},L^{p}}(\rho))_\LL$ in $\om$ with constant  comparable to $C_*$ and the localization Lemma \ref{lemlocLp**}, we get
			\begin{align}\label{eqal627}
				&I_{1,b} = \int S_\rho(\chi_{KQ}\, g)(y) \,\partial_{t_y} v(y)\,d\sigma(y) \leq \|g\,\chi_{KQ}\|_{L^{p'}(\pom)} 
				\|S_\rho(\partial_{t}v)\|_{L^{p}(1.1 KQ)} \\
				& \lesssim C_*\,\|g\,\chi_{KQ}\|_{L^{p'}(\pom)} \bigg(\|\partial_{\nu_A} v\|_{L^{p,1}(4KQ)} + \bigg(\avint_{A(x_Q,K\ell(Q),4C_5K\ell(Q))\cap\Omega}
				\!|\nabla v|^2\,dm\bigg)^{1/2}\!\!\sigma(KQ)^{1/p}\bigg),\nonumber
			\end{align}
			with $C_5$ as in Lemma \ref{lemlocLp**}.  Plugging \eqref{eq:Ib-conormal} and \eqref{eq:Ib-interior} into \eqref{eqal627}, we conclude that  
			$$
			I_{1,b} \lesssim C_*\,(K^{-\alpha} + C(K)\,\wt C_q\,\ve^{a_0})\int_{Q} \cM_{\sigma,p'}g\,d\sigma.
			$$
			
		}
		
		\vv
		\subsection*{\bf Final estimates for $T_l$}
		Gathering the estimates obtained for $I_{1,a}$ and $I_{1,b}$, we get
		$$
		I_{1}\leq I_{1,a} + I_{1,b} \lesssim \wt C_q\,C(K)\int_Q \cM_{\sigma,p'}g\,d\sigma + C_*
		\,(K^{-\alpha} + C(K)\,\ve^{a_0})\int_{Q} \cM_{\sigma,p'}g\,d\sigma.
		$$
		Combining this with \rf{eqtl*} and \rf{eqI2}, we derive
		\begin{align*}
			\lambda\,T_l \lesssim \lambda\,(I_1+I_2) & \lesssim
			\wt C_q\,C(K)\int_Q \cM_{\sigma,p'}g\,d\sigma + C_*\,(K^{-\alpha} + C(K)\,\wt C_q\,\ve^{a_0})\int_{Q} \cM_{\sigma,p'}g\,d\sigma
			\\
			&\quad+ C_*\,C(K)\,\ve^{1/p}\int_{Q} \cM_{\sigma,p'}g\,d\sigma\\
			& \lesssim \wt C_q\,C(K)\int_Q \cM_{\sigma,p'}g\,d\sigma + C_*\,(K^{-\alpha} + C(K)\,\wt C_q\,\ve^{a_1})\int_{Q} \cM_{\sigma,p'}g\,d\sigma,
		\end{align*}
		where $a_1=\min(a_0,\frac1{p})$.
		
		\vv
		\subsection*{\bf End of the proof}
		The above bound  for $T_l$ in conjunction with the one in \rf{eqcas00} and the one for $T_f$ in \rf{eqtf94},  implies
		$$\lambda( T_l + T_f) \lesssim
		\wt C_q\,C(K)\int_Q \cM_{\sigma,p'}g\,d\sigma + C_*\,C(K)\,\wt C_q\,\ve^{a_1}\!\int_{Q} \cM_{\sigma,p'}g\,d\sigma
		+C_*\,K^{-\alpha}\!\int_{Q} \cM_{\sigma,p'}g\,d\sigma.
		$$
		Recalling \rf{eqEQ2}, we conclude that
		$$\sigma(E_Q) \lesssim 
		\frac1\lambda\,\big( 
		\wt C_q\,C(K)+ C_*\,C(K)\,\wt C_q\,\ve^{a_1}
		+C_*\,K^{-\alpha}\big) \int_{Q} \cM_{\sigma,p'}g\,d\sigma.
		$$
		Notice that, if we denote
		$$\gamma(K,\ve) := C(K)\,\wt C_q\,\ve^{a_1}+K^{-\alpha},$$
		then $\gamma(K,\ve)$ can be taken arbitrarily small, first choosing $K$ large enough and then $\ve$ small enough.
		In this way, we have
		\begin{align*}
			\sigma(V_{A\lambda}) = \sum_{Q\in\WW_\lambda}
			\sigma(E_Q) & \lesssim 
			\frac1\lambda\,\big( 
			\wt C_q\,C(K)  + C_*\,\gamma(K,\ve)\big) \sum_{Q\in\WW_\lambda}\int_{Q} \cM_{\sigma,p'}g\,d\sigma\\
			& = \frac1\lambda\,\big( 
			\wt C_q\,C(K) + C_*\,\gamma(K,\ve)\big) \int_{V_\lambda} \cM_{\sigma,p'}g\,d\sigma.
		\end{align*}
		By  Kolmogorov's inequality and the weak $(p',p')$ boundedness of $\cM_{\sigma,p'}$, we obtain 
		$$ \frac1\lambda\int_{V_\lambda} \cM_{\sigma,p'}g\,d\sigma \lesssim \frac1\lambda\,\sigma(V_\lambda)^{1/p}\,\|\cM_{\sigma,p'}g\|_{L^{p',\infty}(\pom)}
		\lesssim \frac1\lambda\,\sigma(V_\lambda)^{1/p}\,\|g\|_{L^{p'}(\pom)}.$$
		Using the inequality $a^{1/p}\,b^{1/p'}\leq \frac ap + \frac b{p'}$ for $a,b>0$, we can write,
		for arbitrary constants $\kappa\in (0,1)$ and   $d>0$, 
		$$\frac d\lambda \int_{V_\lambda} \cM_{\sigma,p'}g\,d\sigma 
		\leq \kappa\,\sigma(V_\lambda) + C(\kappa)\,\frac{d^{p'}}{\lambda^{p'}}\,\|g\|_{L^{p'}(\pom)}^{p'},$$
		with $C(\kappa)$ depending on $p,p'$ besides $\kappa$.
		Therefore, choosing $d=\wt C_q\,C(K) + C_*\,\gamma(K,\ve)$ (and changing $\kappa$ by $C\kappa$ if necessary), we derive
		$$
		\sigma(V_{A\lambda})  \leq  \kappa\,\sigma(V_\lambda) + 
		C(\kappa)\,\frac{\wt C_q^{p'}\,C(K)+  C_*^{p'}\,\gamma(K,\ve)^{p'}}{\lambda^{p'}}\,\|g\|_{L^{p'}(\pom)}^{p'}.$$
		Multiplying by $(A\lambda)^{p'}$, we obtain
		$$(A\lambda)^{p'}\,
		\sigma(V_{A\lambda})  \leq  \kappa\,A^{p'}\lambda^{p'}\,\sigma(V_\lambda) + 
		C(\kappa)\,A^{p'}\big(\wt C_q^{p'}\,C(K)+  C_*^{p'}\,\gamma(K,\ve)^{p'}\big)\,\|g\|_{L^{p'}(\pom)}^{p'},$$
		and,  if we choose $\kappa=1/(2A^{p'})$ and take supremum in $\lambda$,   we deduce
		\begin{align*}
			\sup_{\lambda>0} (\lambda^{p'}\,\sigma(V_\lambda) )=\|\cM_\sigma(u) \|_{L^{p',\infty}(\pom)}^{p'} &\leq \tfrac{1}{2} \|\cM_\sigma(u)\|_{L^{p',\infty}(\pom)} ^{p'}\\ 
			& +  C(A)\big(\wt C_q^{p'}\,C(K)+  C_*^{p'}\,\gamma(K,\ve)^{p'}\big)\,\|g\|_{L^{p'}(\pom)}^{p'}.
		\end{align*}
		Moving the first term on the right hand side of the inequality above to the left,  this implies that 
		$$
		\|u\|_{L^{p',\infty}(\pom)}^{p'} \leq C(A)\big(\wt C_q^{p'}\,C(K)+  C_*^{p'}\,\gamma(K,\ve)^{p'}\big)\,\|g\|_{L^{p'}(\pom)}^{p'}.
		$$
		Since this holds for any Lipschitz function $g\in L^{p'}(\pom)$ and for all tangent fields $t=t_{j,k}$, we get
		$$C_* ^{p'} \approx C(N^R_{L^{p'},L^{p',\infty}}(\rho))^{p'} \lesssim  C(A)\big(\wt C_q^{p'}\,C(K)+  C_*^{p'}\,\gamma(K,\ve)^{p'}\big).$$
		As remarked above, $\gamma(K,\ve)$ can be taken arbitrarily small for suitable choices of $K$ and $\ve$,  concluding  that
		$$C_* \equiv C_*(\rho) \lesssim C(p,A,K,\ve)\,\wt C_q,$$
		for suitable choices of $A,K,\ve$, uniformly with respect to $\rho\in (0,\diam(\pom))$. This completes the proof of the theorem. 
	\end{proof}

	\vv
	
	\begin{rem}
		If in the preceding proof we choose $U_{\xi,r}=\Omega$ (and so $\wt\Omega=\Omega$) for all $\xi\in\pom$, $r>0$,  we obtain a new proof of the fact that $(N_q)_\LL$ solvability implies $(N_p)_\LL$ solvability for $1<p<q$. Notice that in this situation, the term $I_{1,b}$ above vanishes and so the arguments become simpler.
		This new proof does not require to use the solvability of the Neumann problem in any suitable Hardy space, unlike the classical approach.
	\end{rem}
	\vv
	
	
	\appendix
	
	\section{Proof of \rf{eqnatural1} and \rf{eqnatural2}}\label{appendix1}
	
	It suffices to prove the following result.
	
	\begin{lemma}\label{lemapp}
		Let $\Omega\subset\R^{n+1}$ be a chord-arc domain. Let $B$ be an open ball centered in $\pom$ and let $\wt\Omega\subset\R^{n+1}$
		be another chord-arc domain that contains $3\overline B\cap \Omega$ and such that $G=\pom\cap \partial\wt\Omega\cap \overline B\neq \varnothing$.
		Let $A$ be an $(n+1)\times(n+1)$ matrix with Dini mean oscillation, let $\LL={\rm div}(A\nabla\cdot)$ and suppose that $(N_q)_{\LL}$ is solvable in $\wt\Omega$ with constant $\wt C_q$, for some $q\in (1,\infty)$.
		Let $\psi_0\in L^q(\sigma)$ be supported on $G$ with $\int \psi_0\,d\sigma=0$ and let $w= \wt N_\sigma(\psi_0)$. Then
		\begin{equation}\label{eqnatural1*} 
			\partial_{\nu_A} w(y) = \psi_0(y)\quad \mbox{ for $\sigma$-a.e.\ $y\in G$},
		\end{equation}
		and
		\begin{equation}\label{eqnatural3*}
			\|\partial_{\nu_A} w\|_{L^q(\sigma|_{2B})}\lesssim \wt C_q\,
			\|\psi_0\|_{L^q(\wt\sigma)}.
		\end{equation}
	\end{lemma}
	\vv
	
	In the lemma, we denote by $\sigma$ the surface measure on $\pom$, by $\nu$ the outer unit normal for $\Omega$, by
	$\partial_{\nu_A} w$ the conormal derivative of $w$ in $\pom\cap 3B$ (in the weak sense),
	and by $\wt N$ the Neumann function for $\wt\Omega$.
	
	Notice that $w$ is not defined in the whole $\Omega$, in general, since $\wt \Omega$ may not contain $\Omega$. 
	So, under the assumptions of the lemma, we can only define $\partial_{\nu_A} w$ locally in $\pom\cap 3B$. We say that
	$\partial_{\nu_A}w=g$ locally in $\pom\cap 3B$ in the weak sense if for any Lipschitz function supported on $3B$,
	$$\int A\nabla w\cdot\nabla \vphi\,dm = \int_\pom g\,\vphi\,d\sigma.$$
	The identity \rf{eqnatural1*}  should be understood in this sense.

	
	\begin{proof}[Proof of Lemma \ref{lemapp}]
		\noi {\bf Existence of $\partial_{\nu_A} w$ and proof of \rf{eqnatural3*}.}
		First we check that the conormal derivative $\partial_{\nu_A} w$ exists locally in $\pom\cap 2B$ in the weak sense and it belongs to $L^q(\sigma|_{2B})$. 
		Let $\vphi:\pom\to\R$ be a Lipschitz function supported on $2B\cap\pom$,  extend it to a Lipschitz function in $\Omega$, which we still denote by $\vphi$, so that it vanishes in $\Omega\setminus 3B$,  and define
		$$T_w(\vphi) = \int_\Omega A\nabla w \cdot\nabla\vphi\,dm.$$
		Using the fact that $\LL w =0$ in $3B$, it easily follows  that the definition does not depend on the precise extension of $\vphi$.
		
		To check that $T_w$  extends to a bounded functional in $L^{q'}(\sigma|_{2B})$, consider a partition
		of $\Omega$ into a family $\WW(\Omega)$ of dyadic Whitney cubes as in Section \ref{secwhitney}, and let $\Omega_k$ be the interior of the closure of the cubes from $\WW(\Omega)$ with side length at least $2^{-k}$.
		From the fact that $\cN_{\wt\Omega}(\nabla w)\in L^q(\wt \sigma)$ (where $\wt \sigma$ is the surface measure on $\partial\wt\Omega$),
		it follows that $\nabla w\in L^q(\wt\Omega)$, and thus by dominated convergence,
		$$T_w(\vphi) = \lim_{k\to\infty}\int_{\Omega_k} A\nabla w \cdot\nabla\vphi\,dm.$$
		Since $\nabla w$ is uniformly continuous in $\Omega\cap 3B$  with modulus of continuity $\theta(t):=\int_0^t \omega_A(r)\,\frac{dr}r$ (by \cite{DK}), we can use the divergence theorem in \cite[Theorem 2.8]{HMT} applied to $\Omega_k$ and $\vphi A\nabla w$
		to derive
		$$\int_{\Omega_k} {\rm div}(\vphi A\nabla w)\,dm =
		\int_{\partial\Omega_k} \vphi A\nabla w\cdot\nu\,d\HH^n.$$
		Indeed, the assumptions in that theorem hold, since in the sense of distributions ${\rm div}(\vphi A\nabla w) = \nabla\vphi A\nabla w$ is continuous in a neighborhood of $\overline{\Omega_k}$ (because
		$A$ is continuous\footnote{In fact, in any compact set, $A$ agrees a.e.\ with a uniformly continuous function.}, and $\nabla\vphi$ and $\nabla w|_{3B}$ are continuous), and also $\cN_{\Omega_k}(\vphi A\nabla w)\in L^p(\HH^n|_{\pom_k})$ and the pointwise non-tangential trace of $\vphi A\nabla w|_{\pom_k}$ exists, using again the continuity of $A$ and of $\nabla w$ in a neighborhood of $\overline{\Omega_k}$.  
		Then we have
		\begin{align*}
			|T_w(\vphi)| &\leq \limsup_{k\to\infty}\bigg|\int_{\Omega_k} {\rm div}(\vphi A\nabla w)\,dm\bigg| =
			\limsup_{k\to\infty}\bigg|\int_{\partial\Omega_k} \vphi A\nabla w\cdot\nu_{\om_k}\,d\HH^n\bigg|\\
			& \lesssim \limsup_{k\to\infty}\int_{\partial\Omega_k\cap 2B} |\vphi\,\nabla w|\,d\HH^n.
		\end{align*}
		Now we intend to apply Lemma \ref{lemaux2}
		with  $\mu_k=\HH^n|_{\pom_k\cap\wt\Omega}$ and $\wt \Omega$ in place of $\Omega$. Remark that the maximal operator  $\cM_{\mu_k}$ is bounded from $L^{q'}(\mu_k)$ to $L^{q'}(\wt\sigma)$. Indeed,  for any $f\in L^{q'}(\mu_k)$ and $x\in\R^{n+1}$, by the Ahlfors regularity of $\pom_k$ (implied by the Ahlfors regularity of $\pom$), we have
		$$\cM_{\mu_k}(f)(x)\lesssim \sup_{r>0} \frac1{r^n}\,\int_{B(x,r)}|f|\,d\mu_k=: \cM_n(f)(x).$$
		Using the Ahlfors regularity (or just the polynomial growth) of $\wt\sigma$, it follows easily that $\cM_n$ is bounded from $L^1(\mu_k)$ to
		$L^{1,\infty}(\wt\sigma)$. Also, it is immediate that $\cM_{\mu_k}$ is bounded from $L^\infty(\mu_k)$ to $L^\infty(\wt\sigma)$. By interpolation then, it is bounded in from $L^{q'}(\mu_k)$ to $L^{q'}(\wt\sigma)$.
		
		Using also the  solvability of $(N_q)_\LL$ in $\wt\Omega$ and \rf{eqgradpunt}, we can now use Lemma \ref{lemaux2} to get
		\begin{align*}
			|T_w(\vphi)|&\lesssim \limsup_{k\to\infty}\int_{\partial\wt\Omega} \cN_{\wt\Omega}(\nabla w)\,\cM_{\mu_k}(\vphi)\,d\wt\sigma\\
			&\lesssim \wt C_q\,\|\partial_{\wt\nu_A}w\|_{L^q(\wt\sigma)}\,\limsup_{k\to\infty}\|\cM_{\mu_k}(\vphi)\|_{L^{q'}(\wt\sigma)}\lesssim\wt C_q\,
			\|\psi_0\|_{L^q(\wt\sigma)}\,\limsup_{k\to\infty}\|\vphi\|_{L^{q'}(\mu_k)}.
		\end{align*}
		Using that $\vphi$ is a Lipschitz function in $\overline{\Omega}$, it follows easily that
		$\limsup_{k\to\infty}\|\vphi\|_{L^{q'}(\mu_k)}\lesssim \|\vphi\|_{L^{q'}(\sigma)}.$
		Consequently, $$|T_w(\vphi)|\lesssim \wt C_q\,\|\partial_{\wt\nu_A}w\|_{L^q(\wt\sigma)}\,\|\vphi\|_{L^{q'}(\sigma)},$$ and
		by the Hahn-Banach theorem, $T_w$ extends to a bounded functional in $L^{q'}(\sigma|_{2B})$, as wished. Then, by the Riesz representation theorem, there exists some function 
		$g\in L^{q}(\sigma|_{2B})$ such that
		$$\int_\Omega A\nabla w \cdot\nabla\vphi\,dm = T_w(\vphi) = \int_{\pom} g\,\vphi\,d\sigma$$
		for any Lipschitz function $\vphi:\pom\to\R$ supported on $2B\cap\pom$, whose Lipschitz extension is supported in $3B$, and moreover
		$$\|g\|_{L^{q}(\sigma|_{2B})} = \|T_w\|_{L^{q}(\sigma|_{2B})\to\R} \lesssim \wt C_q\,\|\psi_0\|_{L^q(\wt\sigma)},$$
		which proves \rf{eqnatural3*}.
		
		\vv
		\noi {\bf Proof of \rf{eqnatural1*}.} 
		The $n$-rectifiability and $n$-Ahlfors regularity of $\pom\cup\wpom$ also
		implies the existence of tangents $\HH^n$-a.e.\ in $\pom\cup\wpom$. That is, for any $\HH^n$-a.e.\  $\xi\in\pom\cup\wpom$ there exists a hyperplane $L_\xi$ through $\xi$ such that following holds. For any $\theta\in (0,1)$, there exists $r(\xi,\theta)>0$ such that
		\begin{equation}\label{eqtan5}
			(\pom\cup\wpom)\cap B(x,r(\xi,\theta)) \setminus X(\xi,L_\xi,\theta)=\varnothing,
		\end{equation}
		where $X(\xi,L_\xi,\theta)$ stands for the closed cone
		$$X(\xi,L_\xi,\theta)= \{x\in\R^{n+1}:\dist(y,L_\xi) \leq \theta\,|x-\xi|\}.$$
		The existence of tangents for $\HH^n$-a.e.\  $\xi\in\pom\cup\wpom$ follows from the fact that, being $\pom\cup\wpom$ $n$-rectifiable,
		there exists an approximate tangent for $\HH^n$-a.e.\  $\xi\in\pom\cup\wpom$ (see \cite[Chapter 15]{Mattila}, for example), and then the 
		Ahlfors regularity of $\pom\cup\wpom$ implies that any approximate tangent is also a tangent in the sense above.
		
		For a small $\theta>0$ to be fixed below, let $G_k =\{x\in G: r(x,\theta)\geq 1/k\}$, so that
		$$G = \bigcup_{k\geq1} G_k\cup Z,$$
		with $\HH^n(Z)=0$.
		To prove \rf{eqnatural1*}, it suffices
		to show that $\partial_{\nu_A} w(y) = \psi_0(y)$ for $\sigma$-a.e.\ $y\in  G_k$. To this end, we will prove that for any
		compact set $F\subset G_k$,
		\begin{equation}\label{eqigF0}
			\int_F \partial_{\nu_A} w\,d\sigma = \int_F \psi_0\,d\sigma. 
		\end{equation}
		
		For a compact set $F\subset G_k$, let $\vphi_{F,\delta}:\partial\wt\Omega:\to\R$ be a Lipschitz function such that
		$$\chi_F\leq \vphi_{F,\delta}\leq \chi_{U_\delta(F)\cap\wpom},$$ where  $U_\delta(F):=\{x \in \R^{n+1}: \dist(x,F) < \delta\}$ is the $\delta$-neighborhood of $F$. Notice that we have that $F\subset \overline{B}$ because $G_k\subset G\subset \overline{B}$.
		We denote by $\CC_{\wt\Omega}$ the following Carleson operator, acting over functions or vector fields
		$F:\wt\Omega\to\R$:
		$$\CC_{\wt\Omega}(F)(\xi) = \sup_{r>0}\frac1{r^n}\int_{\wt\Omega\cap B(\xi,r)}|F(y)|\,dy,\qquad \xi\in\wpom.$$

		\begin{claim}\label{claimff}
			For each $\delta>0$, there exists an extension $\wt \vphi_{F,\delta}:\overline{\wt\Omega}\to\R$ of $\vphi_{F,\delta}$ which is Lipschitz in 
			$\overline{\wt\Omega}$, such that
			\begin{itemize}
				\item[(a)] $0\leq \wt\vphi_{F,\delta}\leq 1$ and $\wt \vphi_{F,\delta}$ is supported on $\overline{\wt\Omega}\cap 3B$, 
				
				\item[(b)] $\wt\vphi_{F,\delta}\,\chi_\pom \to\chi_F$ in $L^{q'}(\sigma)$ as $\delta\to0$, and
				
				\item[(c)] $\|\CC_{\wt\Omega}(\chi_{\wt\Omega\setminus\Omega}\nabla \wt \vphi_{F,\delta})\|_{L^{q'}(\wt\sigma)}\to 0$ as $\delta\to0$.
			\end{itemize}
		\end{claim}
		
		Assume this claim for the moment and let us see how \rf{eqigF0} follows. By the properties of the Neumann function $\wt N$ and the  definition of $\partial_{\nu_A} w$, we have
		\begin{align}\label{eqal8fv}
			\int_\wpom \psi_0\, \vphi_{F,\delta}\,d\wt\sigma & = \int_{\wt\Omega} A \nabla w\,\nabla\wt\vphi_{F,\delta}\,dm
			= \int_{\Omega} A \nabla w\,\nabla\wt\vphi_{F,\delta}\,dm + \int_{\wt\Omega\setminus\Omega} A \nabla w\,\nabla\wt\vphi_{F,\delta}\,dm\\
			&= \int_{\pom} \partial_{\nu_A} w\,\,\wt\vphi_{F,\delta}\,d\sigma + \int_{\wt\Omega\setminus\Omega} A\nabla w\,\nabla\wt\vphi_{F,\delta}\,dm.
			\notag
		\end{align}
		By the definition of $\vphi_{F,\delta}$ and by (b) and the fact that $\|\partial_{\nu_A} w\|_{L^q(\sigma|_{2B})}<\infty$, we have, as $\delta\to0$, 
		$$\int_\wpom \psi_0\, \vphi_{F,\delta}\,d\wt\sigma\to \int_F \psi_0\, d\wt\sigma\quad \text{ and }\quad
		\int_{\pom} \partial_{\nu_A} w\,\wt\vphi_{F,\delta}\,d\sigma \to \int_F \partial_{\nu_A} w\,d\sigma.$$
		
		To bound the last integral on the right hand side of \rf{eqal8fv} we use the duality result of Hyt\"onen and Ros\'en as proved in \cite[Proposition 2.4]{MPT} (see \cite{HR} for the case when $\Omega=\R^{n+1}_+$) and the solvability of $(N_q)_\LL$. Indeed, 
		\begin{align*}
			\int_{\wt\Omega\setminus\Omega} |A\nabla w\,\nabla\wt\vphi_{F,\delta}|\,dm &\lesssim
			\|\cN_{\wt\Omega}(\nabla w)\|_{L^q(\wt\sigma)}\,\|\CC_{\wt\Omega}(\chi_{\wt\Omega\setminus\Omega}\nabla\wt\vphi_{F,\delta})\|_{L^{q'}(\wt\sigma)}\\
			&\lesssim \|\psi_0\|_{L^q(\wt\sigma)}\,\|\CC_{\wt\Omega}(\chi_{\wt\Omega\setminus\Omega}\nabla\wt\vphi_{F,\delta})\|_{L^{q'}(\wt\sigma)}.
		\end{align*}
		By (c) in the claim above, $\|\CC_{\wt\Omega}(\chi_{\wt\Omega\setminus\Omega}\nabla\wt\vphi_{F,\delta})\|_{L^{q'}(\wt\sigma)}\to0$
		as $\delta \to 0$, and so 
		$$\int_{\wt\Omega\setminus\Omega} |\nabla w\,\nabla\wt\vphi_{F,\delta}|\,dm \to 0 
		$$
		as $\delta\to0$.
		So letting $\delta\to0$ in the identity \rf{eqal8fv}, we obtain \rf{eqigF0}, which proves the lemma, modulo Claim \ref{claimff}.
	\end{proof}
	
	\vv
	In order to prove Claim \ref{claimff}, we will show first the following.
	
	\begin{lemma}\label{lemdyex9}
		Under the assumptions of Lemma \ref{lemapp}, let $F\subset  G_k$ be a compact set as in Claim \ref{claimff}. Then there exists a function $\psi_F:\wt\Omega\to\R$ satisfying the
		following:
		\begin{itemize}
			\item[(a)] $\psi_F\in {\rm Lip}_{loc}(\wt\Omega)$ and $0\leq\psi_F\leq1$.
			\item[(b)] $\supp\psi_F\subset\overline{\wt\Omega}\cap\overline{\Omega} \cap 3B$.
			\item[(c)] $\psi_F\to\chi_F$ non-tangentially $\wt\sigma$-a.e.\ in $\wt\Omega$.
		\end{itemize}
	\end{lemma}
	
	Recall that $G_k\subset G =\pom\cap \wpom\cap \overline B$, so that $F\subset \overline B$ too.
	
	\begin{proof}
		Consider the Whitney decomposition of $\wt\Omega$ described in Section \ref{secwhitney} and let $\WW(\wt\Omega)$ be the family of Whitney cubes.
		Consider a partition of unity in $\wt\Omega$ with a family of functions $\{\eta_Q\}_{Q\in\WW(\wt\Omega)}$ such that
		$$\chi_{\wt\Omega}=\sum_{Q\in\WW(\wt\Omega)}\eta_Q,$$
		with  $\eta_Q\in C^\infty$, supported 
		on $2Q$, such that $\|\nabla\eta_Q\|_\infty\lesssim 1/\ell(Q)$, for every $Q\in\WW(\wt\Omega)$. For a small $\tau\ll r(B)$ to be chosen below,
		let $\WW_\tau(\wt\Omega)$ be the family of the cubes $Q\in\WW(\wt\Omega)$ such that $\ell(Q)\leq \tau$ and $\dist(Q,B) \leq \tau$.
		Then we set 
		\begin{equation}\label{eqdefpsif}
			\psi_F= \sum_{Q\in \WW_\tau(\wt\Omega)} m_{b(Q),\wt\sigma}(\chi_F)\,\eta_Q,
		\end{equation}
		where $b(Q)$ is the boundary cube in $\wpom$ associated with $Q$ (see Subsection \ref{secwhitney}).
		It is clear that $\psi_F$ satisfies (a). Taking into account that $F\subset B$ and the definition of $\psi_F$, (c) follows easily using
		the Lebesgue differentiation theorem. See \cite[Lemma 3.1]{MZ} for more details, for example.
		
		So it remains to prove the property (b). From the choice of $\tau\ll r(B)$ and the fact that $F\subset B$, it follows that $\supp\psi_F\subset 2B$. So it suffices to
		show that $\supp\psi_F\subset\overline{\wt\Omega}\cap\overline{\Omega}$ if $\tau$ is taken small enough. First observe that since the 
		cubes $Q\in\WW_\tau(\wt\Omega)$ have side length at most $\tau$, we have that $m_{b(Q),\wt\sigma}(\chi_F)=0$ unless $\dist(Q,F)\lesssim\tau$. 
		This implies that for some fixed constant $C_9$ (depending on $n$ and the Ahlfors regularity of $\wpom$),
		$$\psi_F(x) =0 \quad \mbox{ for all $x\in\wt\Omega$ such that $\dist(x,F)\geq C_9\tau.$}$$
		
		Let $x_0\in\wt\Omega$ be such that $\psi_F(x_0)\neq0$. Our objective is to prove that $x_0\in \Omega$. Clearly, this implies
		that $\supp(\psi_F)\subset \overline \Omega$.
		Let $\xi\in F$ be such that $|x_0-\xi|=\dist(x_0,F)$.
		By the discussion above, $r:=|x_0-\xi|\leq C_9\tau$. 
		We claim that 
		\begin{equation}\label{eqclaim834}
			(\wt\Omega\setminus\Omega)\cap B(\xi,3r)\subset U_{C\theta r}(L_\xi),
		\end{equation}
		for $\tau$ small enough, where $L_\xi$ is the tangent hyperplane defined in \rf{eqtan5}. Indeed, the condition \rf{eqtan5} ensures that 
		$$(\pom\cup \wpom) \cap B(\xi,3r)\subset U_{C\theta r}(L_\xi)$$
		if $3r\leq 1/k$, by the definition of $G_k$. We assume $3C_9\tau\leq 1/k$, so that the preceding condition holds.
		Let $B_\xi^+$, $B_\xi^-$ be the two components of $B(\xi,3r)\setminus U_{C\theta r}(L_\xi)$. By a connectivity argument, any component $B_\xi^\pm$
		is either contained in $\Omega$ or in $\R^{n+1}\setminus \overline{\Omega}$. By the interior and exterior corkscrew conditions for $\Omega$, if $\theta$ is assumed to be small enough, one component, say $B_\xi^+$, is contained in $\Omega$ and the other, $B_\xi^-$, in $\R^{n+1}\setminus \overline{\Omega}$. The same argument applies to $\wt\Omega$: one of the components
		$B_\xi^\pm$ is contained in $\wt\Omega$ and the other in $\R^{n+1}\setminus \overline{\wt\Omega}$. The fact that $\Omega\cap B(\xi,r)\subset\wt\Omega\cap B(\xi,r)$ implies that $B_\xi^+$ is both contained in $\Omega$ and in $\wt\Omega$, while $B_\xi^-$ is contained both in $\R^{n+1}\setminus \overline{\Omega}$ and in $\R^{n+1}\setminus \overline{\wt\Omega}$.
		It is immediate to check that this implies \rf{eqclaim834}.
		
		Aiming for a contradiction, suppose that $x_0\in\wt\Omega\setminus\Omega$. By \rf{eqclaim834} and the discussion above, $x_0\in U_{C\theta r}(L_\xi)$. Since $B_\xi^-\subset \R^{n+1}\setminus \overline{\wt\Omega}$, we deduce that $$\dist(x_0,\wpom) = \dist(x_0,\R^{n+1}\setminus \overline{\wt\Omega})\leq C\theta r.$$
		For $\theta$ small enough, this implies that if $x_0\in 2Q$ for some $Q\in \WW_\tau(\wt\Omega)$, then 
		$$b(Q)\subset B\big(x_0,\tfrac12r\big)\cap \wpom \subset \wpom\setminus F,$$
		using that $\xi$ is the closest point from $F$ to $x_0$ and that $r=|x_0-\xi|$ for the last inclusion.
		Consequently, $m_{b(Q),\wt\sigma}(\chi_F)=0$, and so $\psi_F(x_0)=0$ by the definition of $\psi_F$ in \rf{eqdefpsif}.
		This contradicts the assumption that $\psi_F(x_0)\neq0$, and thus $x_0\not\in\wt\Omega\setminus\Omega$.
	\end{proof}
	\vv
	
	To prove Claim \ref{claimff} we will need the following result about Varopoulos type extensions from \cite{MZ} (see \cite{HR} for the case when $\Omega:=\R^{n+1}_+$).
	
	\begin{theorem}{\cite[Theorem 1.3]{MZ}}\label{teoextMZ}
		Let $\wt \Omega\subset\R^{n+1}$ be a chord-arc domain. For any $f\in L^p(\wt\sigma)$ with $p\in (1,\infty]$, there exists a function
		$u:\wt\Omega\to\R$ such that:
		\begin{itemize}
			\item[(a)] $u\in C^\infty(\Omega)$.
			\item[(b)] $\|\cN_{\wt\Omega}(u)\|_{L^p(\wt\sigma)} + \|\cN_{\wt\Omega}(\delta_{\wt\Omega}\nabla u)\|_{L^p(\wt\sigma)}\lesssim \|f\|_{L^p(\wt\sigma)}$.
			\item[(c)] $\|\CC_{\wt\Omega}(\nabla u)\|_{L^p(\wt\sigma)}\lesssim \|f\|_{L^p(\wt\sigma)}$.
			\item[(d)] For $\wt\sigma$-a.e.\ $\xi\in\wpom$,
			$${\rm nt}\text{-}\!\lim_{x\to\xi} \avint_{B(x,\delta_{\wt\Omega}(x)/2)}u\,dm = f(\xi).$$
		\end{itemize}
	\end{theorem}
	
	Above, ${\rm nt}\text{-}\!\lim$ stands for non-tangential limit. We remark that in \cite{MZ} the authors prove other similar results for domains more general than chord-arc domains. 
	
	\begin{rem}
		The property (d) in Theorem \ref{teoextMZ} can be strengthened: for $\wt\sigma$-a.e.\ $\xi\in\wpom$, it holds
		\begin{equation}\label{eqconv719}
			{\rm nt}\text{-}\!\lim_{x\to\xi} \avint_{B(x,\delta_{\wt\Omega}(x)/2)}|u - f(\xi)|\,dm = 0.
		\end{equation}
		To check this, we apply Poincar\'e's
		inequality to get
		\begin{align*}
			\avint_{B(x,\delta_{\wt\Omega}(x)/2)} & |u - f(\xi)|\,dm \\
			&\leq \avint_{B(x,\delta_{\wt\Omega}(x)/2)}|u - m_{B(x,\delta_{\wt\Omega}(x)/2)}(u)|\,dm + |m_{B(x,\delta_{\wt\Omega}(x)/2)}(u) - f(\xi)|\\
			& \lesssim \delta_{\wt\Omega}(x) \avint_{B(x,\delta_{\wt\Omega}(x)/2)}|\nabla u|\,dm + |m_{B(x,\delta_{\wt\Omega}(x)/2)}(u) - f(\xi)|.
		\end{align*}
		The second summand on the right hand side converges to $0$ non-tangentially by (d) in Theorem~\ref{teoextMZ}. To show that the first one also converges to $0$ non-tangentially $\wt\sigma$-a.e., we use the fact that the area functional defined by
		$$A_{\wt\Omega}^{(\beta)}(\nabla u)(\xi) = \int_{\gamma_\beta(\xi)}
		|\nabla u|\,\frac{dm}{\delta_{\wt\Omega}(x)^n}$$
		satisfies
		$$\|A_{\wt\Omega}^{(\beta)}(\nabla u)\|_{L^p(\wt\sigma)} \approx |\CC_{\wt\Omega}(\nabla u)\|_{L^p(\wt\sigma)},$$
		by \cite[Lemma 2.5]{MZ}.
		So $A_{\wt\Omega}^{(\beta)}(\nabla u)(\xi)<\infty$ for $\wt\sigma$-a.e.\ $\xi\in\wpom$, by the property (c) in Theorem \ref{teoextMZ}. Thus, for such points $\xi$ and for all $x\in\gamma_\alpha(\xi)$,
		with suitable $\beta$ depending on $\alpha$, we have
		$$ \delta_{\wt\Omega}(x) \avint_{B(x,\delta_{\wt\Omega}(x)/2)}|\nabla u|\,dm \lesssim \int_{\gamma_\beta(\xi)\cap B(\xi,C|x-\xi|)}
		|\nabla u|\,\frac{dm}{\delta_{\wt\Omega}(x)^n}\to 0$$
		as $|x-\xi|\to0$.
	\end{rem}

	\vv
	\begin{proof}[\bf Proof of Claim \ref{claimff}]
		Consider the function $\psi_F$ in Lemma \ref{lemdyex9}, and let $u_{F,\delta}$ the Varopoulos type extension for the function $f=
		\vphi_{F,\delta}-\chi_F$ given by Theorem \ref{teoextMZ}. We let 
		$$v_{F,\delta} = \psi_F + u_{F,\delta}.$$
		Observe that $v_{F,\delta}$ is locally Lipschitz in $\Omega$, it extends $\vphi_{F,\delta}$ to $\wt\Omega$, in the sense that 
		\begin{equation}\label{eqnont341}
			{\rm nt}\text{-}\!\lim_{x\to\xi} \avint_{B(x,\delta_{\wt\Omega}(x)/2)}|v_{F,\delta} -\vphi_{F,\delta}(\xi)|\,dm = 0
		\end{equation}
		for $\wt\sigma$-a.e.\ $\xi\in\wpom$ by Lemma \ref{lemdyex9} (c) and \rf{eqconv719}. Further, from the fact that
		$\supp\psi_F\subset\overline{\wt\Omega} \cap \overline{\Omega}\cap 3B$ we infer that $\chi_{\wt\Omega\setminus\Omega}\nabla \psi_F=0$ 
		a.e.\
		with respect to Lebesgue measure, and then by Theorem \ref{teoextMZ} (c), we get
		\begin{align}\label{eqprim57}
			\|\CC_{\wt\Omega}(\chi_{\wt\Omega\setminus\Omega}\nabla v_{F,\delta})\|_{L^{q'}(\wt\sigma)} &\leq \|\CC_{\wt\Omega}(\chi_{\wt\Omega\setminus\Omega}\nabla \psi_F)\|_{L^{q'}(\wt\sigma)}
			+ \|\CC_{\wt\Omega}(\nabla u_{F,\delta})\|_{L^{q'}(\wt\sigma)} \\
			& \lesssim 0 + \|\vphi_{F,\delta}-\chi_F\|_{L^{q'}(\wt\sigma)}
			\leq \wt\sigma(U_\delta(F)\setminus F)^{1/{q'}}\to 0\nonumber
		\end{align}
		as $\delta\to0$. By replacing $v_{F,\delta}$ by $\max(0,\min(1,v_{F,\delta}))$ if necessary, we can ensure that $0\leq v_{F,\delta}\leq 1$ and the
		preceding properties still hold.

		The function $v_{F,\delta}$ may not extend continuously to $\overline{\wt\Omega}$, and so we cannot choose $\wt\vphi_{F,\delta}$
		to be equal to $v_{F,\delta}$. So, inspired by the proof of \cite[Theorem 1.4]{MZ}, we will modify it as follows. Consider the family of Whitney cubes $\WW(\wt\Omega)$ in the
		proof of Lemma \ref{lemdyex9} and the associated partition of $\chi_{\wt\Omega}$ given by the functions $\{\eta_Q\}_{Q\in\WW(\wt\Omega)}$.
		Consider the regularized dyadic extension of $\vphi_{F,\delta}$ defined by
		$$h_{F,\delta} = \sum_{Q\in \WW(\wt\Omega)} m_{b(Q),\wt\sigma}(\vphi_{F,\delta})\,\eta_Q.$$
		Using that $\vphi_{F,\delta}$ is Lipschitz in $\wpom$, it follows easily that $h_{F,\delta}$ is a Lipschitz extension of $\vphi_{F,\delta}$ to $\overline{\wt\Omega}$.
		
		For a small $\ve\ll r(B)$,
		let $\WW^\ve(\wt\Omega)$ be the family of the cubes $Q\in\WW(\wt\Omega)$ such that $\ell(Q)\leq \ve$ and let $\chi_\ve$ be the function
		$$\chi_\ve = \sum_{Q\in \WW^\ve(\wt\Omega)} \eta_Q.$$
		Notice that $\chi_\ve$ is Lipschitz in $\wt\Omega$,  it equals $1$ in  $\wt\Omega\cap U_{c\ve}(\wpom)$,  for some $c>0$, and 
		$\supp(\nabla\chi_\ve)$ is contained in a family of cubes from $\WW(\wt\Omega)$ with side length comparable to $\ve$. For some $\ve\in (0,\delta)$ to be chosen below, we consider the function in $\wt\Omega$ defined by
		$$\vphi_{\delta,\ve} = \chi_\ve \,h_{F,\delta} + (1-\chi_\ve) \,v_{F,\delta}.$$
		We also set $\vphi_{\delta,\ve}=\vphi_{F,\delta}$ on $\wpom$.
		
		We will check now that we may take $\wt\vphi_{F,\delta} = \vphi_{\delta,\ve}$ in $\wt \Omega$ for some $\ve$ small enough. By construction, 
		$\vphi_{\delta,\ve}$ coincides with $h_{F,\delta}$ in a $(c\,\ve)$-neighborhood of $\wpom$ and taking into account that $v_{F,\delta}\in
		{\rm Lip}_{loc}(\wt\Omega)$, it follows that $\vphi_{\delta,\ve}$ is Lipschitz in $\Omega$. 
		Also, it is easily checked that $0\leq \vphi_{\delta,\ve}\leq 1$ and $\vphi_{\delta,\ve}$ is supported on $\overline{\wt\Omega}\cap 3B$. So
		it just remains to prove that $\vphi_{\delta,\ve}$ satisfies the properties (b) and (c) of Claim \ref{claimff}.
		First we will show (b), that is,
		$\vphi_{\delta,\ve}\,\chi_\pom \to\chi_F$ in $L^{q'}(\sigma)$ as $\delta\to0$. To this end, we write
		\begin{align*}
			\|\vphi_{\delta,\ve} - \chi_F\|_{L^{q'}(\sigma)} &= \|\vphi_{\delta,\ve}\|_{L^{q'}(\sigma|_{\pom\setminus F})}
			\leq 
			\|\chi_\ve \,h_{F,\delta}\|_{L^{q'}(\sigma|_{\pom\setminus F})} +\|v_{F,\delta}\|_{L^{q'}(\sigma|_{\pom\setminus \wpom})}.
		\end{align*}
		Observe now that $\chi_\ve h_{F,\delta}$ is supported in a $(C\ve)$-neighborhood of $\supp(\vphi_{F,\delta})$, and so in a 
		$(C'\delta)$-neighborhood of $F$, since $\ve\leq\delta$. Hence,
		$$\|\chi_\ve \,h_{F,\delta}\|_{L^{q'}(\sigma|_{\pom\setminus F})}\leq \sigma(U_{C'\delta}(F)\setminus F)^{1/q'}\to 0$$
		as $\delta\to0$. On the other hand, since $\psi_F$ vanishes in $\pom\setminus \wpom$,
		$$\|v_{F,\delta}\|_{L^{q'}(\sigma|_{\pom\setminus \wpom})} = \|\psi_F + u_{F,\delta}\|_{L^{q'}(\sigma|_{\pom\setminus \wpom})} = 
		\|u_{F,\delta}\|_{L^{q'}(\sigma|_{\pom\setminus \wpom})}.
		$$
		To estimate the last term on the right hand side we apply Lemma \ref{lemaux2} and Theorem \ref{teoextMZ} (b) (recall that
		$u_{F,\delta}$ is a Varopoulos extension of $\vphi_{F,\delta}-\chi_F$):
		$$\|u_{F,\delta}\|_{L^{q'}(\sigma|_{\pom\setminus \wpom})}\lesssim \|\cN_{\wt\Omega}(u_{F,\delta})\|_{L^{q'}(\wt\sigma)}
		\lesssim \|\vphi_{F,\delta}-\chi_F\|_{L^{q'}(\wt\sigma)}\leq \wt\sigma(U_\delta(F)\setminus F)^{1/q'}\to 0$$
		as $\delta\to0$. This completes of the proof of the fact that $\|\vphi_{\delta,\ve} - \chi_F\|_{L^{q'}(\sigma)}\to0$.
		
		\vv
		Next we turn our attention to the property (c) in Claim \ref{claimff}.
		We split
		$$\nabla\vphi_{\delta,\ve} = \nabla\chi_\ve (h_{F,\delta} - v_{F,\delta}) + \chi_\ve \,\nabla h_{F,\delta}
		+ (1-\chi_\ve) \,\nabla v_{F,\delta},$$
		so that
		\begin{align}\label{eqt1t2t3}
			\|\CC_{\wt\Omega}(\chi_{\wt\Omega\setminus\Omega}\nabla \wt \vphi_{\delta,\ve})\|_{L^{q'}(\wt\sigma)}
			& \leq
			\|\CC_{\wt\Omega}(\nabla\chi_\ve (h_{F,\delta} - v_{F,\delta}))\|_{L^{q'}(\wt\sigma)}
			+
			\|\CC_{\wt\Omega}(\chi_\ve \,\nabla h_{F,\delta})\|_{L^{q'}(\wt\sigma)}\\
			&\quad +
			\|\CC_{\wt\Omega}(\chi_{\wt\Omega\setminus\Omega}\nabla v_{F,\delta})\|_{L^{q'}(\wt\sigma)} \nonumber\\
			& =: T_1 + T_2+ T_3.\nonumber
		\end{align}
		We have already shown that $T_3\to0$ as $\delta\to 0$ in \rf{eqprim57}. To estimate $T_1$, we write, for $\xi\in \wpom$, 
		\begin{align*}
			\CC_{\wt\Omega}(\nabla\chi_\ve (h_{F,\delta} - v_{F,\delta}))(\xi) & =
			\sup_{r>0} \frac1{r^n}\int_{\wt\Omega\cap B(\xi,r)} |\nabla\chi_\ve (h_{F,\delta} - v_{F,\delta})|\,dm \\
			& \lesssim \sup_{r>0} \frac1{\ve\,r^n}\int_{\wt\Omega\cap B(\xi,r)\cap \supp(\nabla\chi_\ve)} |h_{F,\delta} - v_{F,\delta}|\,dm.
		\end{align*}
		Denote by $\WW^{(\ve)}(\wt\Omega)$ the family 
		family of cubes from $\WW(\wt\Omega)$ with side length comparable to $\ve$ such that 
		$\supp(\nabla\chi_\ve)$ is contained in the union of cubes from this family. 
		Also, for a function $u:\wt\Omega\to\R$, $\zeta\in\wpom$, and $t>0$, let
		$$\wh{\cN}_{\wt\Omega,t} u(\zeta) = \sup_{x\in\gamma_{\wt\Omega}(\zeta)\cap B(\zeta,t)}\avint_{B(x,\delta_\Omega(x)/2)}|u|\,dm,$$
		for a cone $\gamma_{\wt\Omega}(\xi)$ with big enough aperture.
		Then, for $0<r\leq \diam(\wt\Omega)$, we have
		\begin{align*}
			\frac1{\ve\,r^n}\int_{\wt\Omega\cap B(\xi,r)\cap \supp(\nabla\chi_\ve)} |h_{F,\delta} - v_{F,\delta}|\,dm
			& \leq \frac1{\ve\,r^n} \sum_{Q\in \WW^{(\ve)}(\wt\Omega)} \int_{Q \cap B(\xi,r)} |h_{F,\delta} - v_{F,\delta}|\,dm\\
			&\lesssim \frac1{r^n} \sum_{P\in \DD_{\mu,Q,\ve}} \inf_{\zeta \in P} \wh{\cN}_{\wt\Omega,C\ve}(|h_{F,\delta} - v_{F,\delta}|)(\zeta)\,\ell(P)^{n}\\
			&\lesssim \frac1{r^n}\int_{B(\xi,Cr)\cap\wpom}\wh{\cN}_{\wt\Omega,C\ve}(|h_{F,\delta} - v_{F,\delta}|)\,d\wt\sigma,
		\end{align*}
		where we denoted by $\DD_{\mu,Q,\ve}$ the family of cubes of the form $P=b(Q)$ for some $Q\in \WW^{(\ve)}(\wt\Omega)$ with $Q\cap B(\xi,r)\neq\varnothing$ (notice that the latter condition implies that $r\gtrsim\ve$). So we get
		$$\CC_{\wt\Omega}(\nabla\chi_\ve (h_{F,\delta} - v_{F,\delta}))(\xi) \lesssim \cM_{\wt\sigma}(\wh{\cN}_{\wt\Omega,C\ve}(|h_{F,\delta} - v_{F,\delta}|))(\xi)\quad\mbox{ for all $\xi\in\pom$},$$
		and therefore
		$$\|\CC_{\wt\Omega}(\nabla\chi_\ve (h_{F,\delta} - v_{F,\delta}))\|_{L^{q'}(\wt\sigma)}
		\lesssim 
		\|\wh{\cN}_{\wt\Omega,C\ve}(|h_{F,\delta} - v_{F,\delta}|)\|_{L^{q'}(\wt\sigma)}.
		$$
		Notice now that $ \|\wh{\cN}_{\wt\Omega,C\ve}(|h_{F,\delta} - v_{F,\delta}|)\|_{L^{q'}(\wt\sigma)}<\infty$, because
		$h_{F,\delta}$, $v_{F,\delta}$, and thus $\wh{\cN}_{\wt\Omega,C\ve}(|h_{F,\delta} - v_{F,\delta}|)$, are uniformly bounded.
		Also, for any fixed $\delta>0$ we have
		$$\wh{\cN}_{\wt\Omega,C\ve}(|h_{F,\delta} - v_{F,\delta}|)(\xi) \leq 
		\wh{\cN}_{\wt\Omega,C\ve}(|h_{F,\delta} - \vphi_{F,\delta}(\xi)|)(\xi) + 
		\wh{\cN}_{\wt\Omega,C\ve}(|v_{F,\delta} - \vphi_{F,\delta}(\xi)|)(\xi)\to 0 $$
		as $\ve\to0$ for $\wt\sigma$-a.e.\ $\xi\in\wpom$, because of \rf{eqnont341} (which also holds with 
		$h_{F,\delta}$ in place of $v_{F,\delta}$).
		Thus,  by dominated convergence, $\|\wh{\cN}_{\wt\Omega,C\ve}(|h_{F,\delta} - v_{F,\delta}|)\|_{L^{q'}(\wt\sigma)}\to0$ as $\ve\to0$, and so we can pick $\ve$ small enough (depending on $\delta$) such that
		$$\|\CC_{\wt\Omega}(\nabla\chi_\ve (h_{F,\delta} - v_{F,\delta}))\|_{L^{q'}(\wt\sigma)}\leq\delta.$$
		\vv
		
		It only remains to show that the term $T_2=\|\CC_{\wt\Omega}(\chi_\ve \nabla h_{F,\delta})\|_{L^{q'}(\wt\sigma)}$ in \rf{eqt1t2t3} goes to $0$ as $\delta\to0$. Since $h_{F,\delta}$ is Lipschitz (with constant depending on $\delta$), for any $\xi\in\wpom$ and any $r>0$, we have
		\begin{align*}
			\CC_{\wt\Omega}(\chi_\ve \nabla h_{F,\delta})(\xi) & =\sup_{r>0}
			\frac1{r^n} \int_{\wt\Omega\cap B(\xi,r)}|\chi_\ve \,\nabla h_{F,\delta}|\,dm \\
			& \leq C(\delta)\,\sup_{r>0}\frac{m(U_{C\ve}(\wpom)\cap B(\xi,r))}{r^n}
			\lesssim C(\delta)\,\ve.
		\end{align*}
		Therefore, $T_2\leq C(\delta)\,\ve\,\wt\sigma(\wpom)^{1/q'}$ and thus we can choose $\ve$ small enough so that $T_2\leq\delta$.
		Altogether, we deduce that
		$$\|\CC_{\wt\Omega}(\chi_{\wt\Omega\setminus\Omega}\nabla \wt \vphi_{\delta,\ve})\|_{L^{q'}(\wt\sigma)}
		\leq T_1 + T_2+ T_3\to0$$
		as $\delta\to0$, which completes the proof of Claim \ref{claimff}.
	\end{proof}

	\vv
	

\end{document}